\documentclass[reqno, 8pt]{amsart}

\usepackage{amsmath,amssymb,amsthm,amsfonts,mathrsfs}
\usepackage{fullpage}
\usepackage{xcolor}
\usepackage{graphicx}
\usepackage{makecell}
\usepackage{listings}
\usepackage{bbm}

\definecolor{codegreen}{rgb}{0,0.6,0}
\definecolor{codegray}{rgb}{0.5,0.5,0.5}
\definecolor{codepurple}{rgb}{0.58,0,0.82}
\definecolor{backcolour}{rgb}{0.95,0.95,0.92}

\lstdefinestyle{mystyle}{
  backgroundcolor=\color{backcolour},   commentstyle=\color{codegreen},
  keywordstyle=\color{magenta},
  numberstyle=\tiny\color{codegray},
  stringstyle=\color{codepurple},
  basicstyle=\ttfamily\footnotesize,
  breakatwhitespace=false,         
  breaklines=true,                 
  captionpos=b,                    
  keepspaces=true,                 
  numbers=left,                    
  numbersep=5pt,                  
  showspaces=false,                
  showstringspaces=false,
  showtabs=false,                  
  tabsize=2
}

\usepackage{listings}
\usepackage[colorlinks=true,
linkcolor=blue,
anchorcolor=blue,
citecolor=red
]{hyperref}
\allowdisplaybreaks 
\newtheorem{theorem1number}{Theorem}
\newtheorem{theorem}{Theorem}[section]
\newtheorem{lemma}[theorem]{Lemma}

\newtheorem*{conjecture*}{Conjecture}
\newtheorem*{theorem*}{Theorem}
\theoremstyle{definition}

\theoremstyle{remark}

\newtheorem*{remark*}{remark}

\newcommand{\on}{\operatorname}

\renewcommand{\geq}{\geqslant}
\renewcommand{\mod}[1]{\on{mod}#1}

\usepackage{colonequals}

\author{\textbf{Runbo Li}}
\address{International Curriculum Center, The High School Affiliated to Renmin University of China, Beijing, China}
\email{runbo.li.carey@gmail.com}

\makeatletter
\@namedef{subjclassname@2020}{\textup{}2020 Mathematics Subject Classification}
\makeatother

\title[]{An average Brun--Titchmarsh theorem and shifted primes with a large prime factor}

\begin{document}

\begin{abstract}
The author studies an average version of Brun--Titchmarsh theorem with large moduli. Using Maynard's recent breakthrough on the Bombieri--Friedlander--Iwaniec type triple convolution estimates, we refine the previous result of Baker and Harman (1996). As an application, we improve a result of Baker and Harman (1998) on shifted primes with a large prime factor, showing that the largest prime factor of $p - 1$ is larger than $p^{0.679}$ for infinitely many primes $p$.
\end{abstract}

\maketitle

\tableofcontents

\section{Introduction}
One of the most important problems in number theory is the twin prime conjecture, which states that there are infinitely many primes $p$ such that $p+2$ is also prime. This conjecture still remains open, and there are three main ways to attack this conjecture. The first way is to study $p+2$ with few primary factors, and the best result is due to Chen \cite{Chen1973}, who proved in 1973 that there are infinitely many primes $p$ such that $p+2$ has at most two prime factors. One can also see \cite{Cai2008}, \cite{Wu2008}, \cite{Lichtman}, \cite{Lichtman3}, \cite{Pascadi2} and the author's preprints \cite{LRB1733} \cite{LRB197} for further refinements. The second way is to study bounded gaps between consecutive primes, and we refer the interested readers to the works of Zhang \cite{ZhangYitang}, Maynard \cite{Maynard600}, Polymath8b \cite{Polymath8b} and Stadlmann \cite{Stadlmann525}.

The third way to attack the twin prime conjecture is to consider the largest prime factor of $p+2$. Let $P^{+}(n)$ denote the largest prime factor of $n$. Since we have $P^{+}(p) = p$, the twin prime conjecture is equivalent to the following statement: There are infinitely many $p$ such that
\begin{equation}
P^{+}(p + 2) > P^{+}(p) = p.
\end{equation}
Or, in general, there are infinitely many $p$ such that
\begin{equation}
P^{+}(p + a) > p
\end{equation}
for a fixed non-zero integer $a$. We don't know how to prove (2), but we can prove some weaker results along this way. Using the Bombieri--Vinogradov Theorem proved in 1965, Goldfeld \cite{Goldfeld1969} showed in 1969 that there are infinitely many $p$ such that
\begin{equation}
P^{+}(p + a) > p^{\boldsymbol{\Theta} - \varepsilon}
\end{equation}
for some $\boldsymbol{\Theta} > \frac{1}{2}$. After Goldfeld, many mathematicians improved the lower bound for $\boldsymbol{\Theta}$. Let $\theta \in \left[\frac{1}{2}, \boldsymbol{\Theta} \right]$. In order to deal with $\theta$ larger than $\frac{1}{2}$, one needs to study an average version of Brun--Titchmarsh inequality, and prove strong upper estimates
\begin{equation}
\pi\left(x; q, a\right) \leqslant \frac{C(\theta) x}{\varphi(q) \log x}
\end{equation}
with an upper constant $C(\theta)$ for almost all moduli $q \leqslant x^{\theta}$ (with at most $O\left(x^{\theta} (\log x)^{-A} \right)$ exceptions). The progression of records for $\boldsymbol{\Theta}$ can be seen in the following table.

\begin{center}
\begin{tabular}{|c|c|c|c|c|}
\hline \boldmath{$\boldsymbol{\Theta}$} & \boldmath{$\approx$} & \textbf{Author} & \textbf{Year} & \textbf{New "almost all" Brun--Titchmarsh result} \\
\hline $1 - \frac{1}{2}e^{-\frac{1}{4}}$ & $0.6105$ & Motohashi \cite{Motohashi1970} & 1970 & $\pi\left(x; q, a\right) \leqslant \frac{\left(\frac{2}{1 - \theta} +\varepsilon \right) x}{\varphi(q) \log x},\ q < x,\ q \sim x^{\theta} \ (\text{for all } q)$ \\
\hline $2 - \frac{3}{2}e^{-\frac{1}{12}}$ & $0.6199$ & Hooley \cite{HooleyBT1} & 1972 & \makecell{ [\cite{HooleyBT1}, Theorem 1] \\ $\pi\left(x; q, a\right) \leqslant \frac{\left(\frac{6}{2 - \theta} +\varepsilon \right) x}{\varphi(q) \log x},\ x^{\frac{1}{2}} < q < x^{\frac{4}{5} },\ q \sim x^{\theta}$ } \\
\hline $\frac{5}{8}$ & $0.625$ & Hooley \cite{Hooleypa} & 1973 & \makecell{ [\cite{DI1984}, (2)] \\ $\pi\left(x; q, a\right) \leqslant \frac{(4+\varepsilon) x}{\varphi(q) \log x},\ x^{\frac{1}{2}} < q < x^{\frac{3}{4} }$ } \\
\hline $\frac{5}{3} - \frac{17^{\frac{5}{7}}}{2^{\frac{4}{7}} 3^{\frac{5}{7}} 5^{\frac{3}{7}}}e^{-\frac{1}{8}}$ & $0.6381$ & Iwaniec \cite{IwaniecBT} & 1982 & \makecell{ [\cite{IwaniecBT}, Theorems 6 and 10] \\ $\pi\left(x; q, a\right) \leqslant \begin{cases}
\frac{\left(\frac{8}{6-7 \theta} +\varepsilon \right) x}{\varphi(q) \log x}, & x^{\frac{1}{2}} < q < x^{\frac{8}{15}},\ q \sim x^{\theta} \ (\text{for all } q) \\
\frac{\left(\frac{12}{5-3 \theta} +\varepsilon \right) x}{\varphi(q) \log x}, & x^{\frac{8}{15}} < q < x^{\frac{2}{3}},\ q \sim x^{\theta}
\end{cases}$ } \\
\hline $1 - \frac{1}{2}e^{-\frac{3}{8}}$ & $0.6563$ & Deshouillers and Iwaniec \cite{DI1984} & 1984 & \makecell{ [\cite{DI1984}, Theorem] \\ $\pi\left(x; q, a\right) \leqslant \frac{\left(\frac{4}{3(1-\theta)}+\varepsilon \right) x}{\varphi(q) \log x},\ x^{\frac{1}{2}} < q < x,\ q \sim x^{\theta}$ } \\
\hline \makecell{ $\frac{1}{20} \exp \left(\frac{-3}{280} \times \right.$ \\ $\left(35-21 \log \left(\frac{13}{12}\right) \right.$ \\ $-60 \log \left(\frac{400}{351}\right)$ \\ $\left. \left. +35 \log \left(\frac{361}{360}\right)\right)\right) \times$ \\ $\left(20 \exp \left(\frac{3}{280} \times \right. \right.$ \\ $\left(35-21 \log \left(\frac{13}{12}\right) \right. $ \\ $-60 \log \left(\frac{400}{351}\right)$ \\ $\left. \left. \left. +35 \log \left(\frac{361}{360}\right)\right)\right)-9\right)$} & $0.6578$ & Fouvry \cite{Fouvry1984} & 1984 & \makecell{ [\cite{Fouvry1984}, Théorème] \\ $\pi\left(x; q, a\right) \leqslant \begin{cases}
\frac{\left(\frac{12}{25-40 \theta} - \log \left(\frac{10(1-\theta)}{9 \theta} \right)+\varepsilon \right) x}{\varphi(q) \log x}, & x^{\frac{1}{2}} < q < x^{\frac{53}{104}},\ q \sim x^{\theta} \\
\frac{\left(\frac{48}{47-56 \theta} - \log \left(\frac{10(1-\theta)}{9 \theta} \right)+\varepsilon \right) x}{\varphi(q) \log x}, & x^{\frac{53}{104}} < q < x^{\frac{10}{19}},\ q \sim x^{\theta} \\
\frac{\left(\frac{48}{47-56 \theta} +\varepsilon \right) x}{\varphi(q) \log x}, & x^{\frac{10}{19}} < q < x^{\frac{11}{20}},\ q \sim x^{\theta}
\end{cases}$ } \\
\hline $0.6683$ & $0.6683$ & Fouvry \cite{Fouvry} & 1985 & \makecell{ [\cite{Fouvry}, Théorème 3] \\ $\pi\left(x; q, a\right) \leqslant \begin{cases}
\frac{(1 + \text{Sieve Integrals} + \varepsilon) x}{\varphi(q) \log x}, & x^{\frac{1}{2}} < q < x^{\frac{17}{32}},\ q \sim x^{\theta} \\
\frac{\left(\frac{22}{15-13 \theta} - \log \left(\frac{4(1-\theta)}{3 \theta} \right)+\varepsilon \right) x}{\varphi(q) \log x}, & x^{\frac{17}{32}} < q < x^{\frac{8}{15}},\ q \sim x^{\theta} \\
\frac{\left(\frac{10}{9-10 \theta} - \log \left(\frac{4(1-\theta)}{3 \theta} \right)+\varepsilon \right) x}{\varphi(q) \log x}, & x^{\frac{8}{15}} < q < x^{\frac{4}{7}},\ q \sim x^{\theta} \\
\frac{\left(\frac{10}{9-10 \theta} +\varepsilon \right) x}{\varphi(q) \log x}, & x^{\frac{4}{7}} < q < x^{\frac{3}{5}},\ q \sim x^{\theta} \\
\frac{\left(\frac{8}{3- \theta} +\varepsilon \right) x}{\varphi(q) \log x}, & x^{\frac{3}{5}} < q < x^{\frac{5}{7}},\ q \sim x^{\theta}
\end{cases}$ } \\
\hline
\end{tabular}
\end{center}

In 1996, Baker and Harman \cite{676} presented a mixed approach to this problem. They combined the traditional average Brun--Titchmarsh approach (with a refined version of [\cite{Fouvry}, Théorème 3] which will be used in Section 6.10) with the alternative sieve developed by Harman \cite{Harman1}, using the estimates of Bombieri--Friedlander--Iwaniec type (see \cite{BFI} and \cite{BFI2}, for example) as arithmetic information inputs. Their method is more powerful when $\theta$ is near to $\frac{1}{2}$. In this way they got $\boldsymbol{\Theta} = 0.676$. Two years later, they \cite{677} further developed a variant of the three-dimensional alternative sieve used in \cite{676} and obtained a saving over one-dimensional sums, leading to a better result $\boldsymbol{\Theta} = 0.677$. Although this result only improved the previous one by $0.001$, it was still a improvement since they used a continuous version of alternative sieve, and any saving on a single point became quite small after integrating. Because of the lack of new arithmetic information, even $0.678$ seems impossible without any new things.

Recently, Maynard \cite{MaynardLargeModuliI} got a new breakthrough on the distribution of primes in arithmetic progressions to large moduli $q = q_1 q_2 \sim x^{\frac{11}{21} -\varepsilon}$. The main theorem in his paper is not applicable here, but we find that one of his lemmas (Proposition 8.3) can be used here as a new arithmetic information input. By inserting this into the sieving process in \cite{676} and \cite{677}, together with a careful numerical calculation for sieve integrals, we get the following sharper result.
\begin{theorem1number}\label{t1}
There are infinitely many primes $p$ such that
$$
P^{+}(p + a) > p^{0.679 - \varepsilon}.
$$
\end{theorem1number}
In order to prove our Theorem~\ref{t1}, we need to give good upper bounds for $C(\theta)$ in (4). Clearly this yields an improvement over the main (numerical) theorem for $C_2(\theta)$ in \cite{676}. All the numerical values of integrals in Section 6 are calculated using Mathematica 14.

Throughout this paper, we always suppose that $\varepsilon$ is a sufficiently small positive constant and $x$ is sufficiently large. Let $\theta \in [0.5, 0.679]$ and $Q = x^{\theta}$. The letter $p$, with or without subscript, is reserved for prime numbers.

\section{An outline of the proof}
Let
$$
\mathcal{A}^q =\{n : n \sim x,\ n \equiv a (\bmod q) \}, \quad \mathcal{B}^q =\{n : n \sim x,\ (n, q) = 1 \},
$$
$$
\mathcal{C}_d=\{n: n d \in \mathcal{C}\}, \quad P(z)=\prod_{p<z} p, \quad S(\mathcal{C}, z)=\sum_{\substack{n \in \mathcal{C} \\ (n, P(z))=1}} 1.
$$

To prove our Theorem~\ref{t1}, it suffices to show that
\begin{equation}
\sum_{x^{0.679} \leqslant p \leqslant x + a} \pi(x; p, -a) \log p > \varepsilon x.
\end{equation}
Note that we have
\begin{align}
\nonumber \sum_{p \leqslant x} \log (p + a) =&\ \sum_{p \leqslant x} \sum_{n | (p + a)} \Lambda(n) \\
\nonumber =&\ \sum_{p_1 \leqslant x} \sum_{p_2 | (p_1 + a)} \log p_2 + o(x) \\
=&\ \sum_{p \leqslant x + a} \pi(x; p, -a) \log p + o(x)
\end{align}
and
\begin{equation}
\sum_{p \leqslant x} \log (p + a) \sim x.
\end{equation}

By the Prime Number Theorem, we have
\begin{equation}
\sum_{p \leqslant x + a} \pi(x; p, -a) \log p > (1 - \varepsilon) x.
\end{equation}
From the Bombieri--Vinogradov theorem and the Brun--Titchmarsh inequality, we have
\begin{equation}
\sum_{p \leqslant x^{\frac{1}{2}} } \pi(x; p, -a) \log p < \left(\frac{1}{2} + \varepsilon \right) x.
\end{equation}
Thus, in order to show (6), we only need to establish that
\begin{equation}
\sum_{x^{\frac{1}{2}} < p \leqslant x^{0.679} } \pi(x; p, -a) \log p < \left(\frac{1}{2} - 3 \varepsilon \right) x.
\end{equation}
We remark that one can also see \cite{LRB7435} and previous articles on that topic to get the idea of using a continuous version of Harman's alternative sieve to give an upper bound for a special sum over primes in a fixed range. By the definition of the sieve function and Prime Number Theorem, we have
\begin{equation}
\sum_{p \in \mathcal{A}^{q}} 1 = S\left(\mathcal{A}^{q}, (2x)^{\frac{1}{2}} \right) \quad \text{and} \quad S\left(\mathcal{B}^{q}, (2x)^{\frac{1}{2}} \right) = (1+o(1)) \frac{x}{\log x}.
\end{equation}
Thus, we need an estimate for $S\left(\mathcal{A}^{q}, (2x)^{\frac{1}{2}} \right)$ that holds for almost all $q \sim Q$. Our aim is to show that the sparser set $\mathcal{A}^{q}$ contains the expected proportion of primes compared to the larger set $\mathcal{B}^{q}$, which requires us to decompose $S\left(\mathcal{A}^{q}, (2x)^{\frac{1}{2}} \right)$ and prove asymptotic formulas of the form
\begin{equation}
S\left(\mathcal{A}^{q}, z \right) = (1+o(1)) \frac{1}{\varphi(q)} S\left(\mathcal{B}^{q}, z \right)
\end{equation}
for some parts of it, and drop the other parts. The dropped parts must be negative, since we want an upper bound. If we have, for almost all $q \sim Q$,
\begin{equation}
S\left(\mathcal{A}^{q}, (2x)^{\frac{1}{2}} \right) < C(\theta) \frac{x}{\varphi(q) \log x},
\end{equation}
then trivially
\begin{equation}
\pi(x; q, a) < C(\theta) \frac{x}{\varphi(q) \log x}.
\end{equation}
Note that $C(\theta)$ is monotonic increasing ([\cite{676}, Theorem 1(i)]). Combining (15) with the bound
\begin{equation}
\int_{0.5}^{0.679} C(\theta) d \theta < \frac{1}{2},
\end{equation}
we can easily deduce (11) and prove Theorem~\ref{t1}. Now we only need to prove (16) by giving upper bounds for the function $C(\theta)$. In Section 6 we shall consider it with $\theta$ lies in different intervals. Note that each $\theta$ corresponds to an independent sieve decomposition.

In the process of giving asymptotic formulas like (13) in Harman's sieve, one needs the following types of arithmetic information (with all variables are dyadic constrained, and all coefficients are divisor-bounded):
\begin{equation}
\sum_{q \sim Q} \gamma_q \sum_{m} a_m \sum_{m n_1 \cdots n_j \in \mathcal{A}^q } 1 = (1+o(1)) \frac{1}{\varphi(q)} \sum_{q \sim Q} \gamma_q \sum_{m} a_m \sum_{m n_1 \cdots n_j \in \mathcal{B}^q } 1 , \tag{Type-I$_j$}
\end{equation}
\begin{equation}
\sum_{q \sim Q} \gamma_q \sum_{m_1 \cdots m_j \in \mathcal{A}^q } a_{1, m_1} a_{2, m_2} \cdots a_{j, m_j} = (1+o(1)) \frac{1}{\varphi(q)} \sum_{q \sim Q} \gamma_q \sum_{m_1 \cdots m_j \in \mathcal{B}^q } a_{1, m_1} a_{2, m_2} \cdots a_{j, m_j} , \tag{Type-II$_j$}
\end{equation}
\begin{equation}
\sum_{q \sim Q} \gamma_q \sum_{m, n} a_m b_n \sum_{m n l \in \mathcal{A}^q } 1 = (1+o(1)) \frac{1}{\varphi(q)} \sum_{q \sim Q} \gamma_q \sum_{m, n} a_m b_n \sum_{m n l \in \mathcal{B}^q } 1 . \tag{Type-I/II}
\end{equation}
We can combine Type-I$_1$ and Type-II information together to prove asymptotic formulas for sums of $S\left(\mathcal{A}_{p_1 \cdots p_n}, x^{\kappa_0}\right)$, and sometimes we can use only Type-II information to give asymptotic formulas for sums of $S\left(\mathcal{A}_{p_1 \cdots p_n}, p_n\right)$. For the two-dimensional or even three-dimensional alternative sieve, Type-I$_2$ and Type-I$_3$ information is required.

\section{Type-II arithmetic information}
In this section we give several results of the form
$$
\sum_{q \sim Q} \gamma_q \sum_{\substack{m_1 \cdots m_j \in \mathcal{A}^q \\ m_i \sim M_i \\ 1 \leqslant i \leqslant j}} a_{1, m_1} a_{2, m_2} \cdots a_{j, m_j} = (1+o(1)) \frac{1}{\varphi(q)} \sum_{q \sim Q} \gamma_q \sum_{\substack{m_1 \cdots m_j \in \mathcal{B}^q \\ m_i \sim M_i \\ 1 \leqslant i \leqslant j}} a_{1, m_1} a_{2, m_2} \cdots a_{j, m_j}
$$
(The errors, here and in lemmas below, can be bounded by $x (\log x)^{-A}$ for some large $A > 1$.) The three conditions we want the coefficients to satisfy are as follows:

(\textbf{Condition 3.1}: Siegel--Walfisz condition) For any $f \geqslant 1$, $k \geqslant 1$, $b \neq 0$ and $(k, b) = 1$, we have
$$
\sum_{\substack{l \sim L \\ l \equiv b (\mod k) \\ (l, f) = 1 }} \lambda_l = \frac{1}{\varphi(k)} \sum_{\substack{l \sim L \\ (l, f k) = 1 }} \lambda_l + O \left(\frac{\left(\sum_{l \sim L} |\lambda_l|^2 \right)^{\frac{1}{2}} L^{\frac{1}{2}} (d(f))^B }{(\log L)^{A}} \right).
$$

(\textbf{Condition 3.2}: No small prime factors) We have $\lambda_l = 0$ whenever $l$ has a prime factor less than $\exp\left(\log x (\log \log x)^{-2} \right)$.

(\textbf{Condition 3.3}: Values often larger than $1$) We have
$$
L^{1 - \varepsilon} \sum_{l \sim L} |\lambda_l|^4 \ll \left(\sum_{l \sim L} |\lambda_l|^2 \right)^{2}.
$$

\begin{lemma}\label{l31} ([\cite{676}, Lemma 3], [\cite{HarmanBOOK}, Lemma 8.1]).
Let $M_1 M_2 = x$ and $\min(M_1, M_2) > x^{\varepsilon^2}$. Suppose that $a_{2, m_2}$ satisfies \textbf{Conditions 3.1--3.3}. If we have
$$
x^{\varepsilon^2 - 1} Q^2 < M_2 < x^{\frac{5}{6} - \varepsilon^2} Q^{-\frac{4}{3}},
$$
then
$$
\sum_{q \sim Q} \gamma_q \sum_{\substack{m_1 m_2 \in \mathcal{A}^q \\ m_i \sim M_i \\ 1 \leqslant i \leqslant 2}} a_{1, m_1} a_{2, m_2} = (1+o(1)) \frac{1}{\varphi(q)} \sum_{q \sim Q} \gamma_q \sum_{\substack{m_1 m_2 \in \mathcal{B}^q \\ m_i \sim M_i \\ 1 \leqslant i \leqslant 2}} a_{1, m_1} a_{2, m_2}.
$$
\end{lemma}
\begin{proof}
This follows from [\cite{BFI}, Theorem 3].
\end{proof}

\begin{lemma}\label{l32} ([\cite{676}, Lemma 4], [\cite{HarmanBOOK}, Lemma 8.2]).
Let $M_1 M_2 M_3 = x$, $\min(M_1, M_2, M_3) > x^{\varepsilon^2}$ and $A > 0$. Suppose that $a_{j, m_j}$ ($j = 1, 2, 3$) satisfy \textbf{Condition 3.2} and $a_{2, m_2}$ also satisfies \textbf{Condition 3.1}. If there exists $B = B(A)$ such that the following conditions
$$
Q (\log x)^{B} < M_1 M_2,
$$
$$
M_1^2 M_2^3 < Q x^{1 - \varepsilon^2},
$$
$$
M_1^5 M_2^2 < x^{2 - \varepsilon^2}
$$
and
$$
M_1^4 M_2^3 < x^{2 - \varepsilon^2}
$$
hold, then
$$
\sum_{q \sim Q} \gamma_q \sum_{\substack{m_1 m_2 m_3 \in \mathcal{A}^q \\ m_i \sim M_i \\ 1 \leqslant i \leqslant 3}} a_{1, m_1} a_{2, m_2} a_{3, m_3} = (1+o(1)) \frac{1}{\varphi(q)} \sum_{q \sim Q} \gamma_q \sum_{\substack{m_1 m_2 m_3 \in \mathcal{B}^q \\ m_i \sim M_i \\ 1 \leqslant i \leqslant 3}} a_{1, m_1} a_{2, m_2} a_{3, m_3}.
$$
\end{lemma}
\begin{proof}
This follows from [\cite{BFI2}, Theorem 3].
\end{proof}

\begin{lemma}\label{l33} ([\cite{676}, Lemma 5], [\cite{HarmanBOOK}, Lemma 8.3]).
Let $M_1 M_2 M_3 = x$, $\min(M_1, M_2, M_3) > x^{\varepsilon^2}$ and $A > 0$. Suppose that $a_{j, m_j}$ ($j = 1, 2, 3$) satisfy \textbf{Condition 3.2} and $a_{2, m_2}$ also satisfies \textbf{Condition 3.1}. If there exists $B = B(A)$ such that the following conditions
$$
Q (\log x)^{B} < M_1 M_2,
$$
$$
M_1 M_2^2 Q^2 < x^{2 - 3 \varepsilon^2}
$$
and
$$
M_1^5 M_2^2 < x^{2 - 3 \varepsilon^2}
$$
hold, then
$$
\sum_{q \sim Q} \gamma_q \sum_{\substack{m_1 m_2 m_3 \in \mathcal{A}^q \\ m_i \sim M_i \\ 1 \leqslant i \leqslant 3}} a_{1, m_1} a_{2, m_2} a_{3, m_3} = (1+o(1)) \frac{1}{\varphi(q)} \sum_{q \sim Q} \gamma_q \sum_{\substack{m_1 m_2 m_3 \in \mathcal{B}^q \\ m_i \sim M_i \\ 1 \leqslant i \leqslant 3}} a_{1, m_1} a_{2, m_2} a_{3, m_3}.
$$
\end{lemma}
\begin{proof}
This is a slight variant of [\cite{BFI2}, Theorem 4].
\end{proof}

\begin{lemma}\label{l34}
Let $M_1 M_2 M_3 = x$, $\min(M_1, M_2, M_3) > x^{\varepsilon^2}$ and $Q < x^{0.7 - \varepsilon^2}$. Suppose that $a_{j, m_j}$ ($j = 1, 2, 3$) satisfy \textbf{Condition 3.2} and $a_{2, m_2}$ also satisfies \textbf{Condition 3.1}. If the following conditions
$$
Q x^{\varepsilon^2} < M_1 M_2,
$$
$$
M_2 < Q^{-1} x^{1 - 2 \varepsilon^2},
$$
$$
M_1 M_2 < Q^{-\frac{1}{7}} x^{\frac{153}{224} - 10 \varepsilon^2}
$$
and
$$
M_1^4 M_2 < Q^{-1} x^{\frac{57}{32} - 10 \varepsilon^2}
$$
hold, then
$$
\sum_{q \sim Q} \gamma_q \sum_{\substack{m_1 m_2 m_3 \in \mathcal{A}^q \\ m_i \sim M_i \\ 1 \leqslant i \leqslant 3}} a_{1, m_1} a_{2, m_2} a_{3, m_3} = (1+o(1)) \frac{1}{\varphi(q)} \sum_{q \sim Q} \gamma_q \sum_{\substack{m_1 m_2 m_3 \in \mathcal{B}^q \\ m_i \sim M_i \\ 1 \leqslant i \leqslant 3}} a_{1, m_1} a_{2, m_2} a_{3, m_3}.
$$
\end{lemma}
\begin{proof}
This follows from [\cite{MaynardLargeModuliI}, Proposition 8.3].
\end{proof}

\section{Sieve asymtotic formulas}
In this section we give asymptotic formulas for sums of sieve functions $S\left(\mathcal{A}_{p_1 \cdots p_n}, p_n\right)$ and $S\left(\mathcal{A}_{p_1 \cdots p_n}, x^{\kappa_0}\right)$ with $\kappa_0 = \kappa$ or $\kappa^{\prime}$, where
\begin{align*}
\kappa = \kappa(\theta) =
\begin{cases}
\frac{5-8\theta}{6} - \varepsilon, & \theta \leqslant \frac{17}{32} - \varepsilon, \\
\frac{5-8\theta}{12} - 3 \varepsilon, & \frac{17}{32} - \varepsilon < \theta \leqslant \frac{7}{13} - \varepsilon, \\
\frac{3-5\theta}{7} - 2 \varepsilon, & \frac{7}{13} - \varepsilon < \theta \leqslant \frac{4}{7} - \varepsilon, \\
\end{cases}
\end{align*}
and
\begin{align*}
\kappa^{\prime} = \kappa^{\prime}(\theta) =
\begin{cases}
\frac{11-20\theta}{6} - 2 \varepsilon, & \frac{7}{13} - \varepsilon < \theta \leqslant \frac{11}{20} - \varepsilon, \\
\kappa, & \text{otherwise}. \\
\end{cases}
\end{align*}
We also write
\begin{align*}
\tau = \tau(\theta) =
\begin{cases}
\frac{3(1-\theta)}{5} - \varepsilon, & \theta \leqslant \frac{11}{21}, \\
\frac{2}{7} - \varepsilon, & \frac{11}{21} < \theta \leqslant \frac{6}{11} - \varepsilon, \\
\frac{5-6\theta}{7} - \varepsilon, & \frac{6}{11} - \varepsilon < \theta, \\
\end{cases}
\end{align*}
and
\begin{align*}
\tau^{\prime} = \tau^{\prime}(\theta) =
\begin{cases}
\frac{5-6\theta}{7}, & \frac{7}{13} - \varepsilon < \theta \leqslant \frac{11}{20} - \varepsilon, \\
\tau, & \text{otherwise}. \\
\end{cases}
\end{align*}

Now, we put $p_i = x^{\alpha_i}$ and write $\boldsymbol{\alpha}_{n}$ to denote $(\alpha_1, \ldots, \alpha_n)$. We shall use the terms \textit{partition} and \textit{exactly partition} many times in the rest of our paper. One can see [\cite{HarmanBOOK}, Page 162] for the definition of them. Next, we shall define some regions that represent parts of the conditions of the asymptotic formulas. Note that all regions below are defined on $\theta \in [\frac{1}{2}, \frac{4}{7}]$.
\begin{align}
\nonumber \boldsymbol{g}_{1} =&\ \left\{(s, t): 2 \theta - 1 < s < \frac{5-8\theta}{6} \right\}, \\
\nonumber \boldsymbol{g}_{2} =&\ \left\{(s, t): s + t > \theta ,\ 2 s + 3 t < 1 + \theta,\ 5 s + 2 t < 2,\ 4 s + 3 t < 2 \right\}, \\
\nonumber \boldsymbol{g}_{3} =&\ \left\{(s, t): s + t > \theta,\ s + 2 t < 2 - 2 \theta,\ 5 s + 2 t < 2 \right\}, \\
\nonumber \boldsymbol{g}_{4} =&\ \left\{(s, t): s + t > \theta,\ t < 1 - \theta,\ s + t < \frac{153}{224} - \frac{1}{7} \theta,\ 4 s + t < \frac{57}{32} - \theta \right\}, \\
\nonumber \boldsymbol{A}_{j} =&\ \left\{\boldsymbol{\alpha}_{j}: \varepsilon^2 \leqslant \alpha_j < \cdots < \alpha_1 < \tau,\ \alpha_1 + \cdots + \alpha_j \leqslant 1 \right\}, \\
\nonumber \boldsymbol{T}^{*} =&\ \left\{(s, t): 0 \leqslant s \leqslant \frac{8\theta-2}{7},\ 0 \leqslant t \leqslant \frac{5-6\theta}{7} \right\}, \\
\nonumber \boldsymbol{T}_{j}^{*} =&\ \left\{\boldsymbol{\alpha}_{j}: \boldsymbol{\alpha}_{j} \text{ partitions exactly into } \boldsymbol{T}^{*} \right\}, \\
\nonumber \boldsymbol{S} =&\ \left\{(s, t): s < 1 - \theta,\ s + 2t < 2 - 2 \theta,\ s + 4t < 2 - \theta \right\}, \\
\nonumber \boldsymbol{S}_{j} =&\ \left\{\boldsymbol{\alpha}_{j}: \boldsymbol{\alpha}_{j} \text{ partitions exactly into } \boldsymbol{S} \right\}, \\
\nonumber \boldsymbol{U}_{j}^{\prime} = \boldsymbol{U}_{j}^{\prime}(\theta) =&\ \left\{\boldsymbol{\alpha}_{j}: \boldsymbol{\alpha}_{j} \in \boldsymbol{A}_{j},\ (\alpha_1, \ldots, \alpha_n, 2\theta-1+\varepsilon) \in \boldsymbol{S}_{j+1} \right\},
\end{align}
\begin{align*}
\boldsymbol{U}_{j} = \boldsymbol{U}_{j}(\theta)  =
\begin{cases}
\boldsymbol{U}_{j}^{\prime}(\theta), & \theta < \frac{7}{13}, \\
\left\{\boldsymbol{\alpha}_{j}: \boldsymbol{\alpha}_{j} \in \boldsymbol{A}_{j},\ \boldsymbol{\alpha}_{j} \in \boldsymbol{T}_{j}^{*} \right\}, & \theta \geqslant \frac{7}{13}, \\
\end{cases}
\end{align*}
\begin{align}
\nonumber \boldsymbol{T}^{**} =&\ \left\{(s, t): \frac{3}{7} < s < 1 - \theta,\ 0 \leqslant t \leqslant \frac{1-s}{2} \right\}, \\
\nonumber \boldsymbol{U}_{j}^{*} =&\ \left\{\boldsymbol{\alpha}_{j}: \boldsymbol{\alpha}_{j} \text{ partitions exactly into } \boldsymbol{T}^{**} \right\}, \\
\nonumber \boldsymbol{G}_{j} =&\ \left\{\boldsymbol{\alpha}_{j}: \boldsymbol{\alpha}_{j} \in \boldsymbol{U}_{j}^{*}, \text{ or } \boldsymbol{\alpha}_{j} \text{ partitions exactly into } \boldsymbol{g}_{1} \cup \boldsymbol{g}_{2} \cup \boldsymbol{g}_{3} \cup \boldsymbol{g}_{4} \right\}, \\
\nonumber \boldsymbol{R} =&\ \left\{\boldsymbol{\alpha}_{2}: \alpha_2 \leqslant \alpha_1,\ \alpha_1 + 2 \alpha_2 \leqslant 1,\ \alpha_1 + 4 \alpha_2 \geqslant 3 - 3 \theta,\ 3 \alpha_2 \geqslant 2 \alpha_1, \right.\\
\nonumber & \left. \quad \max\left(\tau,\ \frac{31 \theta - 15}{3} \right) \leqslant \alpha_1 \leqslant \min\left(\frac{3}{7},\ 4 - 7 \theta \right) \right\}, \\
\nonumber \boldsymbol{R}_0 =&\ \left\{\boldsymbol{\alpha}_{2}: \alpha_1 + 2 \alpha_2 \leqslant 1,\ \alpha_1 + 4 \alpha_2 \geqslant 3 - 3 \theta,\ 3 \alpha_2 \geqslant 2 \alpha_1, \right.\\
\nonumber & \left. \quad \max\left(\frac{19 \theta - 7}{7},\ \frac{50 \theta - 19}{17} \right) \leqslant \alpha_1 \leqslant \frac{3}{7} \right\}, \\
\nonumber \boldsymbol{D}_1 =&\ \left\{\boldsymbol{\alpha}_{3}: \alpha_1 \geqslant \kappa,\ \alpha_2 \geqslant \kappa,\ \alpha_3 \geqslant \kappa,\ \boldsymbol{\alpha}_{3} \notin \boldsymbol{G}_3, \right.\\
\nonumber & \left. \quad (\alpha_1, \alpha_2 + \alpha_3) \in \boldsymbol{R} \text{ with } \alpha_2 \geqslant \alpha_3 \text{ or } (\alpha_1 + \alpha_2, \alpha_3) \in \boldsymbol{R} \text{ with } \alpha_1 \geqslant \alpha_2 \right\}, \\
\nonumber \boldsymbol{D}_2 =&\ \left\{\boldsymbol{\alpha}_{4}: \alpha_1 \geqslant \kappa,\ \alpha_2 \geqslant \kappa,\ \alpha_3 \geqslant \kappa,\ \alpha_4 \geqslant \kappa, \right.\\
\nonumber & \left. \quad \boldsymbol{\alpha}_{4} \notin \boldsymbol{G}_4,\ (\alpha_1 + \alpha_2, \alpha_3 + \alpha_4) \in \boldsymbol{R} \right\}, \\
\nonumber \boldsymbol{D}_3 =&\ \left\{\boldsymbol{\alpha}_{3}: \alpha_1 \geqslant \kappa,\ \alpha_2 \geqslant \kappa,\ \alpha_3 \geqslant \kappa,\ \boldsymbol{\alpha}_{3} \notin \boldsymbol{G}_3, \right.\\
\nonumber & \left. \quad (\alpha_1, \alpha_2 + \alpha_3) \in \boldsymbol{R}_0 \text{ with } \alpha_2 \geqslant \alpha_3 \text{ or } (\alpha_1 + \alpha_2, \alpha_3) \in \boldsymbol{R}_0 \text{ with } \alpha_1 \geqslant \alpha_2 \right\}, \\
\nonumber \boldsymbol{D}_4 =&\ \left\{\boldsymbol{\alpha}_{4}: \alpha_1 \geqslant \kappa,\ \alpha_2 \geqslant \kappa,\ \alpha_3 \geqslant \kappa,\ \alpha_4 \geqslant \kappa, \right.\\
\nonumber & \left. \quad \boldsymbol{\alpha}_{4} \notin \boldsymbol{G}_4,\ (\alpha_1 + \alpha_2, \alpha_3 + \alpha_4) \in \boldsymbol{R}_0 \right\}, \\
\nonumber U_2 =&\ \left\{\boldsymbol{\alpha}_{2}: \frac{1}{2} \leqslant \theta < \frac{17}{32},\ \kappa \leqslant \alpha_1 \leqslant \frac{3}{7},\ \kappa \leqslant \alpha_2 < \min\left(\alpha_1, \frac{1}{2}(1-\alpha_1) \right) \right\}, \\
\nonumber A =&\ \left\{\boldsymbol{\alpha}_{2}: \boldsymbol{\alpha}_{2} \in U_2,\ \boldsymbol{\alpha}_{2} \notin \boldsymbol{G}_{2},\ \alpha_1 + \alpha_2 < \theta \right\}, \\
\nonumber B =&\ \left\{\boldsymbol{\alpha}_{2}: \boldsymbol{\alpha}_{2} \in U_2,\ \boldsymbol{\alpha}_{2} \notin \boldsymbol{G}_{2} \cup A,\ \alpha_1 + 4 \alpha_2 < 3 - 3 \theta \right\}, \\
\nonumber C =&\ \left\{\boldsymbol{\alpha}_{2}: \boldsymbol{\alpha}_{2} \in U_2,\ \boldsymbol{\alpha}_{2} \notin \boldsymbol{G}_{2} \cup A \cup B \right\}.
\end{align}

Then we list some asymptotic formulas for sums whose variables lie in some suitable intervals. Note that from here we write $\mathcal{A}^{q}$ and $\mathcal{B}^{q}$ as $\mathcal{A}$ and $\mathcal{B}$ for simplicity, and the readers should remember the outer summation over $q$.

\begin{lemma}\label{l41} ([\cite{676}, Lemmas 7 and 18], [\cite{HarmanBOOK}, Lemmas 8.6 and 8.14]).
Let $\frac{1}{2} \leqslant \theta < \frac{4}{7}$ and suppose that $\min \boldsymbol{\alpha}_{j} \geq \varepsilon^2$. Then we have, for almost all $q \sim x^{\theta}$,
$$
\sum_{\boldsymbol{\alpha}_{j} \in \boldsymbol{G}_{j}} S\left(\mathcal{A}_{p_1 \cdots p_j}, p_j\right) = (1+o(1)) \frac{1}{\varphi(q)} \sum_{\boldsymbol{\alpha}_{j} \in \boldsymbol{G}_{j}} S\left(\mathcal{B}_{p_1 \cdots p_j}, p_j\right).
$$
\end{lemma}

\begin{lemma}\label{l42} ([\cite{676}, Lemmas 15 and 16], [\cite{HarmanBOOK}, Lemmas 8.11 and 8.12]).
Let $\frac{1}{2} \leqslant \theta < \frac{4}{7}$. Then we have, for almost all $q \sim x^{\theta}$,
$$
\sum_{\boldsymbol{\alpha}_{j} \in \boldsymbol{U}_{j}} S\left(\mathcal{A}_{p_1 \cdots p_j}, x^{\kappa}\right) = (1+o(1)) \frac{1}{\varphi(q)} \sum_{\boldsymbol{\alpha}_{j} \in \boldsymbol{U}_{j}} S\left(\mathcal{B}_{p_1 \cdots p_j}, x^{\kappa}\right).
$$
Let $\frac{1}{2} \leqslant \theta < \frac{11}{20}$. Then we have, for almost all $q \sim x^{\theta}$,
$$
\sum_{\boldsymbol{\alpha}_{j} \in \boldsymbol{U}_{j}^{\prime}} S\left(\mathcal{A}_{p_1 \cdots p_j}, x^{\kappa^{\prime}}\right) = (1+o(1)) \frac{1}{\varphi(q)} \sum_{\boldsymbol{\alpha}_{j} \in \boldsymbol{U}_{j}^{\prime}} S\left(\mathcal{B}_{p_1 \cdots p_j}, x^{\kappa^{\prime}}\right).
$$
\end{lemma}

\begin{lemma}\label{l43} ([\cite{HarmanBOOK}, Lemma 8.15]).
Let $\frac{1}{2} \leqslant \theta < \frac{17}{32}$ and $M < x^{2 - \varepsilon} Q^{-3}$. Then we have, for almost all $q \sim x^{\theta}$ and a divisor-bounded sequence $a_m$,
$$
\sum_{m \sim M} a_m S\left(\mathcal{A}_{m}, x^{\kappa}\right) = (1+o(1)) \frac{1}{\varphi(q)} \sum_{m \sim M} a_m S\left(\mathcal{B}_{m}, x^{\kappa}\right).
$$
Note that we have $x^{\frac{3}{7}} < x^{2 - \varepsilon} Q^{-3}$ when $\theta < \frac{11}{21} - \varepsilon$.
\end{lemma}

\begin{lemma}\label{l44} ([\cite{HarmanBOOK}, Lemma 8.16]).
Let $\frac{1}{2} \leqslant \theta < \frac{11}{21}$. Then we have, for almost all $q \sim x^{\theta}$,
$$
\sum_{\boldsymbol{\alpha}_{2} \in A \cup B} S\left(\mathcal{A}_{p_1 p_2}, x^{\kappa}\right) = (1+o(1)) \frac{1}{\varphi(q)} \sum_{\boldsymbol{\alpha}_{2} \in A \cup B} S\left(\mathcal{B}_{p_1 p_2}, x^{\kappa}\right).
$$
\end{lemma}

\section{High-dimensional sieves}
In this section, we mention several results regarding the upper bounds for some sieve functions. These upper bounds are proved using two- or three-dimensional sieves.

\begin{lemma}\label{l51} ([\cite{676}, Lemma 20], [\cite{HarmanBOOK}, Lemma 8.18]).
Let $\frac{1}{2} \leqslant \theta < \frac{17}{32}$. Then we have, for almost all $q \sim x^{\theta}$,
$$
- \sum_{1-\theta \leqslant \alpha_1 < \frac{1}{2}} S\left(\mathcal{A}_{p_1}, p_1 \right) \leqslant (1+o(1)) \frac{1}{\varphi(q)} \left( - \sum_{1-\theta \leqslant \alpha_1 < \frac{1}{2}} S\left(\mathcal{B}_{p_1}, p_1 \right) + \sum_{\substack{\kappa \leqslant \alpha_3 < \alpha_1 < \frac{\theta}{2} \\ 1 - \theta \leqslant \alpha_1 + \beta \leqslant \theta \\ \alpha_3 < \min\left(\beta, \frac{1}{2}(1-\alpha_1-\beta) \right) \\ (\alpha_1, \beta, \alpha_3) \notin \boldsymbol{G}_{3} }} S\left(\mathcal{B}_{p_1 m p_3}, p_3 \right) \right),
$$
where $m = x^{\beta}$ and $\left(m, P(p_1)\right) = 1$.
\end{lemma}

\begin{lemma}\label{l52} ([\cite{676}, Lemma 24], [\cite{HarmanBOOK}, Lemma 8.23]).
Let $\frac{1}{2} \leqslant \theta < \frac{16}{31}$. Then we have, for almost all $q \sim x^{\theta}$,
$$
\sum_{\boldsymbol{\alpha}_{2} \in \boldsymbol{R}} S\left(\mathcal{A}_{p_1 p_2}, x^{\kappa}\right) \leqslant (1+o(1)) \frac{1}{\varphi(q)} \left( \sum_{\boldsymbol{\alpha}_{2} \in \boldsymbol{R}} S\left(\mathcal{B}_{p_1 p_2}, x^{\kappa}\right) + \frac{x}{\log x} (I_1 + I_2) \right),
$$
where
$$
I_1 = \frac{1}{\kappa} \int_{(t_1, t_2, t_3) \in \boldsymbol{D}_{1}} \frac{\omega\left(\frac{1 - t_1 - t_2 - t_3}{\kappa}\right)}{t_1 t_2 t_3} d t_3 d t_2 d t_1,
$$
$$
I_2 = \frac{1}{\kappa} \int_{(t_1, t_2, t_3, t_4) \in \boldsymbol{D}_{2}} \frac{\omega\left(\frac{1 - t_1 - t_2 - t_3 - t_4}{\kappa}\right)}{t_1 t_2 t_3 t_4} d t_4 d t_3 d t_2 d t_1,
$$
where $\omega(u)$, here and below, denote the Buchstab function determined by the following differential--difference equation
\begin{align*}
\begin{cases}
\omega(u)=\frac{1}{u}, & \quad 1 \leqslant u \leqslant 2, \\
(u \omega(u))^{\prime}= \omega(u-1), & \quad u \geqslant 2 .
\end{cases}
\end{align*}
\end{lemma}

\begin{lemma}\label{l53} ([\cite{677}, Lemma 24]).
Let $\frac{25}{49} \leqslant \theta \leqslant \frac{92}{175}$. Then we have, for almost all $q \sim x^{\theta}$,
$$
\sum_{\boldsymbol{\alpha}_{2} \in \boldsymbol{R}_0} S\left(\mathcal{A}_{p_1 p_2}, x^{\kappa}\right) \leqslant (1+o(1)) \frac{1}{\varphi(q)} \left( \sum_{\boldsymbol{\alpha}_{2} \in \boldsymbol{R}_0} S\left(\mathcal{B}_{p_1 p_2}, x^{\kappa}\right) + \frac{x}{\log x} (I_3 + I_4) \right),
$$
where
$$
I_3 = \frac{1}{\kappa} \int_{(t_1, t_2, t_3) \in \boldsymbol{D}_{3}} \frac{\omega\left(\frac{1 - t_1 - t_2 - t_3}{\kappa}\right)}{t_1 t_2 t_3} d t_3 d t_2 d t_1,
$$
$$
I_4 = \frac{1}{\kappa} \int_{(t_1, t_2, t_3, t_4) \in \boldsymbol{D}_{4}} \frac{\omega\left(\frac{1 - t_1 - t_2 - t_3 - t_4}{\kappa}\right)}{t_1 t_2 t_3 t_4} d t_4 d t_3 d t_2 d t_1.
$$
\end{lemma}

\section{The final decomposition}
In this section, we ignore the presence of $\varepsilon$ for clarity. Let $\omega(u)$ denote the Buchstab function defined in Lemma~\ref{l52}. Let $\frac{1}{2} \leqslant \theta \leqslant 0.679$. We shall split this range of $\theta$ to several subranges and use different methods to treat them and obtain good bounds for $C(\theta)$.

\subsection{Case 1. $\frac{1}{2} \leqslant \theta < \frac{25}{49}$}
In the first subrange, we can use the lemmas in Sections 4 and 5 with their full power. Using Buchstab's identity, we have
\begin{align}
\nonumber S\left(\mathcal{A}, (2x)^{\frac{1}{2}} \right) =&\ S\left(\mathcal{A}, x^{\kappa} \right) - \sum_{\kappa \leqslant \alpha_1 < \frac{1}{2}} S\left(\mathcal{A}_{p_1}, p_1 \right) \\
\nonumber =&\ S\left(\mathcal{A}, x^{\kappa} \right) - \sum_{\kappa \leqslant \alpha_1 \leqslant \frac{3}{7}} S\left(\mathcal{A}_{p_1}, p_1 \right) - \sum_{\frac{3}{7} < \alpha_1 < 1 - \theta} S\left(\mathcal{A}_{p_1}, p_1 \right) - \sum_{1 - \theta \leqslant \alpha_1 < \frac{1}{2}} S\left(\mathcal{A}_{p_1}, p_1 \right) \\
=&\ S_{11} - S_{12} - S_{13} - S_{14}.
\end{align}
By Lemma~\ref{l43}, we can give an asymptotic formula for $S_{11}$. By Lemma~\ref{l41}, we can also give an asymptotic formula for $S_{13}$. For $S_{14}$, we can use Lemma~\ref{l51} to give an upper bound with a loss of
\begin{equation}
\sum_{\substack{\kappa \leqslant \alpha_3 < \alpha_1 < \frac{\theta}{2} \\ 1 - \theta \leqslant \alpha_1 + \beta \leqslant \theta \\ \alpha_3 < \min\left(\beta, \frac{1}{2}(1-\alpha_1-\beta) \right) \\ (\alpha_1, \beta, \alpha_3) \notin \boldsymbol{G}_{3} }} S\left(\mathcal{A}_{p_1 m p_3}, p_3 \right).
\end{equation}

For the remaining $S_{12}$, we perform Buchstab's identity twice again to get
\begin{align}
\nonumber S_{12} =&\ \sum_{\kappa \leqslant \alpha_1 \leqslant \frac{3}{7}} S\left(\mathcal{A}_{p_1}, p_1 \right) \\
\nonumber =&\ \sum_{\kappa \leqslant \alpha_1 \leqslant \frac{3}{7}} S\left(\mathcal{A}_{p_1}, x^{\kappa} \right) - \sum_{\substack{\kappa \leqslant \alpha_1 \leqslant \frac{3}{7} \\ \kappa \leqslant \alpha_2 < \min\left(\alpha_1, \frac{1}{2}(1-\alpha_1)\right)}} S\left(\mathcal{A}_{p_1 p_2}, p_2 \right) \\
\nonumber =&\ \sum_{\kappa \leqslant \alpha_1 \leqslant \frac{3}{7}} S\left(\mathcal{A}_{p_1}, x^{\kappa} \right) - \sum_{\substack{\kappa \leqslant \alpha_1 \leqslant \frac{3}{7} \\ \kappa \leqslant \alpha_2 < \min\left(\alpha_1, \frac{1}{2}(1-\alpha_1)\right) \\ \boldsymbol{\alpha}_{2} \in \boldsymbol{G}_{2} }} S\left(\mathcal{A}_{p_1 p_2}, p_2 \right) - \sum_{\substack{\kappa \leqslant \alpha_1 \leqslant \frac{3}{7} \\ \kappa \leqslant \alpha_2 < \min\left(\alpha_1, \frac{1}{2}(1-\alpha_1)\right) \\ \boldsymbol{\alpha}_{2} \notin \boldsymbol{G}_{2} }} S\left(\mathcal{A}_{p_1 p_2}, p_2 \right) \\
=&\ S_{121} - S_{122} - S_{123}.
\end{align}
We can give asymptotic formulas for $S_{121}$ by Lemma~\ref{l43} and for $S_{122}$ by Lemma~\ref{l41}. For $S_{123}$, we can easily find that $\boldsymbol{\alpha}_{2} \in U_2$. Thus, we split it into two parts $\boldsymbol{\alpha}_{2} \in A \cup B$ and $\boldsymbol{\alpha}_{2} \in C$:
\begin{align}
\nonumber S_{123} =&\ \sum_{\boldsymbol{\alpha}_{2} \in A \cup B } S\left(\mathcal{A}_{p_1 p_2}, p_2 \right) + \sum_{\boldsymbol{\alpha}_{2} \in C } S\left(\mathcal{A}_{p_1 p_2}, p_2 \right) \\
\nonumber =&\ \sum_{\boldsymbol{\alpha}_{2} \in A \cup B } S\left(\mathcal{A}_{p_1 p_2}, x^{\kappa} \right) - \sum_{\substack{\boldsymbol{\alpha}_{2} \in A \cup B \\ \kappa \leqslant \alpha_3 < \min\left(\alpha_2, \frac{1}{2}(1-\alpha_1 -\alpha_2)\right) \\ \boldsymbol{\alpha}_{3} \in \boldsymbol{G}_{3} }} S\left(\mathcal{A}_{p_1 p_2 p_3}, p_3 \right) - \sum_{\substack{\boldsymbol{\alpha}_{2} \in A \cup B \\ \kappa \leqslant \alpha_3 < \min\left(\alpha_2, \frac{1}{2}(1-\alpha_1 -\alpha_2)\right) \\ \boldsymbol{\alpha}_{3} \notin \boldsymbol{G}_{3} }} S\left(\mathcal{A}_{p_1 p_2 p_3}, p_3 \right) \\
\nonumber & + \sum_{\boldsymbol{\alpha}_{2} \in C } S\left(\mathcal{A}_{p_1 p_2}, x^{\kappa} \right) - \sum_{\substack{\boldsymbol{\alpha}_{2} \in C \\ \kappa \leqslant \alpha_3 < \min\left(\alpha_2, \frac{1}{2}(1-\alpha_1 -\alpha_2)\right) \\ \boldsymbol{\alpha}_{3} \in \boldsymbol{G}_{3} }} S\left(\mathcal{A}_{p_1 p_2 p_3}, p_3 \right) - \sum_{\substack{\boldsymbol{\alpha}_{2} \in C \\ \kappa \leqslant \alpha_3 < \min\left(\alpha_2, \frac{1}{2}(1-\alpha_1 -\alpha_2)\right) \\ \boldsymbol{\alpha}_{3} \notin \boldsymbol{G}_{3} }} S\left(\mathcal{A}_{p_1 p_2 p_3}, p_3 \right) \\
=&\ S_{1231} - S_{1232} - S_{1233} + S_{1234} - S_{1235} - S_{1236}.
\end{align}
For $S_{1231}$, we can use Lemma~\ref{l44} to give an asymptotic formula. We can also give asymptotic formulas for $S_{1232}$ and $S_{1235}$ by Lemma~\ref{l41}. However, for $S_{1234}$ we cannot give an asymptotic formula. Note that we cannot discard any part of $S_{1234}$ since it has a positive sign. Now we must use Lemma~\ref{l52} to give an upper bound for this sum. After checking the boundaries of the region $\boldsymbol{R}$, we find that $\boldsymbol{R}$ covers the whole region $C$, and the loss from $S_{1234}$ is the sum of two corresponding integrals $I_1$ and $I_2$. We discard the whole sum $S_{1236}$.

Now we only need to deal with the remaining sum $S_{1233}$. We want to use Buchstab's identity twice again, hence we need to check whether the conditions in Lemma~\ref{l42} are fulfilled. That is, if we have $\boldsymbol{\alpha}_{3} \in \boldsymbol{U}_3$ and
$$
(\alpha_1, \alpha_2, \alpha_3, \alpha_3) \in \boldsymbol{U}_4,
$$
for parts of $S_{1233}$ (say, $S_{1233}^{\prime}$), then we can use Buchstab's identity twice again to get
\begin{align}
\nonumber S_{1233}^{\prime} =&\ \sum_{\substack{\boldsymbol{\alpha}_{2} \in A \cup B \\ \kappa \leqslant \alpha_3 < \min\left(\alpha_2, \frac{1}{2}(1-\alpha_1 -\alpha_2)\right) \\ \boldsymbol{\alpha}_{3} \notin \boldsymbol{G}_{3} \\ \boldsymbol{\alpha}_{3} \in \boldsymbol{U}_3 \\ (\alpha_1, \alpha_2, \alpha_3, \alpha_3) \in \boldsymbol{U}_4 }} S\left(\mathcal{A}_{p_1 p_2 p_3}, p_3 \right) \\
\nonumber =&\ \sum_{\substack{\boldsymbol{\alpha}_{2} \in A \cup B \\ \kappa \leqslant \alpha_3 < \min\left(\alpha_2, \frac{1}{2}(1-\alpha_1 -\alpha_2)\right) \\ \boldsymbol{\alpha}_{3} \notin \boldsymbol{G}_{3} \boldsymbol{\alpha}_{3} \in \boldsymbol{U}_3 \\ (\alpha_1, \alpha_2, \alpha_3, \alpha_3) \in \boldsymbol{U}_4 }} S\left(\mathcal{A}_{p_1 p_2 p_3}, x^{\kappa} \right) - \sum_{\substack{\boldsymbol{\alpha}_{2} \in A \cup B \\ \kappa \leqslant \alpha_3 < \min\left(\alpha_2, \frac{1}{2}(1-\alpha_1 -\alpha_2)\right) \\ \boldsymbol{\alpha}_{3} \notin \boldsymbol{G}_{3} \\ \boldsymbol{\alpha}_{3} \in \boldsymbol{U}_3 \\ (\alpha_1, \alpha_2, \alpha_3, \alpha_3) \in \boldsymbol{U}_4 \\ \kappa \leqslant \alpha_4 < \min\left(\alpha_3, \frac{1}{2}(1-\alpha_1 -\alpha_2 -\alpha_3)\right) \\ \boldsymbol{\alpha}_{4} \in \boldsymbol{G}_{4} }} S\left(\mathcal{A}_{p_1 p_2 p_3 p_4}, p_4 \right) \\
\nonumber & - \sum_{\substack{\boldsymbol{\alpha}_{2} \in A \cup B \\ \kappa \leqslant \alpha_3 < \min\left(\alpha_2, \frac{1}{2}(1-\alpha_1 -\alpha_2)\right) \\ \boldsymbol{\alpha}_{3} \notin \boldsymbol{G}_{3} \\ \boldsymbol{\alpha}_{3} \in \boldsymbol{U}_3 \\ (\alpha_1, \alpha_2, \alpha_3, \alpha_3) \in \boldsymbol{U}_4 \\ \kappa \leqslant \alpha_4 < \min\left(\alpha_3, \frac{1}{2}(1-\alpha_1 -\alpha_2 -\alpha_3)\right) \\ \boldsymbol{\alpha}_{4} \notin \boldsymbol{G}_{4} }} S\left(\mathcal{A}_{p_1 p_2 p_3 p_4}, x^{\kappa} \right) + \sum_{\substack{\boldsymbol{\alpha}_{2} \in A \cup B \\ \kappa \leqslant \alpha_3 < \min\left(\alpha_2, \frac{1}{2}(1-\alpha_1 -\alpha_2)\right) \\ \boldsymbol{\alpha}_{3} \notin \boldsymbol{G}_{3} \\ \boldsymbol{\alpha}_{3} \in \boldsymbol{U}_3 \\ (\alpha_1, \alpha_2, \alpha_3, \alpha_3) \in \boldsymbol{U}_4 \\ \kappa \leqslant \alpha_4 < \min\left(\alpha_3, \frac{1}{2}(1-\alpha_1 -\alpha_2-\alpha_3)\right) \\ \boldsymbol{\alpha}_{4} \notin \boldsymbol{G}_{4} \\ \kappa \leqslant \alpha_5 < \min\left(\alpha_4, \frac{1}{2}(1-\alpha_1 -\alpha_2 -\alpha_3 -\alpha_4)\right) \\ \boldsymbol{\alpha}_{5} \in \boldsymbol{G}_{5} }} S\left(\mathcal{A}_{p_1 p_2 p_3 p_4 p_5}, p_5 \right) \\
\nonumber & + \sum_{\substack{\boldsymbol{\alpha}_{2} \in A \cup B \\ \kappa \leqslant \alpha_3 < \min\left(\alpha_2, \frac{1}{2}(1-\alpha_1 -\alpha_2)\right) \\ \boldsymbol{\alpha}_{3} \notin \boldsymbol{G}_{3} \\ \boldsymbol{\alpha}_{3} \in \boldsymbol{U}_3 \\ (\alpha_1, \alpha_2, \alpha_3, \alpha_3) \in \boldsymbol{U}_4 \\ \kappa \leqslant \alpha_4 < \min\left(\alpha_3, \frac{1}{2}(1-\alpha_1 -\alpha_2-\alpha_3)\right) \\ \boldsymbol{\alpha}_{4} \notin \boldsymbol{G}_{4} \\ \kappa \leqslant \alpha_5 < \min\left(\alpha_4, \frac{1}{2}(1-\alpha_1 -\alpha_2 -\alpha_3 -\alpha_4)\right) \\ \boldsymbol{\alpha}_{5} \notin \boldsymbol{G}_{5} }} S\left(\mathcal{A}_{p_1 p_2 p_3 p_4 p_5}, p_5 \right) \\
=&\ S_{12331} - S_{12332} - S_{12333} + S_{12334} + S_{12335}.
\end{align}
We can give asymptotic formulas for $S_{12332}$ and $S_{12334}$ by Lemma~\ref{l41}. Since $\boldsymbol{\alpha}_{3} \in \boldsymbol{U}_3$, we can use Lemma~\ref{l42} to give an asymptotic formula for $S_{12331}$. Again, since $\alpha_4 < \alpha_3$ and $(\alpha_1, \alpha_2, \alpha_3, \alpha_3) \in \boldsymbol{U}_4$, we have $\boldsymbol{\alpha}_{4} \in \boldsymbol{U}_4$ and thus Lemma~\ref{l42} is usable for $S_{12333}$. We discard $S_{12335}$ and the remaining parts of $S_{1233}$, leading to a three-dimensional integral and a five-dimensional integral of loss. Note that we can use the reversed Buchstab's identity to gain a small saving over the remaining parts of $S_{1233}$. The process can be seen as the following:
\begin{align}
\nonumber S_{1233} - S_{1233}^{\prime} =&\ \sum_{\substack{\boldsymbol{\alpha}_{2} \in A \cup B \\ \kappa \leqslant \alpha_3 < \min\left(\alpha_2, \frac{1}{2}(1-\alpha_1 -\alpha_2)\right) \\ \boldsymbol{\alpha}_{3} \notin \boldsymbol{G}_{3} \\ \text{either } \boldsymbol{\alpha}_{3} \notin \boldsymbol{U}_3 \text{ or } (\alpha_1, \alpha_2, \alpha_3, \alpha_3) \notin \boldsymbol{U}_4 }} S\left(\mathcal{A}_{p_1 p_2 p_3}, p_3 \right) \\
\nonumber =&\ \sum_{\substack{\boldsymbol{\alpha}_{2} \in A \cup B \\ \kappa \leqslant \alpha_3 < \min\left(\alpha_2, \frac{1}{2}(1-\alpha_1 -\alpha_2)\right) \\ \boldsymbol{\alpha}_{3} \notin \boldsymbol{G}_{3} \\ \text{either } \boldsymbol{\alpha}_{3} \notin \boldsymbol{U}_3 \text{ or } (\alpha_1, \alpha_2, \alpha_3, \alpha_3) \notin \boldsymbol{U}_4 \\ \alpha_3 \geqslant \frac{1}{2}(1-\alpha_1-\alpha_2-\alpha_3) }} S\left(\mathcal{A}_{p_1 p_2 p_3}, p_3 \right) + \sum_{\substack{\boldsymbol{\alpha}_{2} \in A \cup B \\ \kappa \leqslant \alpha_3 < \min\left(\alpha_2, \frac{1}{2}(1-\alpha_1 -\alpha_2)\right) \\ \boldsymbol{\alpha}_{3} \notin \boldsymbol{G}_{3} \\ \text{either } \boldsymbol{\alpha}_{3} \notin \boldsymbol{U}_3 \text{ or } (\alpha_1, \alpha_2, \alpha_3, \alpha_3) \notin \boldsymbol{U}_4 \\ \alpha_3 < \frac{1}{2}(1-\alpha_1-\alpha_2-\alpha_3) }} S\left(\mathcal{A}_{p_1 p_2 p_3}, p_3 \right) \\
\nonumber =&\ \sum_{\substack{\boldsymbol{\alpha}_{2} \in A \cup B \\ \kappa \leqslant \alpha_3 < \min\left(\alpha_2, \frac{1}{2}(1-\alpha_1 -\alpha_2)\right) \\ \boldsymbol{\alpha}_{3} \notin \boldsymbol{G}_{3} \\ \text{either } \boldsymbol{\alpha}_{3} \notin \boldsymbol{U}_3 \text{ or } (\alpha_1, \alpha_2, \alpha_3, \alpha_3) \notin \boldsymbol{U}_4 \\ \alpha_3 \geqslant \frac{1}{2}(1-\alpha_1-\alpha_2-\alpha_3) }} S\left(\mathcal{A}_{p_1 p_2 p_3}, p_3 \right) + \sum_{\substack{\boldsymbol{\alpha}_{2} \in A \cup B \\ \kappa \leqslant \alpha_3 < \min\left(\alpha_2, \frac{1}{2}(1-\alpha_1 -\alpha_2)\right) \\ \boldsymbol{\alpha}_{3} \notin \boldsymbol{G}_{3} \\ \text{either } \boldsymbol{\alpha}_{3} \notin \boldsymbol{U}_3 \text{ or } (\alpha_1, \alpha_2, \alpha_3, \alpha_3) \notin \boldsymbol{U}_4 \\ \alpha_3 < \frac{1}{2}(1-\alpha_1-\alpha_2-\alpha_3) }} S\left(\mathcal{A}_{p_1 p_2 p_3}, \left(\frac{x}{p_1 p_2 p_3} \right)^{\frac{1}{2}} \right) \\
\nonumber & + \sum_{\substack{\boldsymbol{\alpha}_{2} \in A \cup B \\ \kappa \leqslant \alpha_3 < \min\left(\alpha_2, \frac{1}{2}(1-\alpha_1 -\alpha_2)\right) \\ \boldsymbol{\alpha}_{3} \notin \boldsymbol{G}_{3} \\ \text{either } \boldsymbol{\alpha}_{3} \notin \boldsymbol{U}_3 \text{ or } (\alpha_1, \alpha_2, \alpha_3, \alpha_3) \notin \boldsymbol{U}_4 \\ \alpha_3 < \alpha_4 < \frac{1}{2}(1-\alpha_1-\alpha_2-\alpha_3) \\ \boldsymbol{\alpha}_{4} \in \boldsymbol{G}_{4} }} S\left(\mathcal{A}_{p_1 p_2 p_3 p_4}, p_4 \right) + \sum_{\substack{\boldsymbol{\alpha}_{2} \in A \cup B \\ \kappa \leqslant \alpha_3 < \min\left(\alpha_2, \frac{1}{2}(1-\alpha_1 -\alpha_2)\right) \\ \boldsymbol{\alpha}_{3} \notin \boldsymbol{G}_{3} \\ \text{either } \boldsymbol{\alpha}_{3} \notin \boldsymbol{U}_3 \text{ or } (\alpha_1, \alpha_2, \alpha_3, \alpha_3) \notin \boldsymbol{U}_4 \\ \alpha_3 < \alpha_4 < \frac{1}{2}(1-\alpha_1-\alpha_2-\alpha_3) \\ \boldsymbol{\alpha}_{4} \notin \boldsymbol{G}_{4} \\ \alpha_4 \geqslant \frac{1}{2}(1-\alpha_1-\alpha_2-\alpha_3-\alpha_4) }} S\left(\mathcal{A}_{p_1 p_2 p_3 p_4}, p_4 \right) \\
\nonumber & + \sum_{\substack{\boldsymbol{\alpha}_{2} \in A \cup B \\ \kappa \leqslant \alpha_3 < \min\left(\alpha_2, \frac{1}{2}(1-\alpha_1 -\alpha_2)\right) \\ \boldsymbol{\alpha}_{3} \notin \boldsymbol{G}_{3} \\ \text{either } \boldsymbol{\alpha}_{3} \notin \boldsymbol{U}_3 \text{ or } (\alpha_1, \alpha_2, \alpha_3, \alpha_3) \notin \boldsymbol{U}_4 \\ \alpha_3 < \alpha_4 < \frac{1}{2}(1-\alpha_1-\alpha_2-\alpha_3) \\ \boldsymbol{\alpha}_{4} \notin \boldsymbol{G}_{4} \\ \alpha_4 < \frac{1}{2}(1-\alpha_1-\alpha_2-\alpha_3-\alpha_4) }} S\left(\mathcal{A}_{p_1 p_2 p_3 p_4}, p_4 \right) \\
\nonumber =&\ \sum_{\substack{\boldsymbol{\alpha}_{2} \in A \cup B \\ \kappa \leqslant \alpha_3 < \min\left(\alpha_2, \frac{1}{2}(1-\alpha_1 -\alpha_2)\right) \\ \boldsymbol{\alpha}_{3} \notin \boldsymbol{G}_{3} \\ \text{either } \boldsymbol{\alpha}_{3} \notin \boldsymbol{U}_3 \text{ or } (\alpha_1, \alpha_2, \alpha_3, \alpha_3) \notin \boldsymbol{U}_4 \\ \alpha_3 \geqslant \frac{1}{2}(1-\alpha_1-\alpha_2-\alpha_3) }} S\left(\mathcal{A}_{p_1 p_2 p_3}, p_3 \right) + \sum_{\substack{\boldsymbol{\alpha}_{2} \in A \cup B \\ \kappa \leqslant \alpha_3 < \min\left(\alpha_2, \frac{1}{2}(1-\alpha_1 -\alpha_2)\right) \\ \boldsymbol{\alpha}_{3} \notin \boldsymbol{G}_{3} \\ \text{either } \boldsymbol{\alpha}_{3} \notin \boldsymbol{U}_3 \text{ or } (\alpha_1, \alpha_2, \alpha_3, \alpha_3) \notin \boldsymbol{U}_4 \\ \alpha_3 < \frac{1}{2}(1-\alpha_1-\alpha_2-\alpha_3) }} S\left(\mathcal{A}_{p_1 p_2 p_3}, \left(\frac{x}{p_1 p_2 p_3} \right)^{\frac{1}{2}} \right) \\
\nonumber & + \sum_{\substack{\boldsymbol{\alpha}_{2} \in A \cup B \\ \kappa \leqslant \alpha_3 < \min\left(\alpha_2, \frac{1}{2}(1-\alpha_1 -\alpha_2)\right) \\ \boldsymbol{\alpha}_{3} \notin \boldsymbol{G}_{3} \\ \text{either } \boldsymbol{\alpha}_{3} \notin \boldsymbol{U}_3 \text{ or } (\alpha_1, \alpha_2, \alpha_3, \alpha_3) \notin \boldsymbol{U}_4 \\ \alpha_3 < \alpha_4 < \frac{1}{2}(1-\alpha_1-\alpha_2-\alpha_3) \\ \boldsymbol{\alpha}_{4} \in \boldsymbol{G}_{4} }} S\left(\mathcal{A}_{p_1 p_2 p_3 p_4}, p_4 \right) + \sum_{\substack{\boldsymbol{\alpha}_{2} \in A \cup B \\ \kappa \leqslant \alpha_3 < \min\left(\alpha_2, \frac{1}{2}(1-\alpha_1 -\alpha_2)\right) \\ \boldsymbol{\alpha}_{3} \notin \boldsymbol{G}_{3} \\ \text{either } \boldsymbol{\alpha}_{3} \notin \boldsymbol{U}_3 \text{ or } (\alpha_1, \alpha_2, \alpha_3, \alpha_3) \notin \boldsymbol{U}_4 \\ \alpha_3 < \alpha_4 < \frac{1}{2}(1-\alpha_1-\alpha_2-\alpha_3) \\ \boldsymbol{\alpha}_{4} \notin \boldsymbol{G}_{4} \\ \alpha_4 \geqslant \frac{1}{2}(1-\alpha_1-\alpha_2-\alpha_3-\alpha_4) }} S\left(\mathcal{A}_{p_1 p_2 p_3 p_4}, p_4 \right) \\
\nonumber & + \sum_{\substack{\boldsymbol{\alpha}_{2} \in A \cup B \\ \kappa \leqslant \alpha_3 < \min\left(\alpha_2, \frac{1}{2}(1-\alpha_1 -\alpha_2)\right) \\ \boldsymbol{\alpha}_{3} \notin \boldsymbol{G}_{3} \\ \text{either } \boldsymbol{\alpha}_{3} \notin \boldsymbol{U}_3 \text{ or } (\alpha_1, \alpha_2, \alpha_3, \alpha_3) \notin \boldsymbol{U}_4 \\ \alpha_3 < \alpha_4 < \frac{1}{2}(1-\alpha_1-\alpha_2-\alpha_3) \\ \boldsymbol{\alpha}_{4} \notin \boldsymbol{G}_{4} \\ \alpha_4 < \frac{1}{2}(1-\alpha_1-\alpha_2-\alpha_3-\alpha_4) }} S\left(\mathcal{A}_{p_1 p_2 p_3 p_4}, \left(\frac{x}{p_1 p_2 p_3 p_4} \right)^{\frac{1}{2}} \right) \\
\nonumber & + \sum_{\substack{\boldsymbol{\alpha}_{2} \in A \cup B \\ \kappa \leqslant \alpha_3 < \min\left(\alpha_2, \frac{1}{2}(1-\alpha_1 -\alpha_2)\right) \\ \boldsymbol{\alpha}_{3} \notin \boldsymbol{G}_{3} \\ \text{either } \boldsymbol{\alpha}_{3} \notin \boldsymbol{U}_3 \text{ or } (\alpha_1, \alpha_2, \alpha_3, \alpha_3) \notin \boldsymbol{U}_4 \\ \alpha_3 < \alpha_4 < \frac{1}{2}(1-\alpha_1-\alpha_2-\alpha_3) \\ \boldsymbol{\alpha}_{4} \notin \boldsymbol{G}_{4} \\ \alpha_4 < \alpha_5 < \frac{1}{2}(1-\alpha_1-\alpha_2-\alpha_3-\alpha_4) \\ \boldsymbol{\alpha}_{5} \in \boldsymbol{G}_{5} }} S\left(\mathcal{A}_{p_1 p_2 p_3 p_4 p_5}, p_5 \right) \\
\nonumber & + \sum_{\substack{\boldsymbol{\alpha}_{2} \in A \cup B \\ \kappa \leqslant \alpha_3 < \min\left(\alpha_2, \frac{1}{2}(1-\alpha_1 -\alpha_2)\right) \\ \boldsymbol{\alpha}_{3} \notin \boldsymbol{G}_{3} \\ \text{either } \boldsymbol{\alpha}_{3} \notin \boldsymbol{U}_3 \text{ or } (\alpha_1, \alpha_2, \alpha_3, \alpha_3) \notin \boldsymbol{U}_4 \\ \alpha_3 < \alpha_4 < \frac{1}{2}(1-\alpha_1-\alpha_2-\alpha_3) \\ \boldsymbol{\alpha}_{4} \notin \boldsymbol{G}_{4} \\ \alpha_4 < \alpha_5 < \frac{1}{2}(1-\alpha_1-\alpha_2-\alpha_3-\alpha_4) \\ \boldsymbol{\alpha}_{5} \notin \boldsymbol{G}_{5} }} S\left(\mathcal{A}_{p_1 p_2 p_3 p_4 p_5}, p_5 \right) \\
=&\ S_{1241} + S_{1242} + S_{1243} + S_{1244} + S_{1245} + S_{1246} + S_{1247}.
\end{align}
We can use Lemma~\ref{l41} to deal with $S_{1243}$ and $S_{1246}$. Therefore, we can subtract them from the loss of $S_{1233}$. By (17)--(22), we get the total loss when $\frac{1}{2} \leqslant \theta < \frac{25}{49}$ of
\begin{align}
\nonumber & \left(\int_{(t_1, t_2, t_3) \in L_{101}} \frac{\omega\left(\frac{t_2}{t_1} \right) \omega\left(\frac{1-t_1-t_2-t_3}{t_3} \right)}{t_1^2 t_3^2} d t_3 d t_2 d t_1 \right) \\
\nonumber + & \left(\frac{1}{\kappa} \int_{\kappa}^{\frac{1}{2}} \int_{\kappa}^{\frac{1}{2}} \int_{\kappa}^{\frac{1}{2}} \mathbbm{1}_{(t_1, t_2, t_3) \in \boldsymbol{D}_1 } \frac{\omega\left(\frac{1-t_1-t_2-t_3}{\kappa} \right)}{t_1 t_2 t_3} d t_3 d t_2 d t_1 \right) \\
\nonumber + & \left(\frac{1}{\kappa} \int_{\kappa}^{\frac{1}{2}} \int_{\kappa}^{\frac{1}{2}} \int_{\kappa}^{\frac{1}{2}} \int_{\kappa}^{\frac{1}{2}} \mathbbm{1}_{(t_1, t_2, t_3, t_4) \in \boldsymbol{D}_2 } \frac{\omega\left(\frac{1-t_1-t_2-t_3-t_4}{\kappa} \right)}{t_1 t_2 t_3 t_4} d t_4 d t_3 d t_2 d t_1 \right) \\
\nonumber + & \left(\int_{(t_1, t_2, t_3) \in L_{104}} \frac{\omega\left(\frac{1-t_1-t_2-t_3}{t_3} \right)}{t_1 t_2 t_3^2} d t_3 d t_2 d t_1 \right) \\
\nonumber + & \left(\int_{(t_1, t_2, t_3) \in L_{105}} \frac{\omega\left(\frac{1-t_1-t_2-t_3}{t_3} \right)}{t_1 t_2 t_3^2} d t_3 d t_2 d t_1 \right) \\
\nonumber - & \left(\int_{(t_1, t_2, t_3, t_4) \in L_{106}} \frac{\omega\left(\frac{1-t_1-t_2-t_3-t_4}{t_4} \right)}{t_1 t_2 t_3 t_4^2} d t_4 d t_3 d t_2 d t_1 \right) \\
\nonumber - & \left(\int_{(t_1, t_2, t_3, t_4, t_5) \in L_{107}} \frac{\omega\left(\frac{1-t_1-t_2-t_3-t_4-t_5}{t_5} \right)}{t_1 t_2 t_3 t_4 t_5^2} d t_5 d t_4 d t_3 d t_2 d t_1 \right) \\
\nonumber + & \left(\int_{(t_1, t_2, t_3, t_4, t_5) \in L_{108}} \frac{\omega\left(\frac{1-t_1-t_2-t_3-t_4-t_5}{t_5} \right)}{t_1 t_2 t_3 t_4 t_5^2} d t_5 d t_4 d t_3 d t_2 d t_1 \right) \\
=&\ E_{101}(\theta) + E_{102}(\theta) + E_{103}(\theta) + E_{104}(\theta) + E_{105}(\theta) - E_{106}(\theta) - E_{107}(\theta) + E_{108}(\theta),
\end{align}
where
\begin{align}
\nonumber L_{101}(\boldsymbol{\alpha}_{3}) :=&\ \left\{ \kappa \leqslant \alpha_1 \leqslant \frac{\theta}{2},\ 1 - \theta - \alpha_1 \leqslant \alpha_2 \leqslant \theta - \alpha_1, \right. \\
\nonumber & \left. \quad \kappa \leqslant \alpha_3 \leqslant \min\left(\alpha_2, \frac{1}{2}(1-\alpha_1 -\alpha_2) \right),\ \boldsymbol{\alpha}_{3} \notin \boldsymbol{G}_{3} \right\}, \\
\nonumber L_{104}(\boldsymbol{\alpha}_{3}) :=&\ \left\{ \boldsymbol{\alpha}_{2} \in C,\ \kappa \leqslant \alpha_3 < \min\left(\alpha_2, \frac{1}{2}(1-\alpha_1 -\alpha_2) \right),\ \boldsymbol{\alpha}_{3} \notin \boldsymbol{G}_{3}, \right. \\
\nonumber & \left. \quad \kappa \leqslant \alpha_1 < \frac{1}{2},\ \kappa \leqslant \alpha_2 < \min\left(\alpha_1, \frac{1}{2}(1-\alpha_1) \right) \right\}, \\
\nonumber L_{105}(\boldsymbol{\alpha}_{3}) :=&\ \left\{ \boldsymbol{\alpha}_{2} \in A \cup B,\ \kappa \leqslant \alpha_3 < \min\left(\alpha_2, \frac{1}{2}(1-\alpha_1 -\alpha_2) \right),\ \boldsymbol{\alpha}_{3} \notin \boldsymbol{G}_{3}, \right. \\
\nonumber & \quad \text{either } \boldsymbol{\alpha}_{3} \notin \boldsymbol{U}_3 \text{ or } (\alpha_1, \alpha_2, \alpha_3, \alpha_3) \notin \boldsymbol{U}_4, \\
\nonumber & \left. \quad \kappa \leqslant \alpha_1 < \frac{1}{2},\ \kappa \leqslant \alpha_2 < \min\left(\alpha_1, \frac{1}{2}(1-\alpha_1) \right) \right\}, \\
\nonumber L_{106}(\boldsymbol{\alpha}_{4}) :=&\ \left\{ \boldsymbol{\alpha}_{2} \in A \cup B,\ \kappa \leqslant \alpha_3 < \min\left(\alpha_2, \frac{1}{2}(1-\alpha_1 -\alpha_2) \right),\ \boldsymbol{\alpha}_{3} \notin \boldsymbol{G}_{3}, \right. \\
\nonumber & \quad \text{either } \boldsymbol{\alpha}_{3} \notin \boldsymbol{U}_3 \text{ or } (\alpha_1, \alpha_2, \alpha_3, \alpha_3) \notin \boldsymbol{U}_4, \\
\nonumber & \quad \alpha_3 < \alpha_4 < \frac{1}{2}(1-\alpha_1-\alpha_2-\alpha_3),\ \boldsymbol{\alpha}_{4} \in \boldsymbol{G}_{4}, \\
\nonumber & \left. \quad \kappa \leqslant \alpha_1 < \frac{1}{2},\ \kappa \leqslant \alpha_2 < \min\left(\alpha_1, \frac{1}{2}(1-\alpha_1) \right) \right\}, \\
\nonumber L_{107}(\boldsymbol{\alpha}_{5}) :=&\ \left\{ \boldsymbol{\alpha}_{2} \in A \cup B,\ \kappa \leqslant \alpha_3 < \min\left(\alpha_2, \frac{1}{2}(1-\alpha_1 -\alpha_2) \right),\ \boldsymbol{\alpha}_{3} \notin \boldsymbol{G}_{3}, \right. \\
\nonumber & \quad \text{either } \boldsymbol{\alpha}_{3} \notin \boldsymbol{U}_3 \text{ or } (\alpha_1, \alpha_2, \alpha_3, \alpha_3) \notin \boldsymbol{U}_4, \\
\nonumber & \quad \alpha_3 < \alpha_4 < \frac{1}{2}(1-\alpha_1-\alpha_2-\alpha_3),\ \boldsymbol{\alpha}_{4} \notin \boldsymbol{G}_{4}, \\
\nonumber & \quad \alpha_4 < \alpha_5 < \frac{1}{2}(1-\alpha_1-\alpha_2-\alpha_3-\alpha_4),\ \boldsymbol{\alpha}_{5} \in \boldsymbol{G}_{5}, \\
\nonumber & \left. \quad \kappa \leqslant \alpha_1 < \frac{1}{2},\ \kappa \leqslant \alpha_2 < \min\left(\alpha_1, \frac{1}{2}(1-\alpha_1) \right) \right\}, \\
\nonumber L_{108}(\boldsymbol{\alpha}_{5}) :=&\ \left\{ \boldsymbol{\alpha}_{2} \in A \cup B,\ \kappa \leqslant \alpha_3 < \min\left(\alpha_2, \frac{1}{2}(1-\alpha_1 -\alpha_2) \right),\ \boldsymbol{\alpha}_{3} \notin \boldsymbol{G}_{3}, \right. \\
\nonumber & \quad \boldsymbol{\alpha}_{3} \in \boldsymbol{U}_3,\ (\alpha_1, \alpha_2, \alpha_3, \alpha_3) \in \boldsymbol{U}_4, \\
\nonumber & \quad \kappa \leqslant \alpha_4 < \min\left(\alpha_3, \frac{1}{2}(1-\alpha_1 -\alpha_2-\alpha_3) \right),\ \boldsymbol{\alpha}_{4} \notin \boldsymbol{G}_{4}, \\
\nonumber & \quad \kappa \leqslant \alpha_5 < \min\left(\alpha_4, \frac{1}{2}(1-\alpha_1 -\alpha_2-\alpha_3-\alpha_4) \right),\ \boldsymbol{\alpha}_{5} \notin \boldsymbol{G}_{5}, \\
\nonumber & \left. \quad \kappa \leqslant \alpha_1 < \frac{1}{2},\ \kappa \leqslant \alpha_2 < \min\left(\alpha_1, \frac{1}{2}(1-\alpha_1) \right) \right\}.
\end{align}

Finally, for $\frac{1}{2} \leqslant \theta < \frac{25}{49}$ we have
\begin{equation}
C(\theta) \leqslant C_{1}(\theta) := 1 + \left(E_{101}(\theta) + E_{102}(\theta) + E_{103}(\theta) + E_{104}(\theta) + E_{105}(\theta) - E_{106}(\theta) - E_{107}(\theta) + E_{108}(\theta) \right).
\end{equation}

\begin{center}
\begin{tabular}{|c|c|}
\hline \boldmath{$\theta$} & \boldmath{$C_{1}(\theta)$}\\
\hline $0.501$ & $1.00001$ \\
\hline $0.503$ & $1.00124$ \\
\hline $0.505$ & $1.00633$ \\
\hline $0.508$ & $1.01001$ \\
\hline $0.51$ & $1.015$ \\
\hline $\frac{25}{49}$ & $1.02079$ \\
\hline
\end{tabular}
\end{center}

\subsection{Case 2. $\frac{25}{49} \leqslant \theta < \frac{21}{41}$}
In this subrange, we note that $4 - 7 \theta \leqslant \frac{3}{7}$ and the upper bound for $\alpha_1 \in \boldsymbol{R}$ becomes $4 - 7 \theta$. Since we cannot discard any parts of the sum
\begin{equation}
\sum_{\boldsymbol{\alpha}_{2} \in C } S\left(\mathcal{A}_{p_1 p_2}, p_2 \right),
\end{equation}
we must split the sum
$$
\sum_{\kappa \leqslant \alpha_1 \leqslant \frac{3}{7}} S\left(\mathcal{A}_{p_1}, p_1 \right)
$$
into two sums. Now we start our decomposition by using Buchstab's identity:
\begin{align}
\nonumber S\left(\mathcal{A}, (2x)^{\frac{1}{2}} \right) =&\ S\left(\mathcal{A}, x^{\kappa} \right) - \sum_{\kappa \leqslant \alpha_1 < \frac{1}{2}} S\left(\mathcal{A}_{p_1}, p_1 \right) \\
\nonumber =&\ S\left(\mathcal{A}, x^{\kappa} \right) - \sum_{\kappa \leqslant \alpha_1 < 4 - 7 \theta} S\left(\mathcal{A}_{p_1}, p_1 \right) - \sum_{4 - 7 \theta \leqslant \alpha_1 \leqslant \frac{3}{7}} S\left(\mathcal{A}_{p_1}, p_1 \right) \\
\nonumber & - \sum_{\frac{3}{7} < \alpha_1 < 1 - \theta} S\left(\mathcal{A}_{p_1}, p_1 \right) - \sum_{1 - \theta \leqslant \alpha_1 < \frac{1}{2}} S\left(\mathcal{A}_{p_1}, p_1 \right) \\
=&\ S_{21} - S_{22} - S_{23} - S_{24} - S_{25}.
\end{align}
Just as $S_{11}$ and $S_{13}$ in \textbf{Case 1}, we can give asymptotic formulas for $S_{21}$ by Lemma~\ref{l43} and for $S_{24}$ by Lemma~\ref{l41}. For $S_{25}$, we can use Lemma~\ref{l51} to give an upper bound similar to that for $S_{14}$. We can use the same process as in the estimation $S_{12}$ to show a similar upper bound for $S_{22}$, with a loss consisting of 7 integrals ($I_1$, $I_2$, 3D for $\boldsymbol{\alpha}_{2} \in C$, and 3D$^{+}$, 4D$^{-}$, 5D$^{-}$ and 5D$^{+}$ for $\boldsymbol{\alpha}_{2} \in A \cup B$). Now we only need to deal with the remaining $S_{23}$. In this sum Lemma~\ref{l52} is not applicable, but we can use Lemma~\ref{l53} to deal with some parts satisfying
$$
\max\left(\frac{19 \theta - 7}{7},\ \frac{50 \theta - 19}{17} \right) \leqslant \alpha_1 \leqslant \frac{3}{7}, \quad \text{or} \quad \frac{19 \theta - 7}{7} \leqslant \alpha_1 \leqslant \frac{3}{7} \quad \text{in this case}.
$$
That is, we use Buchstab's identity on $S_{23}$ twice again, giving the following expression.
\begin{align}
\nonumber S_{23} =&\ \sum_{4 - 7 \theta \leqslant \alpha_1 \leqslant \frac{3}{7}} S\left(\mathcal{A}_{p_1}, p_1 \right) \\
\nonumber =&\ \sum_{\substack{4 - 7 \theta \leqslant \alpha_1 \leqslant \frac{3}{7} \\ \theta \notin \left[\frac{19 \theta - 7}{7}, \frac{3}{7}\right]}} S\left(\mathcal{A}_{p_1}, p_1 \right) + \sum_{\substack{4 - 7 \theta \leqslant \alpha_1 \leqslant \frac{3}{7} \\ \theta \in \left[\frac{19 \theta - 7}{7}, \frac{3}{7}\right]}} S\left(\mathcal{A}_{p_1}, p_1 \right) \\
\nonumber =&\ \sum_{\substack{4 - 7 \theta \leqslant \alpha_1 \leqslant \frac{3}{7} \\ \theta \notin \left[\frac{19 \theta - 7}{7}, \frac{3}{7}\right]}} S\left(\mathcal{A}_{p_1}, p_1 \right) + \sum_{\substack{4 - 7 \theta \leqslant \alpha_1 \leqslant \frac{3}{7} \\ \theta \in \left[\frac{19 \theta - 7}{7}, \frac{3}{7}\right]}} S\left(\mathcal{A}_{p_1}, x^{\kappa} \right) \\ 
\nonumber & - \sum_{\substack{4 - 7 \theta \leqslant \alpha_1 \leqslant \frac{3}{7} \\ \theta \in \left[\frac{19 \theta - 7}{7}, \frac{3}{7}\right] \\ \boldsymbol{\alpha}_{2} \in A \cup B }} S\left(\mathcal{A}_{p_1 p_2}, x^{\kappa} \right) - \sum_{\substack{4 - 7 \theta \leqslant \alpha_1 \leqslant \frac{3}{7} \\ \theta \in \left[\frac{19 \theta - 7}{7}, \frac{3}{7}\right] \\ \boldsymbol{\alpha}_{2} \in C }} S\left(\mathcal{A}_{p_1 p_2}, x^{\kappa} \right) \\
\nonumber & + \sum_{\substack{4 - 7 \theta \leqslant \alpha_1 \leqslant \frac{3}{7} \\ \theta \in \left[\frac{19 \theta - 7}{7}, \frac{3}{7}\right] \\ \boldsymbol{\alpha}_{2} \in A \cup B \cup C \\ \kappa \leqslant \alpha_3 < \min\left(\alpha_2, \frac{1}{2}(1-\alpha_1 -\alpha_2)\right) \\ \boldsymbol{\alpha}_{3} \in \boldsymbol{G}_{3} }} S\left(\mathcal{A}_{p_1 p_2 p_3}, p_3 \right) + \sum_{\substack{4 - 7 \theta \leqslant \alpha_1 \leqslant \frac{3}{7} \\ \theta \in \left[\frac{19 \theta - 7}{7}, \frac{3}{7}\right] \\ \boldsymbol{\alpha}_{2} \in A \cup B \cup C \\ \kappa \leqslant \alpha_3 < \min\left(\alpha_2, \frac{1}{2}(1-\alpha_1 -\alpha_2)\right) \\ \boldsymbol{\alpha}_{3} \notin \boldsymbol{G}_{3} }} S\left(\mathcal{A}_{p_1 p_2 p_3}, p_3 \right) \\
=&\ S_{231} + S_{232} - S_{233} - S_{234} + S_{235} + S_{236}.
\end{align}
We give asymptotic formulas for $S_{232}$ (Lemma~\ref{l43}), $S_{233}$ (Lemma~\ref{l44}) and $S_{235}$ (Lemma~\ref{l41}). After checking the conditions we find that $\boldsymbol{R}_0$ covers parts of $C$ satisfy $\alpha_1 \in \left[\frac{19 \theta - 7}{7}, \frac{3}{7}\right]$. We use Lemma~\ref{l53} to give an upper bound for $S_{234}$ and discard the whole of $S_{231}$ and $S_{236}$, leading to 4 loss integrals ($S_{231} \to$ 1D for $4 - 7 \theta \leqslant \alpha_1 < \frac{19 \theta - 7}{7}$, $S_{234} \to$ $I_3$ and $I_4$, and $S_{236} \to$ 3D for $\frac{19 \theta - 7}{7} \leqslant \alpha_1 \leqslant \frac{3}{7}$). Together, the total loss when $\frac{25}{49} \leqslant \theta < \frac{21}{41}$ can be bounded by
\begin{align}
\nonumber & \left(\int_{(t_1, t_2, t_3) \in L_{201}} \frac{\omega\left(\frac{t_2}{t_1} \right) \omega\left(\frac{1-t_1-t_2-t_3}{t_3} \right)}{t_1^2 t_3^2} d t_3 d t_2 d t_1 \right) \\
\nonumber + & \left(\frac{1}{\kappa} \int_{\kappa}^{\frac{1}{2}} \int_{\kappa}^{\frac{1}{2}} \int_{\kappa}^{\frac{1}{2}} \mathbbm{1}_{(t_1, t_2, t_3) \in \boldsymbol{D}_1 } \frac{\omega\left(\frac{1-t_1-t_2-t_3}{\kappa} \right)}{t_1 t_2 t_3} d t_3 d t_2 d t_1 \right) \\
\nonumber + & \left(\frac{1}{\kappa} \int_{\kappa}^{\frac{1}{2}} \int_{\kappa}^{\frac{1}{2}} \int_{\kappa}^{\frac{1}{2}} \int_{\kappa}^{\frac{1}{2}} \mathbbm{1}_{(t_1, t_2, t_3, t_4) \in \boldsymbol{D}_2 } \frac{\omega\left(\frac{1-t_1-t_2-t_3-t_4}{\kappa} \right)}{t_1 t_2 t_3 t_4} d t_4 d t_3 d t_2 d t_1 \right) \\
\nonumber + & \left(\int_{(t_1, t_2, t_3) \in L_{204}} \frac{\omega\left(\frac{1-t_1-t_2-t_3}{t_3} \right)}{t_1 t_2 t_3^2} d t_3 d t_2 d t_1 \right) \\
\nonumber + & \left(\int_{(t_1, t_2, t_3) \in L_{205}} \frac{\omega\left(\frac{1-t_1-t_2-t_3}{t_3} \right)}{t_1 t_2 t_3^2} d t_3 d t_2 d t_1 \right) \\
\nonumber - & \left(\int_{(t_1, t_2, t_3, t_4) \in L_{206}} \frac{\omega\left(\frac{1-t_1-t_2-t_3-t_4}{t_4} \right)}{t_1 t_2 t_3 t_4^2} d t_4 d t_3 d t_2 d t_1 \right) \\
\nonumber - & \left(\int_{(t_1, t_2, t_3, t_4, t_5) \in L_{207}} \frac{\omega\left(\frac{1-t_1-t_2-t_3-t_4-t_5}{t_5} \right)}{t_1 t_2 t_3 t_4 t_5^2} d t_5 d t_4 d t_3 d t_2 d t_1 \right) \\
\nonumber + & \left(\int_{(t_1, t_2, t_3, t_4, t_5) \in L_{208}} \frac{\omega\left(\frac{1-t_1-t_2-t_3-t_4-t_5}{t_5} \right)}{t_1 t_2 t_3 t_4 t_5^2} d t_5 d t_4 d t_3 d t_2 d t_1 \right) \\
\nonumber + & \left(\int_{(t_1) \in L_{209}} \frac{\omega\left(\frac{1-t_1}{t_1} \right)}{t_1^2} d t_1 \right) \\
\nonumber + & \left(\frac{1}{\kappa} \int_{\kappa}^{\frac{1}{2}} \int_{\kappa}^{\frac{1}{2}} \int_{\kappa}^{\frac{1}{2}} \mathbbm{1}_{(t_1, t_2, t_3) \in \boldsymbol{D}_3 } \frac{\omega\left(\frac{1-t_1-t_2-t_3}{\kappa} \right)}{t_1 t_2 t_3} d t_3 d t_2 d t_1 \right) \\
\nonumber + & \left(\frac{1}{\kappa} \int_{\kappa}^{\frac{1}{2}} \int_{\kappa}^{\frac{1}{2}} \int_{\kappa}^{\frac{1}{2}} \int_{\kappa}^{\frac{1}{2}} \mathbbm{1}_{(t_1, t_2, t_3, t_4) \in \boldsymbol{D}_4 } \frac{\omega\left(\frac{1-t_1-t_2-t_3-t_4}{\kappa} \right)}{t_1 t_2 t_3 t_4} d t_4 d t_3 d t_2 d t_1 \right) \\
\nonumber + & \left(\int_{(t_1, t_2, t_3) \in L_{212}} \frac{\omega\left(\frac{1-t_1-t_2-t_3}{t_3} \right)}{t_1 t_2 t_3^2} d t_3 d t_2 d t_1 \right) \\
\nonumber =&\ E_{201}(\theta) + E_{202}(\theta) + E_{203}(\theta) + E_{204}(\theta) + E_{205}(\theta) - E_{206}(\theta) - E_{207}(\theta) \\ 
& + E_{208}(\theta) + E_{209}(\theta) + E_{210}(\theta) + E_{211}(\theta) + E_{212}(\theta),
\end{align}
where
\begin{align}
\nonumber L_{201}(\boldsymbol{\alpha}_{3}) :=&\ \left\{ \kappa \leqslant \alpha_1 \leqslant \frac{\theta}{2},\ 1 - \theta - \alpha_1 \leqslant \alpha_2 \leqslant \theta - \alpha_1, \right. \\
\nonumber & \left. \quad \kappa \leqslant \alpha_3 \leqslant \min\left(\alpha_2, \frac{1}{2}(1-\alpha_1 -\alpha_2) \right),\ \boldsymbol{\alpha}_{3} \notin \boldsymbol{G}_{3} \right\}, \\
\nonumber L_{204}(\boldsymbol{\alpha}_{3}) :=&\ \left\{ \alpha_1 < 4 - 7 \theta,\ \boldsymbol{\alpha}_{2} \in C,\ \kappa \leqslant \alpha_3 < \min\left(\alpha_2, \frac{1}{2}(1-\alpha_1 -\alpha_2) \right),\ \boldsymbol{\alpha}_{3} \notin \boldsymbol{G}_{3}, \right. \\
\nonumber & \left. \quad \kappa \leqslant \alpha_1 < \frac{1}{2},\ \kappa \leqslant \alpha_2 < \min\left(\alpha_1, \frac{1}{2}(1-\alpha_1) \right) \right\}, \\
\nonumber L_{205}(\boldsymbol{\alpha}_{3}) :=&\ \left\{\alpha_1 < 4 - 7 \theta,\ \boldsymbol{\alpha}_{2} \in A \cup B,\ \kappa \leqslant \alpha_3 < \min\left(\alpha_2, \frac{1}{2}(1-\alpha_1 -\alpha_2) \right),\ \boldsymbol{\alpha}_{3} \notin \boldsymbol{G}_{3}, \right. \\
\nonumber & \quad \text{either } \boldsymbol{\alpha}_{3} \notin \boldsymbol{U}_3 \text{ or } (\alpha_1, \alpha_2, \alpha_3, \alpha_3) \notin \boldsymbol{U}_4, \\
\nonumber & \left. \quad \kappa \leqslant \alpha_1 < \frac{1}{2},\ \kappa \leqslant \alpha_2 < \min\left(\alpha_1, \frac{1}{2}(1-\alpha_1) \right) \right\}, \\
\nonumber L_{206}(\boldsymbol{\alpha}_{4}) :=&\ \left\{\alpha_1 < 4 - 7 \theta,\ \boldsymbol{\alpha}_{2} \in A \cup B,\ \kappa \leqslant \alpha_3 < \min\left(\alpha_2, \frac{1}{2}(1-\alpha_1 -\alpha_2) \right),\ \boldsymbol{\alpha}_{3} \notin \boldsymbol{G}_{3}, \right. \\
\nonumber & \quad \text{either } \boldsymbol{\alpha}_{3} \notin \boldsymbol{U}_3 \text{ or } (\alpha_1, \alpha_2, \alpha_3, \alpha_3) \notin \boldsymbol{U}_4, \\
\nonumber & \quad \alpha_3 < \alpha_4 < \frac{1}{2}(1-\alpha_1-\alpha_2-\alpha_3),\ \boldsymbol{\alpha}_{4} \in \boldsymbol{G}_{4}, \\
\nonumber & \left. \quad \kappa \leqslant \alpha_1 < \frac{1}{2},\ \kappa \leqslant \alpha_2 < \min\left(\alpha_1, \frac{1}{2}(1-\alpha_1) \right) \right\}, \\
\nonumber L_{207}(\boldsymbol{\alpha}_{5}) :=&\ \left\{\alpha_1 < 4 - 7 \theta,\ \boldsymbol{\alpha}_{2} \in A \cup B,\ \kappa \leqslant \alpha_3 < \min\left(\alpha_2, \frac{1}{2}(1-\alpha_1 -\alpha_2) \right),\ \boldsymbol{\alpha}_{3} \notin \boldsymbol{G}_{3}, \right. \\
\nonumber & \quad \text{either } \boldsymbol{\alpha}_{3} \notin \boldsymbol{U}_3 \text{ or } (\alpha_1, \alpha_2, \alpha_3, \alpha_3) \notin \boldsymbol{U}_4, \\
\nonumber & \quad \alpha_3 < \alpha_4 < \frac{1}{2}(1-\alpha_1-\alpha_2-\alpha_3),\ \boldsymbol{\alpha}_{4} \notin \boldsymbol{G}_{4}, \\
\nonumber & \quad \alpha_4 < \alpha_5 < \frac{1}{2}(1-\alpha_1-\alpha_2-\alpha_3-\alpha_4),\ \boldsymbol{\alpha}_{5} \in \boldsymbol{G}_{5}, \\
\nonumber & \left. \quad \kappa \leqslant \alpha_1 < \frac{1}{2},\ \kappa \leqslant \alpha_2 < \min\left(\alpha_1, \frac{1}{2}(1-\alpha_1) \right) \right\}, \\
\nonumber L_{208}(\boldsymbol{\alpha}_{5}) :=&\ \left\{\alpha_1 < 4 - 7 \theta,\ \boldsymbol{\alpha}_{2} \in A \cup B,\ \kappa \leqslant \alpha_3 < \min\left(\alpha_2, \frac{1}{2}(1-\alpha_1 -\alpha_2) \right),\ \boldsymbol{\alpha}_{3} \notin \boldsymbol{G}_{3}, \right. \\
\nonumber & \quad \boldsymbol{\alpha}_{3} \in \boldsymbol{U}_3,\ (\alpha_1, \alpha_2, \alpha_3, \alpha_3) \in \boldsymbol{U}_4, \\
\nonumber & \quad \kappa \leqslant \alpha_4 < \min\left(\alpha_3, \frac{1}{2}(1-\alpha_1 -\alpha_2-\alpha_3) \right),\ \boldsymbol{\alpha}_{4} \notin \boldsymbol{G}_{4}, \\
\nonumber & \quad \kappa \leqslant \alpha_5 < \min\left(\alpha_4, \frac{1}{2}(1-\alpha_1 -\alpha_2-\alpha_3-\alpha_4) \right),\ \boldsymbol{\alpha}_{5} \notin \boldsymbol{G}_{5}, \\
\nonumber & \left. \quad \kappa \leqslant \alpha_1 < \frac{1}{2},\ \kappa \leqslant \alpha_2 < \min\left(\alpha_1, \frac{1}{2}(1-\alpha_1) \right) \right\}, \\
\nonumber L_{209}(\alpha_1) :=&\ \left\{ 4 - 7 \theta \leqslant \alpha_1 \leqslant \frac{3}{7},\ \alpha_1 \notin \left[\frac{19 \theta - 7}{7},\ \frac{3}{7}\right] \right\}, \\
\nonumber L_{212}(\boldsymbol{\alpha}_{3}) :=&\ \left\{  4 - 7 \theta \leqslant \alpha_1 \leqslant \frac{3}{7},\ \alpha_1 \in \left[\frac{19 \theta - 7}{7},\ \frac{3}{7}\right],\ \boldsymbol{\alpha}_{2} \notin \boldsymbol{G}_{2}, \right. \\
\nonumber & \quad \kappa \leqslant \alpha_3 < \min\left(\alpha_2, \frac{1}{2}(1-\alpha_1 -\alpha_2) \right),\ \boldsymbol{\alpha}_{3} \notin \boldsymbol{G}_{3}, \\
\nonumber & \left. \quad \kappa \leqslant \alpha_1 < \frac{1}{2},\ \kappa \leqslant \alpha_2 < \min\left(\alpha_1, \frac{1}{2}(1-\alpha_1) \right) \right\}.
\end{align}

Finally, for $\frac{25}{49} \leqslant \theta < \frac{21}{41}$ we have
\begin{align}
\nonumber C(\theta) \leqslant C_{2}(\theta) :=&\ 1 + \left(E_{201}(\theta) + E_{202}(\theta) + E_{203}(\theta) + E_{204}(\theta) + E_{205}(\theta) - E_{206}(\theta) \right. \\
& \left. \qquad - E_{207}(\theta) + E_{208}(\theta) + E_{209}(\theta) + E_{210}(\theta) + E_{211}(\theta) + E_{212}(\theta) \right).
\end{align}

\begin{center}
\begin{tabular}{|c|c|c|c|c|c|c|c|}
\hline \boldmath{$\theta$} & \boldmath{$C_{2}(\theta)$} & \boldmath{$\theta$} & \boldmath{$C_{2}(\theta)$} & \boldmath{$\theta$} & \boldmath{$C_{2}(\theta)$} & \boldmath{$\theta$} & \boldmath{$C_{2}(\theta)$} \\
\hline $0.5103$ & $1.02214$ & $0.5109$ & $1.02349$ & $0.5114$ & $1.02441$ & $0.5119$ & $1.02757$ \\
\hline $0.5104$ & $1.02236$ & $0.5111$ & $1.02365$ & $0.5115$ & $1.02623$ & $0.512$ & $1.03146$ \\
\hline $0.5105$ & $1.0227$ & $0.5112$ & $1.02385$ & $0.51155$ & $1.02639$ & $0.5121$ & $1.0315$ \\
\hline $0.5107$ & $1.02297$ & $0.5113$ & $1.02414$ & $0.5116$ & $1.02656$ & $\frac{21}{41}$ & $1.03169$ \\
\hline
\end{tabular}
\end{center}

\subsection{Case 3. $\frac{21}{41} \leqslant \theta < \frac{16}{31}$}
In this subrange, we have $\frac{31 \theta - 15}{3} \geqslant \tau$, and the lower bound for $\alpha_1 \in \boldsymbol{R}$ becomes $\frac{31 \theta - 15}{3}$. Because the lines $\alpha_1 = \alpha_2$ and $\alpha_1 + 4 \alpha_2 = 3 - 3 \theta$ intersect in $(\alpha_1, \alpha_2) = (\tau, \tau)$, we have $\boldsymbol{\alpha}_{2} \in A \cup B$ if $\alpha_2 < \alpha_1 < \tau$. Thus, we start our decomposition by using Buchstab's identity and split the second sum in the following way:
\begin{align}
\nonumber S\left(\mathcal{A}, (2x)^{\frac{1}{2}} \right) =&\ S\left(\mathcal{A}, x^{\kappa} \right) - \sum_{\kappa \leqslant \alpha_1 < \frac{1}{2}} S\left(\mathcal{A}_{p_1}, p_1 \right) \\
\nonumber =&\ S\left(\mathcal{A}, x^{\kappa} \right) - \sum_{\alpha_1 \in \left[\kappa, \tau \right) \cup \left(\frac{31 \theta - 15}{3}, 4 - 7 \theta \right)} S\left(\mathcal{A}_{p_1}, p_1 \right) - \sum_{\alpha_1 \in \left[\tau, \frac{31 \theta - 15}{3} \right] \cup \left[4 - 7 \theta, \frac{3}{7} \right]} S\left(\mathcal{A}_{p_1}, p_1 \right) \\
\nonumber & - \sum_{\frac{3}{7} < \alpha_1 < 1 - \theta} S\left(\mathcal{A}_{p_1}, p_1 \right) - \sum_{1 - \theta \leqslant \alpha_1 < \frac{1}{2}} S\left(\mathcal{A}_{p_1}, p_1 \right) \\
=&\ S_{31} - S_{32} - S_{33} - S_{34} - S_{35}.
\end{align}
By a decomposing process similar to our decomposition in \textbf{Case 2}, we can get 7 loss integrals for $S_{22}$, 4 loss integrals for $S_{23}$ and a loss integral for $S_{25}$. Thus, the total loss when $\frac{21}{41} \leqslant \theta < \frac{16}{31}$ can be bounded by
\begin{align}
\nonumber & \left(\int_{(t_1, t_2, t_3) \in L_{301}} \frac{\omega\left(\frac{t_2}{t_1} \right) \omega\left(\frac{1-t_1-t_2-t_3}{t_3} \right)}{t_1^2 t_3^2} d t_3 d t_2 d t_1 \right) \\
\nonumber + & \left(\frac{1}{\kappa} \int_{\kappa}^{\frac{1}{2}} \int_{\kappa}^{\frac{1}{2}} \int_{\kappa}^{\frac{1}{2}} \mathbbm{1}_{(t_1, t_2, t_3) \in \boldsymbol{D}_1 } \frac{\omega\left(\frac{1-t_1-t_2-t_3}{\kappa} \right)}{t_1 t_2 t_3} d t_3 d t_2 d t_1 \right) \\
\nonumber + & \left(\frac{1}{\kappa} \int_{\kappa}^{\frac{1}{2}} \int_{\kappa}^{\frac{1}{2}} \int_{\kappa}^{\frac{1}{2}} \int_{\kappa}^{\frac{1}{2}} \mathbbm{1}_{(t_1, t_2, t_3, t_4) \in \boldsymbol{D}_2 } \frac{\omega\left(\frac{1-t_1-t_2-t_3-t_4}{\kappa} \right)}{t_1 t_2 t_3 t_4} d t_4 d t_3 d t_2 d t_1 \right) \\
\nonumber + & \left(\int_{(t_1, t_2, t_3) \in L_{304}} \frac{\omega\left(\frac{1-t_1-t_2-t_3}{t_3} \right)}{t_1 t_2 t_3^2} d t_3 d t_2 d t_1 \right) \\
\nonumber + & \left(\int_{(t_1, t_2, t_3) \in L_{305}} \frac{\omega\left(\frac{1-t_1-t_2-t_3}{t_3} \right)}{t_1 t_2 t_3^2} d t_3 d t_2 d t_1 \right) \\
\nonumber - & \left(\int_{(t_1, t_2, t_3, t_4) \in L_{306}} \frac{\omega\left(\frac{1-t_1-t_2-t_3-t_4}{t_4} \right)}{t_1 t_2 t_3 t_4^2} d t_4 d t_3 d t_2 d t_1 \right) \\
\nonumber - & \left(\int_{(t_1, t_2, t_3, t_4, t_5) \in L_{307}} \frac{\omega\left(\frac{1-t_1-t_2-t_3-t_4-t_5}{t_5} \right)}{t_1 t_2 t_3 t_4 t_5^2} d t_5 d t_4 d t_3 d t_2 d t_1 \right) \\
\nonumber + & \left(\int_{(t_1, t_2, t_3, t_4, t_5) \in L_{308}} \frac{\omega\left(\frac{1-t_1-t_2-t_3-t_4-t_5}{t_5} \right)}{t_1 t_2 t_3 t_4 t_5^2} d t_5 d t_4 d t_3 d t_2 d t_1 \right) \\
\nonumber + & \left(\int_{(t_1) \in L_{309}} \frac{\omega\left(\frac{1-t_1}{t_1} \right)}{t_1^2} d t_1 \right) \\
\nonumber + & \left(\frac{1}{\kappa} \int_{\kappa}^{\frac{1}{2}} \int_{\kappa}^{\frac{1}{2}} \int_{\kappa}^{\frac{1}{2}} \mathbbm{1}_{(t_1, t_2, t_3) \in \boldsymbol{D}_3 } \frac{\omega\left(\frac{1-t_1-t_2-t_3}{\kappa} \right)}{t_1 t_2 t_3} d t_3 d t_2 d t_1 \right) \\
\nonumber + & \left(\frac{1}{\kappa} \int_{\kappa}^{\frac{1}{2}} \int_{\kappa}^{\frac{1}{2}} \int_{\kappa}^{\frac{1}{2}} \int_{\kappa}^{\frac{1}{2}} \mathbbm{1}_{(t_1, t_2, t_3, t_4) \in \boldsymbol{D}_4 } \frac{\omega\left(\frac{1-t_1-t_2-t_3-t_4}{\kappa} \right)}{t_1 t_2 t_3 t_4} d t_4 d t_3 d t_2 d t_1 \right) \\
\nonumber + & \left(\int_{(t_1, t_2, t_3) \in L_{312}} \frac{\omega\left(\frac{1-t_1-t_2-t_3}{t_3} \right)}{t_1 t_2 t_3^2} d t_3 d t_2 d t_1 \right) \\
\nonumber =&\ E_{301}(\theta) + E_{302}(\theta) + E_{303}(\theta) + E_{304}(\theta) + E_{305}(\theta) - E_{306}(\theta) - E_{307}(\theta) \\ 
& + E_{308}(\theta) + E_{309}(\theta) + E_{310}(\theta) + E_{311}(\theta) + E_{312}(\theta),
\end{align}
where
\begin{align}
\nonumber L_{301}(\boldsymbol{\alpha}_{3}) :=&\ \left\{ \kappa \leqslant \alpha_1 \leqslant \frac{\theta}{2},\ 1 - \theta - \alpha_1 \leqslant \alpha_2 \leqslant \theta - \alpha_1, \right. \\
\nonumber & \left. \quad \kappa \leqslant \alpha_3 \leqslant \min\left(\alpha_2, \frac{1}{2}(1-\alpha_1 -\alpha_2) \right),\ \boldsymbol{\alpha}_{3} \notin \boldsymbol{G}_{3} \right\}, \\
\nonumber L_{304}(\boldsymbol{\alpha}_{3}) :=&\ \left\{ \alpha_1 \in \left[\kappa,\ \tau \right) \cup \left(\frac{31 \theta - 15}{3},\ 4 - 7 \theta \right),\ \boldsymbol{\alpha}_{2} \in C, \right. \\
\nonumber & \quad \kappa \leqslant \alpha_3 < \min\left(\alpha_2, \frac{1}{2}(1-\alpha_1 -\alpha_2) \right),\ \boldsymbol{\alpha}_{3} \notin \boldsymbol{G}_{3}, \\
\nonumber & \left. \quad \kappa \leqslant \alpha_1 < \frac{1}{2},\ \kappa \leqslant \alpha_2 < \min\left(\alpha_1, \frac{1}{2}(1-\alpha_1) \right) \right\}, \\
\nonumber L_{305}(\boldsymbol{\alpha}_{3}) :=&\ \left\{\alpha_1 \in \left[\kappa,\ \tau \right) \cup \left(\frac{31 \theta - 15}{3},\ 4 - 7 \theta \right),\ \boldsymbol{\alpha}_{2} \in A \cup B, \right. \\
\nonumber & \quad \kappa \leqslant \alpha_3 < \min\left(\alpha_2, \frac{1}{2}(1-\alpha_1 -\alpha_2) \right),\ \boldsymbol{\alpha}_{3} \notin \boldsymbol{G}_{3}, \\
\nonumber & \quad \text{either } \boldsymbol{\alpha}_{3} \notin \boldsymbol{U}_3 \text{ or } (\alpha_1, \alpha_2, \alpha_3, \alpha_3) \notin \boldsymbol{U}_4, \\
\nonumber & \left. \quad \kappa \leqslant \alpha_1 < \frac{1}{2},\ \kappa \leqslant \alpha_2 < \min\left(\alpha_1, \frac{1}{2}(1-\alpha_1) \right) \right\}, \\
\nonumber L_{306}(\boldsymbol{\alpha}_{4}) :=&\ \left\{\alpha_1 \in \left[\kappa,\ \tau \right) \cup \left(\frac{31 \theta - 15}{3},\ 4 - 7 \theta \right),\ \boldsymbol{\alpha}_{2} \in A \cup B, \right. \\
\nonumber & \quad \kappa \leqslant \alpha_3 < \min\left(\alpha_2, \frac{1}{2}(1-\alpha_1 -\alpha_2) \right),\ \boldsymbol{\alpha}_{3} \notin \boldsymbol{G}_{3}, \\
\nonumber & \quad \text{either } \boldsymbol{\alpha}_{3} \notin \boldsymbol{U}_3 \text{ or } (\alpha_1, \alpha_2, \alpha_3, \alpha_3) \notin \boldsymbol{U}_4, \\
\nonumber & \quad \alpha_3 < \alpha_4 < \frac{1}{2}(1-\alpha_1-\alpha_2-\alpha_3),\ \boldsymbol{\alpha}_{4} \in \boldsymbol{G}_{4}, \\
\nonumber & \left. \quad \kappa \leqslant \alpha_1 < \frac{1}{2},\ \kappa \leqslant \alpha_2 < \min\left(\alpha_1, \frac{1}{2}(1-\alpha_1) \right) \right\}, \\
\nonumber L_{307}(\boldsymbol{\alpha}_{5}) :=&\ \left\{\alpha_1 \in \left[\kappa,\ \tau \right) \cup \left(\frac{31 \theta - 15}{3},\ 4 - 7 \theta \right),\ \boldsymbol{\alpha}_{2} \in A \cup B, \right. \\
\nonumber & \quad \kappa \leqslant \alpha_3 < \min\left(\alpha_2, \frac{1}{2}(1-\alpha_1 -\alpha_2) \right),\ \boldsymbol{\alpha}_{3} \notin \boldsymbol{G}_{3}, \\
\nonumber & \quad \text{either } \boldsymbol{\alpha}_{3} \notin \boldsymbol{U}_3 \text{ or } (\alpha_1, \alpha_2, \alpha_3, \alpha_3) \notin \boldsymbol{U}_4, \\
\nonumber & \quad \alpha_3 < \alpha_4 < \frac{1}{2}(1-\alpha_1-\alpha_2-\alpha_3),\ \boldsymbol{\alpha}_{4} \notin \boldsymbol{G}_{4}, \\
\nonumber & \quad \alpha_4 < \alpha_5 < \frac{1}{2}(1-\alpha_1-\alpha_2-\alpha_3-\alpha_4),\ \boldsymbol{\alpha}_{5} \in \boldsymbol{G}_{5}, \\
\nonumber & \left. \quad \kappa \leqslant \alpha_1 < \frac{1}{2},\ \kappa \leqslant \alpha_2 < \min\left(\alpha_1, \frac{1}{2}(1-\alpha_1) \right) \right\}, \\
\nonumber L_{308}(\boldsymbol{\alpha}_{5}) :=&\ \left\{\alpha_1 \in \left[\kappa,\ \tau \right) \cup \left(\frac{31 \theta - 15}{3},\ 4 - 7 \theta \right),\ \boldsymbol{\alpha}_{2} \in A \cup B, \right. \\
\nonumber & \quad \kappa \leqslant \alpha_3 < \min\left(\alpha_2, \frac{1}{2}(1-\alpha_1 -\alpha_2) \right),\ \boldsymbol{\alpha}_{3} \notin \boldsymbol{G}_{3}, \\
\nonumber & \quad \boldsymbol{\alpha}_{3} \in \boldsymbol{U}_3,\ (\alpha_1, \alpha_2, \alpha_3, \alpha_3) \in \boldsymbol{U}_4, \\
\nonumber & \quad \kappa \leqslant \alpha_4 < \min\left(\alpha_3, \frac{1}{2}(1-\alpha_1 -\alpha_2-\alpha_3) \right),\ \boldsymbol{\alpha}_{4} \notin \boldsymbol{G}_{4}, \\
\nonumber & \quad \kappa \leqslant \alpha_5 < \min\left(\alpha_4, \frac{1}{2}(1-\alpha_1 -\alpha_2-\alpha_3-\alpha_4) \right),\ \boldsymbol{\alpha}_{5} \notin \boldsymbol{G}_{5}, \\
\nonumber & \left. \quad \kappa \leqslant \alpha_1 < \frac{1}{2},\ \kappa \leqslant \alpha_2 < \min\left(\alpha_1, \frac{1}{2}(1-\alpha_1) \right) \right\}, \\
\nonumber L_{309}(\alpha_1) :=&\ \left\{ \alpha_1 \in \left[\tau,\ \frac{31 \theta - 15}{3} \right] \cup \left[4 - 7 \theta,\ \frac{3}{7} \right],\ \alpha_1 \notin \left[\frac{19 \theta - 7}{7},\ \frac{3}{7}\right] \right\}, \\
\nonumber L_{312}(\boldsymbol{\alpha}_{3}) :=&\ \left\{ \alpha_1 \in \left[\tau,\ \frac{31 \theta - 15}{3} \right] \cup \left[4 - 7 \theta,\ \frac{3}{7} \right],\ \alpha_1 \in \left[\frac{19 \theta - 7}{7},\ \frac{3}{7}\right],\ \boldsymbol{\alpha}_{2} \notin \boldsymbol{G}_{2}, \right. \\
\nonumber & \quad \kappa \leqslant \alpha_3 < \min\left(\alpha_2, \frac{1}{2}(1-\alpha_1 -\alpha_2) \right),\ \boldsymbol{\alpha}_{3} \notin \boldsymbol{G}_{3}, \\
\nonumber & \left. \quad \kappa \leqslant \alpha_1 < \frac{1}{2},\ \kappa \leqslant \alpha_2 < \min\left(\alpha_1, \frac{1}{2}(1-\alpha_1) \right) \right\}.
\end{align}

Finally, for $\frac{21}{41} \leqslant \theta < \frac{16}{31}$ we have
\begin{align}
\nonumber C(\theta) \leqslant C_{3}(\theta) :=&\ 1 + \left(E_{301}(\theta) + E_{302}(\theta) + E_{303}(\theta) + E_{304}(\theta) + E_{305}(\theta) - E_{306}(\theta) \right. \\
& \left. \qquad - E_{307}(\theta) + E_{308}(\theta) + E_{309}(\theta) + E_{310}(\theta) + E_{311}(\theta) + E_{312}(\theta) \right).
\end{align}

\begin{center}
\begin{tabular}{|c|c|c|c|c|c|}
\hline \boldmath{$\theta$} & \boldmath{$C_{3}(\theta)$} & \boldmath{$\theta$} & \boldmath{$C_{3}(\theta)$} & \boldmath{$\theta$} & \boldmath{$C_{3}(\theta)$} \\
\hline $0.5124$ & $1.03763$ & $0.5132$ & $1.09167$ & $0.5152$ & $1.23885$ \\
\hline $0.5126$ & $1.05089$ & $0.5135$ & $1.11171$ & $0.5155$ & $1.2675$ \\
\hline $0.5127$ & $1.05778$ & $0.514$ & $1.14303$ & $0.516$ & $1.31988$ \\
\hline $0.5128$ & $1.0652$ & $0.5145$ & $1.17429$ & $0.5161$ & $1.32981$ \\
\hline $0.513$ & $1.07887$ & $0.515$ & $1.21983$ & $\frac{16}{31}$ & $1.3326$ \\
\hline
\end{tabular}
\end{center}

\subsection{Case 4. $\frac{16}{31} \leqslant \theta < \frac{14}{27}$}
From here we note that Lemma~\ref{l52} is not applicable. Buchstab's identity yields
\begin{align}
\nonumber S\left(\mathcal{A}, (2x)^{\frac{1}{2}} \right) =&\ S\left(\mathcal{A}, x^{\kappa} \right) - \sum_{\kappa \leqslant \alpha_1 < \frac{1}{2}} S\left(\mathcal{A}_{p_1}, p_1 \right) \\
\nonumber =&\ S\left(\mathcal{A}, x^{\kappa} \right) - \sum_{\kappa \leqslant \alpha_1 < \tau} S\left(\mathcal{A}_{p_1}, p_1 \right) - \sum_{\tau \leqslant \alpha_1 \leqslant \frac{3}{7}} S\left(\mathcal{A}_{p_1}, p_1 \right) \\
\nonumber & - \sum_{\frac{3}{7} < \alpha_1 < 1 - \theta} S\left(\mathcal{A}_{p_1}, p_1 \right) - \sum_{1 - \theta \leqslant \alpha_1 < \frac{1}{2}} S\left(\mathcal{A}_{p_1}, p_1 \right) \\
=&\ S_{41} - S_{42} - S_{43} - S_{44} - S_{45}.
\end{align}
Now, we can only use Lemma~\ref{l53} to deal with parts of $S_{43}$ satisfy $\alpha_1 \in \left[\frac{19 \theta - 7}{7}, \frac{3}{7}\right]$, and discard the remaining parts. Combining with the loss integrals from $S_{42}$ and $S_{45}$, the total loss when $\frac{16}{31} \leqslant \theta < \frac{14}{27}$ is no more than
\begin{align}
\nonumber & \left(\int_{(t_1, t_2, t_3) \in L_{401}} \frac{\omega\left(\frac{t_2}{t_1} \right) \omega\left(\frac{1-t_1-t_2-t_3}{t_3} \right)}{t_1^2 t_3^2} d t_3 d t_2 d t_1 \right) \\
\nonumber + & \left(\int_{(t_1, t_2, t_3) \in L_{402}} \frac{\omega\left(\frac{1-t_1-t_2-t_3}{t_3} \right)}{t_1 t_2 t_3^2} d t_3 d t_2 d t_1 \right) \\
\nonumber - & \left(\int_{(t_1, t_2, t_3, t_4) \in L_{403}} \frac{\omega\left(\frac{1-t_1-t_2-t_3-t_4}{t_4} \right)}{t_1 t_2 t_3 t_4^2} d t_4 d t_3 d t_2 d t_1 \right) \\
\nonumber - & \left(\int_{(t_1, t_2, t_3, t_4, t_5) \in L_{404}} \frac{\omega\left(\frac{1-t_1-t_2-t_3-t_4-t_5}{t_5} \right)}{t_1 t_2 t_3 t_4 t_5^2} d t_5 d t_4 d t_3 d t_2 d t_1 \right) \\
\nonumber + & \left(\int_{(t_1, t_2, t_3, t_4, t_5) \in L_{405}} \frac{\omega\left(\frac{1-t_1-t_2-t_3-t_4-t_5}{t_5} \right)}{t_1 t_2 t_3 t_4 t_5^2} d t_5 d t_4 d t_3 d t_2 d t_1 \right) \\
\nonumber + & \left(\int_{(t_1) \in L_{406}} \frac{\omega\left(\frac{1-t_1}{t_1} \right)}{t_1^2} d t_1 \right) \\
\nonumber + & \left(\frac{1}{\kappa} \int_{\kappa}^{\frac{1}{2}} \int_{\kappa}^{\frac{1}{2}} \int_{\kappa}^{\frac{1}{2}} \mathbbm{1}_{(t_1, t_2, t_3) \in \boldsymbol{D}_3 } \frac{\omega\left(\frac{1-t_1-t_2-t_3}{\kappa} \right)}{t_1 t_2 t_3} d t_3 d t_2 d t_1 \right) \\
\nonumber + & \left(\frac{1}{\kappa} \int_{\kappa}^{\frac{1}{2}} \int_{\kappa}^{\frac{1}{2}} \int_{\kappa}^{\frac{1}{2}} \int_{\kappa}^{\frac{1}{2}} \mathbbm{1}_{(t_1, t_2, t_3, t_4) \in \boldsymbol{D}_4 } \frac{\omega\left(\frac{1-t_1-t_2-t_3-t_4}{\kappa} \right)}{t_1 t_2 t_3 t_4} d t_4 d t_3 d t_2 d t_1 \right) \\
\nonumber + & \left(\int_{(t_1, t_2, t_3) \in L_{409}} \frac{\omega\left(\frac{1-t_1-t_2-t_3}{t_3} \right)}{t_1 t_2 t_3^2} d t_3 d t_2 d t_1 \right) \\
=&\ E_{401}(\theta) + E_{402}(\theta) - E_{403}(\theta) - E_{404}(\theta) + E_{405}(\theta) + E_{406}(\theta) + E_{407}(\theta) + E_{408}(\theta) + E_{409}(\theta),
\end{align}
where
\begin{align}
\nonumber L_{401}(\boldsymbol{\alpha}_{3}) :=&\ \left\{ \kappa \leqslant \alpha_1 \leqslant \frac{\theta}{2},\ 1 - \theta - \alpha_1 \leqslant \alpha_2 \leqslant \theta - \alpha_1, \right. \\
\nonumber & \left. \quad \kappa \leqslant \alpha_3 \leqslant \min\left(\alpha_2, \frac{1}{2}(1-\alpha_1 -\alpha_2) \right),\ \boldsymbol{\alpha}_{3} \notin \boldsymbol{G}_{3} \right\}, \\
\nonumber L_{402}(\boldsymbol{\alpha}_{3}) :=&\ \left\{ \kappa \leqslant \alpha_1 < \tau,\ \kappa \leqslant \alpha_3 < \min\left(\alpha_2, \frac{1}{2}(1-\alpha_1 -\alpha_2) \right),\ \boldsymbol{\alpha}_{3} \notin \boldsymbol{G}_{3}, \right. \\
\nonumber & \quad \text{either } \boldsymbol{\alpha}_{3} \notin \boldsymbol{U}_3 \text{ or } (\alpha_1, \alpha_2, \alpha_3, \alpha_3) \notin \boldsymbol{U}_4, \\
\nonumber & \left. \quad \kappa \leqslant \alpha_1 < \frac{1}{2},\ \kappa \leqslant \alpha_2 < \min\left(\alpha_1, \frac{1}{2}(1-\alpha_1) \right) \right\}, \\
\nonumber L_{403}(\boldsymbol{\alpha}_{4}) :=&\ \left\{\kappa \leqslant \alpha_1 < \tau,\ \kappa \leqslant \alpha_3 < \min\left(\alpha_2, \frac{1}{2}(1-\alpha_1 -\alpha_2) \right),\ \boldsymbol{\alpha}_{3} \notin \boldsymbol{G}_{3}, \right. \\
\nonumber & \quad \text{either } \boldsymbol{\alpha}_{3} \notin \boldsymbol{U}_3 \text{ or } (\alpha_1, \alpha_2, \alpha_3, \alpha_3) \notin \boldsymbol{U}_4, \\
\nonumber & \quad \alpha_3 < \alpha_4 < \frac{1}{2}(1-\alpha_1-\alpha_2-\alpha_3),\ \boldsymbol{\alpha}_{4} \in \boldsymbol{G}_{4}, \\
\nonumber & \left. \quad \kappa \leqslant \alpha_1 < \frac{1}{2},\ \kappa \leqslant \alpha_2 < \min\left(\alpha_1, \frac{1}{2}(1-\alpha_1) \right) \right\}, \\
\nonumber L_{404}(\boldsymbol{\alpha}_{5}) :=&\ \left\{\kappa \leqslant \alpha_1 < \tau,\ \kappa \leqslant \alpha_3 < \min\left(\alpha_2, \frac{1}{2}(1-\alpha_1 -\alpha_2) \right),\ \boldsymbol{\alpha}_{3} \notin \boldsymbol{G}_{3}, \right. \\
\nonumber & \quad \text{either } \boldsymbol{\alpha}_{3} \notin \boldsymbol{U}_3 \text{ or } (\alpha_1, \alpha_2, \alpha_3, \alpha_3) \notin \boldsymbol{U}_4, \\
\nonumber & \quad \alpha_3 < \alpha_4 < \frac{1}{2}(1-\alpha_1-\alpha_2-\alpha_3),\ \boldsymbol{\alpha}_{4} \notin \boldsymbol{G}_{4}, \\
\nonumber & \quad \alpha_4 < \alpha_5 < \frac{1}{2}(1-\alpha_1-\alpha_2-\alpha_3-\alpha_4),\ \boldsymbol{\alpha}_{5} \in \boldsymbol{G}_{5}, \\
\nonumber & \left. \quad \kappa \leqslant \alpha_1 < \frac{1}{2},\ \kappa \leqslant \alpha_2 < \min\left(\alpha_1, \frac{1}{2}(1-\alpha_1) \right) \right\}, \\
\nonumber L_{405}(\boldsymbol{\alpha}_{5}) :=&\ \left\{\kappa \leqslant \alpha_1 < \tau,\ \kappa \leqslant \alpha_3 < \min\left(\alpha_2, \frac{1}{2}(1-\alpha_1 -\alpha_2) \right),\ \boldsymbol{\alpha}_{3} \notin \boldsymbol{G}_{3}, \right. \\
\nonumber & \quad \boldsymbol{\alpha}_{3} \in \boldsymbol{U}_3,\ (\alpha_1, \alpha_2, \alpha_3, \alpha_3) \in \boldsymbol{U}_4, \\
\nonumber & \quad \kappa \leqslant \alpha_4 < \min\left(\alpha_3, \frac{1}{2}(1-\alpha_1 -\alpha_2-\alpha_3) \right),\ \boldsymbol{\alpha}_{4} \notin \boldsymbol{G}_{4}, \\
\nonumber & \quad \kappa \leqslant \alpha_5 < \min\left(\alpha_4, \frac{1}{2}(1-\alpha_1 -\alpha_2-\alpha_3-\alpha_4) \right),\ \boldsymbol{\alpha}_{5} \notin \boldsymbol{G}_{5}, \\
\nonumber & \left. \quad \kappa \leqslant \alpha_1 < \frac{1}{2},\ \kappa \leqslant \alpha_2 < \min\left(\alpha_1, \frac{1}{2}(1-\alpha_1) \right) \right\}, \\
\nonumber L_{406}(\alpha_1) :=&\ \left\{ \tau \leqslant \alpha_1 \leqslant \frac{3}{7},\ \alpha_1 \notin \left[\frac{19 \theta - 7}{7},\ \frac{3}{7}\right] \right\}, \\
\nonumber L_{409}(\boldsymbol{\alpha}_{3}) :=&\ \left\{ \tau \leqslant \alpha_1 \leqslant \frac{3}{7},\ \alpha_1 \in \left[\frac{19 \theta - 7}{7},\ \frac{3}{7}\right],\ \boldsymbol{\alpha}_{2} \notin \boldsymbol{G}_{2}, \right. \\
\nonumber & \quad \kappa \leqslant \alpha_3 < \min\left(\alpha_2, \frac{1}{2}(1-\alpha_1 -\alpha_2) \right),\ \boldsymbol{\alpha}_{3} \notin \boldsymbol{G}_{3}, \\
\nonumber & \left. \quad \kappa \leqslant \alpha_1 < \frac{1}{2},\ \kappa \leqslant \alpha_2 < \min\left(\alpha_1, \frac{1}{2}(1-\alpha_1) \right) \right\}.
\end{align}

Finally, for $\frac{16}{31} \leqslant \theta < \frac{14}{27}$ we have
\begin{equation}
C(\theta) \leqslant C_{4}(\theta) := 1 + \left(E_{401}(\theta) + E_{402}(\theta) - E_{403}(\theta) - E_{404}(\theta) + E_{405}(\theta) + E_{406}(\theta) + E_{407}(\theta) + E_{408}(\theta) + E_{409}(\theta) \right).
\end{equation}

\begin{center}
\begin{tabular}{|c|c|c|c|c|c|}
\hline \boldmath{$\theta$} & \boldmath{$C_{4}(\theta)$} & \boldmath{$\theta$} & \boldmath{$C_{4}(\theta)$} & \boldmath{$\theta$} & \boldmath{$C_{4}(\theta)$} \\
\hline $0.5163$ & $1.55978$ & $0.5168$ & $1.56977$ & $0.5178$ & $1.59006$ \\
\hline $0.5164$ & $1.56181$ & $0.5169$ & $1.5718$ & $0.518$ & $1.59476$ \\
\hline $0.5165$ & $1.56386$ & $0.517$ & $1.57396$ & $0.5182$ & $1.59927$ \\
\hline $0.5166$ & $1.56604$ & $0.5172$ & $1.5775$ & $0.5185$ & $1.60315$ \\
\hline $0.5167$ & $1.56778$ & $0.5175$ & $1.58362$ & $\frac{14}{27}$ & $1.60358$ \\
\hline
\end{tabular}
\end{center}

\subsection{Case 5. $\frac{14}{27} \leqslant \theta < \frac{11}{21}$}
In this subrange, we have $\frac{50 \theta - 19}{17} \geqslant \frac{19 \theta - 7}{7}$, and the lower bound for $\alpha_1 \in \boldsymbol{R}_{0}$ becomes $\frac{50 \theta - 19}{17}$. By an exactly same process as the decomposition in \textbf{Case 4}, the total loss when $\frac{14}{27} \leqslant \theta < \frac{11}{21}$ can be bounded by
\begin{align}
\nonumber & \left(\int_{(t_1, t_2, t_3) \in L_{501}} \frac{\omega\left(\frac{t_2}{t_1} \right) \omega\left(\frac{1-t_1-t_2-t_3}{t_3} \right)}{t_1^2 t_3^2} d t_3 d t_2 d t_1 \right) \\
\nonumber + & \left(\int_{(t_1, t_2, t_3) \in L_{502}} \frac{\omega\left(\frac{1-t_1-t_2-t_3}{t_3} \right)}{t_1 t_2 t_3^2} d t_3 d t_2 d t_1 \right) \\
\nonumber - & \left(\int_{(t_1, t_2, t_3, t_4) \in L_{503}} \frac{\omega\left(\frac{1-t_1-t_2-t_3-t_4}{t_4} \right)}{t_1 t_2 t_3 t_4^2} d t_4 d t_3 d t_2 d t_1 \right) \\
\nonumber - & \left(\int_{(t_1, t_2, t_3, t_4, t_5) \in L_{504}} \frac{\omega\left(\frac{1-t_1-t_2-t_3-t_4-t_5}{t_5} \right)}{t_1 t_2 t_3 t_4 t_5^2} d t_5 d t_4 d t_3 d t_2 d t_1 \right) \\
\nonumber + & \left(\int_{(t_1, t_2, t_3, t_4, t_5) \in L_{505}} \frac{\omega\left(\frac{1-t_1-t_2-t_3-t_4-t_5}{t_5} \right)}{t_1 t_2 t_3 t_4 t_5^2} d t_5 d t_4 d t_3 d t_2 d t_1 \right) \\
\nonumber + & \left(\int_{(t_1) \in L_{506}} \frac{\omega\left(\frac{1-t_1}{t_1} \right)}{t_1^2} d t_1 \right) \\
\nonumber + & \left(\frac{1}{\kappa} \int_{\kappa}^{\frac{1}{2}} \int_{\kappa}^{\frac{1}{2}} \int_{\kappa}^{\frac{1}{2}} \mathbbm{1}_{(t_1, t_2, t_3) \in \boldsymbol{D}_3 } \frac{\omega\left(\frac{1-t_1-t_2-t_3}{\kappa} \right)}{t_1 t_2 t_3} d t_3 d t_2 d t_1 \right) \\
\nonumber + & \left(\frac{1}{\kappa} \int_{\kappa}^{\frac{1}{2}} \int_{\kappa}^{\frac{1}{2}} \int_{\kappa}^{\frac{1}{2}} \int_{\kappa}^{\frac{1}{2}} \mathbbm{1}_{(t_1, t_2, t_3, t_4) \in \boldsymbol{D}_4 } \frac{\omega\left(\frac{1-t_1-t_2-t_3-t_4}{\kappa} \right)}{t_1 t_2 t_3 t_4} d t_4 d t_3 d t_2 d t_1 \right) \\
\nonumber + & \left(\int_{(t_1, t_2, t_3) \in L_{509}} \frac{\omega\left(\frac{1-t_1-t_2-t_3}{t_3} \right)}{t_1 t_2 t_3^2} d t_3 d t_2 d t_1 \right) \\
=&\ E_{501}(\theta) + E_{502}(\theta) - E_{503}(\theta) - E_{504}(\theta) + E_{505}(\theta) + E_{506}(\theta) + E_{507}(\theta) + E_{508}(\theta) + E_{509}(\theta),
\end{align}
where
\begin{align}
\nonumber L_{501}(\boldsymbol{\alpha}_{3}) :=&\ \left\{ \kappa \leqslant \alpha_1 \leqslant \frac{\theta}{2},\ 1 - \theta - \alpha_1 \leqslant \alpha_2 \leqslant \theta - \alpha_1, \right. \\
\nonumber & \left. \quad \kappa \leqslant \alpha_3 \leqslant \min\left(\alpha_2, \frac{1}{2}(1-\alpha_1 -\alpha_2) \right),\ \boldsymbol{\alpha}_{3} \notin \boldsymbol{G}_{3} \right\}, \\
\nonumber L_{502}(\boldsymbol{\alpha}_{3}) :=&\ \left\{ \kappa \leqslant \alpha_1 < \tau,\ \kappa \leqslant \alpha_3 < \min\left(\alpha_2, \frac{1}{2}(1-\alpha_1 -\alpha_2) \right),\ \boldsymbol{\alpha}_{3} \notin \boldsymbol{G}_{3}, \right. \\
\nonumber & \quad \text{either } \boldsymbol{\alpha}_{3} \notin \boldsymbol{U}_3 \text{ or } (\alpha_1, \alpha_2, \alpha_3, \alpha_3) \notin \boldsymbol{U}_4, \\
\nonumber & \left. \quad \kappa \leqslant \alpha_1 < \frac{1}{2},\ \kappa \leqslant \alpha_2 < \min\left(\alpha_1, \frac{1}{2}(1-\alpha_1) \right) \right\}, \\
\nonumber L_{503}(\boldsymbol{\alpha}_{4}) :=&\ \left\{\kappa \leqslant \alpha_1 < \tau,\ \kappa \leqslant \alpha_3 < \min\left(\alpha_2, \frac{1}{2}(1-\alpha_1 -\alpha_2) \right),\ \boldsymbol{\alpha}_{3} \notin \boldsymbol{G}_{3}, \right. \\
\nonumber & \quad \text{either } \boldsymbol{\alpha}_{3} \notin \boldsymbol{U}_3 \text{ or } (\alpha_1, \alpha_2, \alpha_3, \alpha_3) \notin \boldsymbol{U}_4, \\
\nonumber & \quad \alpha_3 < \alpha_4 < \frac{1}{2}(1-\alpha_1-\alpha_2-\alpha_3),\ \boldsymbol{\alpha}_{4} \in \boldsymbol{G}_{4}, \\
\nonumber & \left. \quad \kappa \leqslant \alpha_1 < \frac{1}{2},\ \kappa \leqslant \alpha_2 < \min\left(\alpha_1, \frac{1}{2}(1-\alpha_1) \right) \right\}, \\
\nonumber L_{504}(\boldsymbol{\alpha}_{5}) :=&\ \left\{\kappa \leqslant \alpha_1 < \tau,\ \kappa \leqslant \alpha_3 < \min\left(\alpha_2, \frac{1}{2}(1-\alpha_1 -\alpha_2) \right),\ \boldsymbol{\alpha}_{3} \notin \boldsymbol{G}_{3}, \right. \\
\nonumber & \quad \text{either } \boldsymbol{\alpha}_{3} \notin \boldsymbol{U}_3 \text{ or } (\alpha_1, \alpha_2, \alpha_3, \alpha_3) \notin \boldsymbol{U}_4, \\
\nonumber & \quad \alpha_3 < \alpha_4 < \frac{1}{2}(1-\alpha_1-\alpha_2-\alpha_3),\ \boldsymbol{\alpha}_{4} \notin \boldsymbol{G}_{4}, \\
\nonumber & \quad \alpha_4 < \alpha_5 < \frac{1}{2}(1-\alpha_1-\alpha_2-\alpha_3-\alpha_4),\ \boldsymbol{\alpha}_{5} \in \boldsymbol{G}_{5}, \\
\nonumber & \left. \quad \kappa \leqslant \alpha_1 < \frac{1}{2},\ \kappa \leqslant \alpha_2 < \min\left(\alpha_1, \frac{1}{2}(1-\alpha_1) \right) \right\}, \\
\nonumber L_{505}(\boldsymbol{\alpha}_{5}) :=&\ \left\{\kappa \leqslant \alpha_1 < \tau,\ \kappa \leqslant \alpha_3 < \min\left(\alpha_2, \frac{1}{2}(1-\alpha_1 -\alpha_2) \right),\ \boldsymbol{\alpha}_{3} \notin \boldsymbol{G}_{3}, \right. \\
\nonumber & \quad \boldsymbol{\alpha}_{3} \in \boldsymbol{U}_3,\ (\alpha_1, \alpha_2, \alpha_3, \alpha_3) \in \boldsymbol{U}_4, \\
\nonumber & \quad \kappa \leqslant \alpha_4 < \min\left(\alpha_3, \frac{1}{2}(1-\alpha_1 -\alpha_2-\alpha_3) \right),\ \boldsymbol{\alpha}_{4} \notin \boldsymbol{G}_{4}, \\
\nonumber & \quad \kappa \leqslant \alpha_5 < \min\left(\alpha_4, \frac{1}{2}(1-\alpha_1 -\alpha_2-\alpha_3-\alpha_4) \right),\ \boldsymbol{\alpha}_{5} \notin \boldsymbol{G}_{5}, \\
\nonumber & \left. \quad \kappa \leqslant \alpha_1 < \frac{1}{2},\ \kappa \leqslant \alpha_2 < \min\left(\alpha_1, \frac{1}{2}(1-\alpha_1) \right) \right\}, \\
\nonumber L_{506}(\alpha_1) :=&\ \left\{ \tau \leqslant \alpha_1 \leqslant \frac{3}{7},\ \alpha_1 \notin \left[\frac{50 \theta - 19}{17},\ \frac{3}{7}\right] \right\}, \\
\nonumber L_{509}(\boldsymbol{\alpha}_{3}) :=&\ \left\{ \tau \leqslant \alpha_1 \leqslant \frac{3}{7},\ \alpha_1 \in \left[\frac{50 \theta - 19}{17},\ \frac{3}{7}\right],\ \boldsymbol{\alpha}_{2} \notin \boldsymbol{G}_{2}, \right. \\
\nonumber & \quad \kappa \leqslant \alpha_3 < \min\left(\alpha_2, \frac{1}{2}(1-\alpha_1 -\alpha_2) \right),\ \boldsymbol{\alpha}_{3} \notin \boldsymbol{G}_{3}, \\
\nonumber & \left. \quad \kappa \leqslant \alpha_1 < \frac{1}{2},\ \kappa \leqslant \alpha_2 < \min\left(\alpha_1, \frac{1}{2}(1-\alpha_1) \right) \right\}.
\end{align}

Finally, for $\frac{14}{27} \leqslant \theta < \frac{11}{21}$ we have
\begin{equation}
C(\theta) \leqslant C_{5}(\theta) := 1 + \left(E_{501}(\theta) + E_{502}(\theta) - E_{503}(\theta) - E_{504}(\theta) + E_{505}(\theta) + E_{506}(\theta) + E_{507}(\theta) + E_{508}(\theta) + E_{509}(\theta) \right).
\end{equation}
We remark that Harman [\cite{HarmanBOOK}, Page 188] mentioned that $C(0.52) < 1.632$, and we decide to use this value since our $C_{5}(0.52)$ is larger than $1.632$.

\begin{center}
\begin{tabular}{|c|c|c|c|c|c|c|c|}
\hline \boldmath{$\theta$} & \boldmath{$C_{5}(\theta)$} & \boldmath{$\theta$} & \boldmath{$C_{5}(\theta)$} & \boldmath{$\theta$} & \boldmath{$C_{5}(\theta)$} & \boldmath{$\theta$} & \boldmath{$C_{5}(\theta)$} \\
\hline $0.5186$ & $1.60737$ & $0.52$ & $1.632$ & $0.5215$ & $1.67861$ & $0.5232$ & $1.71925$ \\
\hline $0.5188$ & $1.60886$ & $0.5205$ & $1.64649$ & $0.522$ & $1.68723$ & $0.5234$ & $1.72482$ \\
\hline $0.519$ & $1.61355$ & $0.5208$ & $1.65321$ & $0.5225$ & $1.70046$ & $0.5236$ & $1.7305$ \\
\hline $0.5195$ & $1.62485$ & $0.5211$ & $1.66246$ & $0.523$ & $1.71379$ & $0.5237$ & $1.7337$ \\
\hline
\end{tabular}
\end{center}

\subsection{Case 6. $\frac{11}{21} \leqslant \theta < \frac{92}{175}$} 
From here, Lemma~\ref{l44} is not applicable. However, we do not need to worry about that. We note that for $\theta \geqslant \frac{11}{21}$, we must have $\boldsymbol{\alpha}_{2} \in A$ if $\alpha_2 < \alpha_1 < \tau = \frac{2}{7}$. If $\theta \leqslant \alpha_1 + \alpha_2 \leqslant 2 \tau = \frac{4}{7}$, we have $\frac{3}{7} \leqslant 1 - \alpha_1 - \alpha_2 < 1 - \theta$ and $\boldsymbol{\alpha}_{2} \in \boldsymbol{G}_{2}$. Otherwise we have $\alpha_1 + \alpha_2 < \theta$ ($\alpha_2 < \alpha_1 < \frac{\theta}{2}$), which means that $(\alpha_1, \alpha_2, 2\theta-1) \in \boldsymbol{S}_{3}$ (one can easily check the corresponding conditions) and $\boldsymbol{\alpha}_{2} \in \boldsymbol{U}_{2}^{\prime}$ (or $\boldsymbol{U}_{2}$ when $\theta < \frac{7}{13}$). Thus, we can use Lemma~\ref{l41} or Lemma~\ref{l42} to deal with the sum
\begin{equation}
\sum_{\kappa \leqslant \alpha_2 < \alpha_1 < \tau} S\left(\mathcal{A}_{p_1 p_2}, x^{\kappa} \right).
\end{equation}
The decomposing process of the remaining sums are similar to those in \textbf{Case 5}.

For the sake of simplicity, we choose to use the simple bound
\begin{equation}
C(\theta) \leqslant 1.725 + 13.125(\theta - 0.523)
\end{equation}
for $\frac{11}{21} \leqslant \theta < \frac{92}{175}$ given by Harman [\cite{HarmanBOOK}, Page 186]. Note that $C(\theta)$ is monotonic increasing, we have
\begin{equation}
C(\theta) \leqslant C\left(\frac{11}{21}\right) \leqslant 1.725 + 13.125\left(\frac{11}{21} - 0.523\right)
\end{equation}
for $0.5237 \leqslant \theta < \frac{11}{21}$. We also use the bound $C(0.525) < 1.75$ (mentioned in [\cite{HarmanBOOK}, Page 188]), which is slightly better than using (39) directly.

\subsection{Case 7. $\frac{92}{175} \leqslant \theta < \frac{17}{32}$}
From here, Lemma~\ref{l53} is not applicable and we must discard the whole of
\begin{equation}
\sum_{\tau \leqslant \alpha_1 < \frac{3}{7}} S\left(\mathcal{A}_{p_1}, p_1 \right).
\end{equation}

For the sake of simplicity, we shall also use the simple bound
\begin{equation}
C(\theta) \leqslant 1.725 + 13.125(\theta - 0.523)
\end{equation}
for $\frac{92}{175} \leqslant \theta < \frac{17}{32}$ given by Harman [\cite{HarmanBOOK}, Page 186].

\subsection{Case 8. $\frac{17}{32} \leqslant \theta < \frac{7}{13}$}
From here, Lemma~\ref{l43} and Lemma~\ref{l51} are not applicable. Buchstab's identity yields
\begin{align}
\nonumber S\left(\mathcal{A}, (2x)^{\frac{1}{2}} \right) =&\ S\left(\mathcal{A}, x^{\kappa} \right) - \sum_{\kappa \leqslant \alpha_1 < \frac{1}{2}} S\left(\mathcal{A}_{p_1}, p_1 \right) \\
\nonumber =&\ S\left(\mathcal{A}, x^{\kappa} \right) - \sum_{\kappa \leqslant \alpha_1 < \tau} S\left(\mathcal{A}_{p_1}, p_1 \right) - \sum_{\tau \leqslant \alpha_1 \leqslant \frac{3}{7}} S\left(\mathcal{A}_{p_1}, p_1 \right) \\
\nonumber & - \sum_{\frac{3}{7} < \alpha_1 < 1 - \theta} S\left(\mathcal{A}_{p_1}, p_1 \right) - \sum_{1 - \theta \leqslant \alpha_1 < \frac{1}{2}} S\left(\mathcal{A}_{p_1}, p_1 \right) \\
=&\ S_{81} - S_{82} - S_{83} - S_{84} - S_{85}.
\end{align}
$S_{81}$ and $S_{84}$ have asymptotic formulas. Since we cannot estimate the corresponding two-dimensional sums, we must discard the whole of $S_{83}$ and $S_{85}$. We can perform Buchstab's identity twice more on $S_{82}$ to get
\begin{align}
\nonumber S_{82} =&\ \sum_{\kappa \leqslant \alpha_1 < \tau} S\left(\mathcal{A}_{p_1}, p_1 \right) \\
\nonumber =&\ \sum_{\kappa \leqslant \alpha_1 < \tau} S\left(\mathcal{A}_{p_1}, x^{\kappa} \right) - \sum_{\substack{\kappa \leqslant \alpha_2 < \alpha_1 < \tau \\ \boldsymbol{\alpha}_{2} \in \boldsymbol{G}_{2} }} S\left(\mathcal{A}_{p_1 p_2}, p_2 \right) - \sum_{\substack{\kappa \leqslant \alpha_2 < \alpha_1 < \tau \\ \boldsymbol{\alpha}_{2} \notin \boldsymbol{G}_{2} }} S\left(\mathcal{A}_{p_1 p_2}, x^{\kappa} \right) \\
\nonumber & + \sum_{\substack{\kappa \leqslant \alpha_2 < \alpha_1 < \tau \\ \boldsymbol{\alpha}_{2} \notin \boldsymbol{G}_{2} \\ \kappa \leqslant \alpha_3 < \min\left(\alpha_2, \frac{1}{2}(1-\alpha_1 -\alpha_2)\right) \\ \boldsymbol{\alpha}_{3} \in \boldsymbol{G}_{3} }} S\left(\mathcal{A}_{p_1 p_2 p_3}, p_3 \right) + \sum_{\substack{\kappa \leqslant \alpha_2 < \alpha_1 < \tau \\ \boldsymbol{\alpha}_{2} \notin \boldsymbol{G}_{2} \\ \kappa \leqslant \alpha_3 < \min\left(\alpha_2, \frac{1}{2}(1-\alpha_1 -\alpha_2)\right) \\ \boldsymbol{\alpha}_{3} \notin \boldsymbol{G}_{3} }} S\left(\mathcal{A}_{p_1 p_2 p_3}, p_3 \right) \\
=&\ S_{821} - S_{822} - S_{823} + S_{824} + S_{825}.
\end{align}
$S_{822}$ and $S_{824}$ have asymptotic formulas. We can use Lemma~\ref{l42} to give an asymptotic formula for $S_{823}$ (and also $S_{821}$) by the discussion in \textbf{Case 6}. For $S_{825}$ we can check the same conditions as those in \textbf{Case 1}: if $\boldsymbol{\alpha}_{3} \in \boldsymbol{U}_3$ and
$$
(\alpha_1, \alpha_2, \alpha_3, \alpha_3) \in \boldsymbol{U}_4,
$$
for parts of $S_{825}$, then we can use Buchstab's identity twice again to get a five-dimensional sum. We can even get a seven-dimensional sum if we have $\boldsymbol{\alpha}_{5} \in \boldsymbol{U}_5$ and
$$
(\alpha_1, \alpha_2, \alpha_3, \alpha_4, \alpha_5, \alpha_5) \in \boldsymbol{U}_6.
$$
For the sums without further decompositions, we use the reversed Buchstab's identity to gain savings from almost-primes counted. Note that the halved value of $\kappa$ in this subrange of $\theta$ doubles the maximum number of prime factors of the elements counted in these sums. Summing up those resulted integrals, the total loss when $\frac{17}{32} \leqslant \theta < \frac{7}{13}$ is no more than
\begin{align}
\nonumber & \left(\int_{1-\theta}^{\frac{1}{2}} \frac{\omega\left(\frac{1-t_1}{t_1} \right)}{t_1^2} d t_1 \right) \\
\nonumber + & \left(\int_{(t_1, t_2, t_3) \in L_{802}} \frac{\omega\left(\frac{1-t_1-t_2-t_3}{t_3} \right)}{t_1 t_2 t_3^2} d t_3 d t_2 d t_1 \right) \\
\nonumber - & \left(\int_{(t_1, t_2, t_3, t_4) \in L_{803}} \frac{\omega\left(\frac{1-t_1-t_2-t_3-t_4}{t_4} \right)}{t_1 t_2 t_3 t_4^2} d t_4 d t_3 d t_2 d t_1 \right) \\
\nonumber - & \left(\int_{(t_1, t_2, t_3, t_4, t_5) \in L_{804}} \frac{\omega\left(\frac{1-t_1-t_2-t_3-t_4-t_5}{t_5} \right)}{t_1 t_2 t_3 t_4 t_5^2} d t_5 d t_4 d t_3 d t_2 d t_1 \right) \\
\nonumber + & \left(\int_{(t_1, t_2, t_3, t_4, t_5) \in L_{805}} \frac{\omega\left(\frac{1-t_1-t_2-t_3-t_4-t_5}{t_5} \right)}{t_1 t_2 t_3 t_4 t_5^2} d t_5 d t_4 d t_3 d t_2 d t_1 \right) \\
\nonumber - & \left(\int_{(t_1, t_2, t_3, t_4, t_5, t_6) \in L_{806}} \frac{\omega\left(\frac{1-t_1-t_2-t_3-t_4-t_5-t_6}{t_6} \right)}{t_1 t_2 t_3 t_4 t_5 t_6^2} d t_6 d t_5 d t_4 d t_3 d t_2 d t_1 \right) \\
\nonumber + & \left(\int_{\tau}^{\frac{3}{7}} \frac{\omega\left(\frac{1-t_1}{t_1} \right)}{t_1^2} d t_1 \right) \\
=&\ E_{801}(\theta) + E_{802}(\theta) - E_{803}(\theta) - E_{804}(\theta) + E_{805}(\theta) - E_{806}(\theta) + E_{807}(\theta),
\end{align}
where
\begin{align}
\nonumber L_{802}(\boldsymbol{\alpha}_{3}) :=&\ \left\{ \kappa \leqslant \alpha_1 < \tau,\ \kappa \leqslant \alpha_3 < \min\left(\alpha_2, \frac{1}{2}(1-\alpha_1 -\alpha_2) \right),\ \boldsymbol{\alpha}_{3} \notin \boldsymbol{G}_{3}, \right. \\
\nonumber & \quad \text{either } \boldsymbol{\alpha}_{3} \notin \boldsymbol{U}_3 \text{ or } (\alpha_1, \alpha_2, \alpha_3, \alpha_3) \notin \boldsymbol{U}_4, \\
\nonumber & \left. \quad \kappa \leqslant \alpha_1 < \frac{1}{2},\ \kappa \leqslant \alpha_2 < \min\left(\alpha_1, \frac{1}{2}(1-\alpha_1) \right) \right\}, \\
\nonumber L_{803}(\boldsymbol{\alpha}_{4}) :=&\ \left\{\kappa \leqslant \alpha_1 < \tau,\ \kappa \leqslant \alpha_3 < \min\left(\alpha_2, \frac{1}{2}(1-\alpha_1 -\alpha_2) \right),\ \boldsymbol{\alpha}_{3} \notin \boldsymbol{G}_{3}, \right. \\
\nonumber & \quad \text{either } \boldsymbol{\alpha}_{3} \notin \boldsymbol{U}_3 \text{ or } (\alpha_1, \alpha_2, \alpha_3, \alpha_3) \notin \boldsymbol{U}_4, \\
\nonumber & \quad \alpha_3 < \alpha_4 < \frac{1}{2}(1-\alpha_1-\alpha_2-\alpha_3),\ \boldsymbol{\alpha}_{4} \in \boldsymbol{G}_{4}, \\
\nonumber & \left. \quad \kappa \leqslant \alpha_1 < \frac{1}{2},\ \kappa \leqslant \alpha_2 < \min\left(\alpha_1, \frac{1}{2}(1-\alpha_1) \right) \right\}, \\
\nonumber L_{804}(\boldsymbol{\alpha}_{5}) :=&\ \left\{\kappa \leqslant \alpha_1 < \tau,\ \kappa \leqslant \alpha_3 < \min\left(\alpha_2, \frac{1}{2}(1-\alpha_1 -\alpha_2) \right),\ \boldsymbol{\alpha}_{3} \notin \boldsymbol{G}_{3}, \right. \\
\nonumber & \quad \text{either } \boldsymbol{\alpha}_{3} \notin \boldsymbol{U}_3 \text{ or } (\alpha_1, \alpha_2, \alpha_3, \alpha_3) \notin \boldsymbol{U}_4, \\
\nonumber & \quad \alpha_3 < \alpha_4 < \frac{1}{2}(1-\alpha_1-\alpha_2-\alpha_3),\ \boldsymbol{\alpha}_{4} \notin \boldsymbol{G}_{4}, \\
\nonumber & \quad \alpha_4 < \alpha_5 < \frac{1}{2}(1-\alpha_1-\alpha_2-\alpha_3-\alpha_4),\ \boldsymbol{\alpha}_{5} \in \boldsymbol{G}_{5}, \\
\nonumber & \left. \quad \kappa \leqslant \alpha_1 < \frac{1}{2},\ \kappa \leqslant \alpha_2 < \min\left(\alpha_1, \frac{1}{2}(1-\alpha_1) \right) \right\}, \\
\nonumber L_{805}(\boldsymbol{\alpha}_{5}) :=&\ \left\{\kappa \leqslant \alpha_1 < \tau,\ \kappa \leqslant \alpha_3 < \min\left(\alpha_2, \frac{1}{2}(1-\alpha_1 -\alpha_2) \right),\ \boldsymbol{\alpha}_{3} \notin \boldsymbol{G}_{3}, \right. \\
\nonumber & \quad \boldsymbol{\alpha}_{3} \in \boldsymbol{U}_3,\ (\alpha_1, \alpha_2, \alpha_3, \alpha_3) \in \boldsymbol{U}_4, \\
\nonumber & \quad \kappa \leqslant \alpha_4 < \min\left(\alpha_3, \frac{1}{2}(1-\alpha_1 -\alpha_2-\alpha_3) \right),\ \boldsymbol{\alpha}_{4} \notin \boldsymbol{G}_{4}, \\
\nonumber & \quad \kappa \leqslant \alpha_5 < \min\left(\alpha_4, \frac{1}{2}(1-\alpha_1 -\alpha_2-\alpha_3-\alpha_4) \right),\ \boldsymbol{\alpha}_{5} \notin \boldsymbol{G}_{5}, \\
\nonumber & \left. \quad \kappa \leqslant \alpha_1 < \frac{1}{2},\ \kappa \leqslant \alpha_2 < \min\left(\alpha_1, \frac{1}{2}(1-\alpha_1) \right) \right\}, \\
\nonumber L_{806}(\boldsymbol{\alpha}_{6}) :=&\ \left\{\kappa \leqslant \alpha_1 < \tau,\ \kappa \leqslant \alpha_3 < \min\left(\alpha_2, \frac{1}{2}(1-\alpha_1 -\alpha_2) \right),\ \boldsymbol{\alpha}_{3} \notin \boldsymbol{G}_{3}, \right. \\
\nonumber & \quad \boldsymbol{\alpha}_{3} \in \boldsymbol{U}_3,\ (\alpha_1, \alpha_2, \alpha_3, \alpha_3) \in \boldsymbol{U}_4, \\
\nonumber & \quad \kappa \leqslant \alpha_4 < \min\left(\alpha_3, \frac{1}{2}(1-\alpha_1 -\alpha_2-\alpha_3) \right),\ \boldsymbol{\alpha}_{4} \notin \boldsymbol{G}_{4}, \\
\nonumber & \quad \kappa \leqslant \alpha_5 < \min\left(\alpha_4, \frac{1}{2}(1-\alpha_1 -\alpha_2-\alpha_3-\alpha_4) \right),\ \boldsymbol{\alpha}_{5} \notin \boldsymbol{G}_{5}, \\
\nonumber & \quad \alpha_5 < \alpha_6 < \frac{1}{2}(1-\alpha_1-\alpha_2-\alpha_3-\alpha_4-\alpha_5),\ \boldsymbol{\alpha}_{6} \in \boldsymbol{G}_{6}, \\
\nonumber & \left. \quad \kappa \leqslant \alpha_1 < \frac{1}{2},\ \kappa \leqslant \alpha_2 < \min\left(\alpha_1, \frac{1}{2}(1-\alpha_1) \right) \right\}.
\end{align}

Finally, for $\frac{17}{32} \leqslant \theta < \frac{7}{13}$ we have
\begin{equation}
C(\theta) \leqslant C_{8}(\theta) := 1 + \left(E_{801}(\theta) + E_{802}(\theta) - E_{803}(\theta) - E_{804}(\theta) + E_{805}(\theta) - E_{806}(\theta) + E_{807}(\theta) \right).
\end{equation}

\begin{center}
\begin{tabular}{|c|c|c|c|c|c|c|c|}
\hline \boldmath{$\theta$} & \boldmath{$C_{8}(\theta)$} & \boldmath{$\theta$} & \boldmath{$C_{8}(\theta)$} & \boldmath{$\theta$} & \boldmath{$C_{8}(\theta)$} & \boldmath{$\theta$} & \boldmath{$C_{8}(\theta)$} \\
\hline $0.5316$ & $1.8603$ & $0.5342$ & $1.87985$ & $0.5353$ & $1.90671$ & $0.5373$ & $1.93236$ \\
\hline $0.5328$ & $1.86462$ & $0.5343$ & $1.89006$ & $0.5355$ & $1.91223$ & $0.5374$ & $1.93399$ \\
\hline $0.533$ & $1.86571$ & $0.5345$ & $1.89323$ & $0.5358$ & $1.91673$ & $0.5376$ & $1.93644$ \\
\hline $0.5333$ & $1.86834$ & $0.5348$ & $1.8991$ & $0.536$ & $1.91998$ & $0.5377$ & $1.93829$ \\
\hline $0.5335$ & $1.87103$ & $0.5349$ & $1.90152$ & $0.5363$ & $1.92343$ & $0.5378$ & $1.94163$ \\
\hline $0.534$ & $1.87798$ & $0.5351$ & $1.90373$ & $0.5365$ & $1.92673$ & $0.5381$ & $1.94862$ \\
\hline $0.5341$ & $1.87885$ & $0.5352$ & $1.90524$ & $0.537$ & $1.92762$ & $0.5384$ & $1.94911$ \\
\hline
\end{tabular}
\end{center}

\subsection{Case 9. $\frac{7}{13} \leqslant \theta < \frac{6}{11}$} From here, ${U}_{j}$ and ${U}_{j}^{\prime}$ do not cover the same region. Since we have $\tau^{\prime} = \frac{5-6\theta}{7} < \frac{2}{7} = \tau$ and $\kappa^{\prime} = \frac{11-20\theta}{6} < \frac{3-5\theta}{7} = \kappa$ in this subrange, Buchstab's identity gives us
\begin{align}
\nonumber S\left(\mathcal{A}, (2x)^{\frac{1}{2}} \right) =&\ S\left(\mathcal{A}, x^{\kappa} \right) - \sum_{\kappa \leqslant \alpha_1 < \frac{1}{2}} S\left(\mathcal{A}_{p_1}, p_1 \right) \\
\nonumber =&\ S\left(\mathcal{A}, x^{\kappa} \right) - \sum_{\kappa \leqslant \alpha_1 < \tau^{\prime}} S\left(\mathcal{A}_{p_1}, p_1 \right) - \sum_{\tau^{\prime} \leqslant \alpha_1 < \tau} S\left(\mathcal{A}_{p_1}, p_1 \right) - \sum_{\tau \leqslant \alpha_1 \leqslant \frac{3}{7}} S\left(\mathcal{A}_{p_1}, p_1 \right) \\
\nonumber & - \sum_{\frac{3}{7} < \alpha_1 < 1 - \theta} S\left(\mathcal{A}_{p_1}, p_1 \right) - \sum_{1 - \theta \leqslant \alpha_1 < \frac{1}{2}} S\left(\mathcal{A}_{p_1}, p_1 \right) \\
=&\ S_{91} - S_{92} - S_{93} - S_{94} - S_{95} - S_{96}.
\end{align}
We give asymptotic formulas for $S_{91}$ and $S_{95}$, and discard $S_{94}$ and $S_{96}$. For $S_{92}$ and $S_{93}$, we need to check several conditions. For $S_{92}$ we have $\alpha_2 < \alpha_1 < \frac{5-6\theta}{7}$. Since $\frac{8\theta-2}{7} > \frac{5-6\theta}{7}$, we have $\boldsymbol{\alpha}_{2} \in \boldsymbol{T}_{2}^{*}$ and $\boldsymbol{\alpha}_{2} \in \boldsymbol{U}_{2}$. Now we can use Buchstab's identity to get
\begin{align}
\nonumber S_{92} =&\ \sum_{\kappa \leqslant \alpha_1 < \tau^{\prime}} S\left(\mathcal{A}_{p_1}, p_1 \right) \\
\nonumber =&\ \sum_{\kappa \leqslant \alpha_1 < \tau^{\prime}} S\left(\mathcal{A}_{p_1}, x^{\kappa} \right) - \sum_{\substack{\kappa \leqslant \alpha_2 < \alpha_1 < \tau^{\prime} \\ \boldsymbol{\alpha}_{2} \in \boldsymbol{G}_{2} }} S\left(\mathcal{A}_{p_1 p_2}, p_2 \right) - \sum_{\substack{\kappa \leqslant \alpha_2 < \alpha_1 < \tau^{\prime} \\ \boldsymbol{\alpha}_{2} \notin \boldsymbol{G}_{2} }} S\left(\mathcal{A}_{p_1 p_2}, x^{\kappa} \right) \\
\nonumber & + \sum_{\substack{\kappa \leqslant \alpha_2 < \alpha_1 < \tau^{\prime} \\ \boldsymbol{\alpha}_{2} \notin \boldsymbol{G}_{2} \\ \kappa \leqslant \alpha_3 < \min\left(\alpha_2, \frac{1}{2}(1-\alpha_1 -\alpha_2)\right) \\ \boldsymbol{\alpha}_{3} \in \boldsymbol{G}_{3} }} S\left(\mathcal{A}_{p_1 p_2 p_3}, p_3 \right) + \sum_{\substack{\kappa \leqslant \alpha_2 < \alpha_1 < \tau^{\prime} \\ \boldsymbol{\alpha}_{2} \notin \boldsymbol{G}_{2} \\ \kappa \leqslant \alpha_3 < \min\left(\alpha_2, \frac{1}{2}(1-\alpha_1 -\alpha_2)\right) \\ \boldsymbol{\alpha}_{3} \notin \boldsymbol{G}_{3} }} S\left(\mathcal{A}_{p_1 p_2 p_3}, p_3 \right) \\
=&\ S_{921} - S_{922} - S_{923} + S_{924} + S_{925}.
\end{align}
Now we can find the loss from $S_{92}$ by the exactly same process used in $S_{82}$. For $S_{93}$ we shall use a similar process. Since $\alpha_1 \geqslant \tau^{\prime}$ in $S_{93}$, we cannot ensure that $\boldsymbol{\alpha}_{2} \in \boldsymbol{T}_{2}^{*}$ after a Buchstab iteration. However, by the discussion in \textbf{Case 6}, we only need to consider the case $\alpha_2 < \alpha_1 < \frac{\theta}{2}$. Since $\theta < \frac{6}{11}$, we find that $\left(\frac{\theta}{2}, \frac{\theta}{2}, 2 \theta - 1 \right) \in \boldsymbol{S}_{3}$, hence $(\alpha_1, \alpha_2, 2 \theta - 1) \in \boldsymbol{S}_{3}$ and $\boldsymbol{\alpha}_{2} \in \boldsymbol{U}_{2}^{\prime}$. By a similar Buchstab iteration process (with $\kappa^{\prime}$ and $\tau^{\prime}$), we get
\begin{align}
\nonumber S_{93} =&\ \sum_{\tau^{\prime} \leqslant \alpha_1 < \tau} S\left(\mathcal{A}_{p_1}, p_1 \right) \\
\nonumber =&\ \sum_{\tau^{\prime} \leqslant \alpha_1 < \tau} S\left(\mathcal{A}_{p_1}, x^{\kappa^{\prime}} \right) - \sum_{\substack{\tau^{\prime} \leqslant \alpha_1 < \tau \\ \kappa^{\prime} \leqslant \alpha_2 < \min\left(\alpha_1, \frac{1}{2}(1-\alpha_1)\right) \\ \boldsymbol{\alpha}_{2} \in \boldsymbol{G}_{2} }} S\left(\mathcal{A}_{p_1 p_2}, p_2 \right) - \sum_{\substack{\tau^{\prime} \leqslant \alpha_1 < \tau \\ \kappa^{\prime} \leqslant \alpha_2 < \min\left(\alpha_1, \frac{1}{2}(1-\alpha_1)\right) \\ \boldsymbol{\alpha}_{2} \notin \boldsymbol{G}_{2} }} S\left(\mathcal{A}_{p_1 p_2}, x^{\kappa^{\prime}} \right) \\
\nonumber & + \sum_{\substack{\tau^{\prime} \leqslant \alpha_1 < \tau \\ \kappa^{\prime} \leqslant \alpha_2 < \min\left(\alpha_1, \frac{1}{2}(1-\alpha_1)\right) \\ \boldsymbol{\alpha}_{2} \notin \boldsymbol{G}_{2} \\ \kappa^{\prime} \leqslant \alpha_3 < \min\left(\alpha_2, \frac{1}{2}(1-\alpha_1 -\alpha_2)\right) \\ \boldsymbol{\alpha}_{3} \in \boldsymbol{G}_{3} }} S\left(\mathcal{A}_{p_1 p_2 p_3}, p_3 \right) + \sum_{\substack{\tau^{\prime} \leqslant \alpha_1 < \tau \\ \kappa^{\prime} \leqslant \alpha_2 < \min\left(\alpha_1, \frac{1}{2}(1-\alpha_1)\right) \\ \boldsymbol{\alpha}_{2} \notin \boldsymbol{G}_{2} \\ \kappa^{\prime} \leqslant \alpha_3 < \min\left(\alpha_2, \frac{1}{2}(1-\alpha_1 -\alpha_2)\right) \\ \boldsymbol{\alpha}_{3} \notin \boldsymbol{G}_{3} }} S\left(\mathcal{A}_{p_1 p_2 p_3}, p_3 \right) \\
=&\ S_{931} - S_{932} - S_{933} + S_{934} + S_{935}.
\end{align}
We can find the loss from $S_{93}$ by almost the same process as used in $S_{92}$. We need to replace many $\kappa$ with $\kappa^{\prime}$ and $\boldsymbol{U}_{j}$ with $\boldsymbol{U}_{j}^{\prime}$ in the resulting integrals. We note that seven-dimensional sums are considered in both $S_{92}$ and $S_{93}$. Putting those integrals together, the total loss when $\frac{7}{13} \leqslant \theta < \frac{6}{11}$ can be bounded by
\begin{align}
\nonumber & \left(\int_{1-\theta}^{\frac{1}{2}} \frac{\omega\left(\frac{1-t_1}{t_1} \right)}{t_1^2} d t_1 \right) \\
\nonumber + & \left(\int_{(t_1, t_2, t_3) \in L_{902}} \frac{\omega\left(\frac{1-t_1-t_2-t_3}{t_3} \right)}{t_1 t_2 t_3^2} d t_3 d t_2 d t_1 \right) \\
\nonumber - & \left(\int_{(t_1, t_2, t_3, t_4) \in L_{903}} \frac{\omega\left(\frac{1-t_1-t_2-t_3-t_4}{t_4} \right)}{t_1 t_2 t_3 t_4^2} d t_4 d t_3 d t_2 d t_1 \right) \\
\nonumber - & \left(\int_{(t_1, t_2, t_3, t_4, t_5) \in L_{904}} \frac{\omega\left(\frac{1-t_1-t_2-t_3-t_4-t_5}{t_5} \right)}{t_1 t_2 t_3 t_4 t_5^2} d t_5 d t_4 d t_3 d t_2 d t_1 \right) \\
\nonumber + & \left(\int_{(t_1, t_2, t_3, t_4, t_5) \in L_{905}} \frac{\omega\left(\frac{1-t_1-t_2-t_3-t_4-t_5}{t_5} \right)}{t_1 t_2 t_3 t_4 t_5^2} d t_5 d t_4 d t_3 d t_2 d t_1 \right) \\
\nonumber - & \left(\int_{(t_1, t_2, t_3, t_4, t_5, t_6) \in L_{906}} \frac{\omega\left(\frac{1-t_1-t_2-t_3-t_4-t_5-t_6}{t_6} \right)}{t_1 t_2 t_3 t_4 t_5 t_6^2} d t_6 d t_5 d t_4 d t_3 d t_2 d t_1 \right) \\
\nonumber - & \left(\int_{(t_1, t_2, t_3, t_4, t_5, t_6, t_7) \in L_{907}} \frac{\omega\left(\frac{1-t_1-t_2-t_3-t_4-t_5-t_6-t_7}{t_7} \right)}{t_1 t_2 t_3 t_4 t_5 t_6 t_7^2} d t_7 d t_6 d t_5 d t_4 d t_3 d t_2 d t_1 \right) \\
\nonumber + & \left(\int_{(t_1, t_2, t_3, t_4, t_5, t_6, t_7) \in L_{908}} \frac{\omega\left(\frac{1-t_1-t_2-t_3-t_4-t_5-t_6-t_7}{t_7} \right)}{t_1 t_2 t_3 t_4 t_5 t_6 t_7^2} d t_7 d t_6 d t_5 d t_4 d t_3 d t_2 d t_1 \right) \\
\nonumber + & \left(\int_{(t_1, t_2, t_3) \in L_{909}} \frac{\omega\left(\frac{1-t_1-t_2-t_3}{t_3} \right)}{t_1 t_2 t_3^2} d t_3 d t_2 d t_1 \right) \\
\nonumber - & \left(\int_{(t_1, t_2, t_3, t_4) \in L_{910}} \frac{\omega\left(\frac{1-t_1-t_2-t_3-t_4}{t_4} \right)}{t_1 t_2 t_3 t_4^2} d t_4 d t_3 d t_2 d t_1 \right) \\
\nonumber - & \left(\int_{(t_1, t_2, t_3, t_4, t_5) \in L_{911}} \frac{\omega\left(\frac{1-t_1-t_2-t_3-t_4-t_5}{t_5} \right)}{t_1 t_2 t_3 t_4 t_5^2} d t_5 d t_4 d t_3 d t_2 d t_1 \right) \\
\nonumber + & \left(\int_{(t_1, t_2, t_3, t_4, t_5) \in L_{912}} \frac{\omega\left(\frac{1-t_1-t_2-t_3-t_4-t_5}{t_5} \right)}{t_1 t_2 t_3 t_4 t_5^2} d t_5 d t_4 d t_3 d t_2 d t_1 \right) \\
\nonumber - & \left(\int_{(t_1, t_2, t_3, t_4, t_5, t_6) \in L_{913}} \frac{\omega\left(\frac{1-t_1-t_2-t_3-t_4-t_5-t_6}{t_6} \right)}{t_1 t_2 t_3 t_4 t_5 t_6^2} d t_6 d t_5 d t_4 d t_3 d t_2 d t_1 \right) \\
\nonumber - & \left(\int_{(t_1, t_2, t_3, t_4, t_5, t_6, t_7) \in L_{914}} \frac{\omega\left(\frac{1-t_1-t_2-t_3-t_4-t_5-t_6-t_7}{t_7} \right)}{t_1 t_2 t_3 t_4 t_5 t_6 t_7^2} d t_7 d t_6 d t_5 d t_4 d t_3 d t_2 d t_1 \right) \\
\nonumber + & \left(\int_{(t_1, t_2, t_3, t_4, t_5, t_6, t_7) \in L_{915}} \frac{\omega\left(\frac{1-t_1-t_2-t_3-t_4-t_5-t_6-t_7}{t_7} \right)}{t_1 t_2 t_3 t_4 t_5 t_6 t_7^2} d t_7 d t_6 d t_5 d t_4 d t_3 d t_2 d t_1 \right) \\
\nonumber + & \left(\int_{\tau}^{\frac{3}{7}} \frac{\omega\left(\frac{1-t_1}{t_1} \right)}{t_1^2} d t_1 \right) \\
\nonumber =&\ E_{901}(\theta) + E_{902}(\theta) - E_{903}(\theta) - E_{904}(\theta) + E_{905}(\theta) - E_{906}(\theta) - E_{907}(\theta) + E_{908}(\theta) \\
& + E_{909}(\theta) - E_{910}(\theta) - E_{911}(\theta) + E_{912}(\theta) - E_{913}(\theta) - E_{914}(\theta) + E_{915}(\theta) + E_{916}(\theta),
\end{align}
where
\begin{align}
\nonumber L_{902}(\boldsymbol{\alpha}_{3}) :=&\ \left\{ \kappa \leqslant \alpha_1 < \tau^{\prime},\ \kappa \leqslant \alpha_3 < \min\left(\alpha_2, \frac{1}{2}(1-\alpha_1 -\alpha_2) \right),\ \boldsymbol{\alpha}_{3} \notin \boldsymbol{G}_{3}, \right. \\
\nonumber & \quad \text{either } \boldsymbol{\alpha}_{3} \notin \boldsymbol{U}_3 \text{ or } (\alpha_1, \alpha_2, \alpha_3, \alpha_3) \notin \boldsymbol{U}_4, \\
\nonumber & \left. \quad \kappa \leqslant \alpha_1 < \frac{1}{2},\ \kappa \leqslant \alpha_2 < \min\left(\alpha_1, \frac{1}{2}(1-\alpha_1) \right) \right\}, \\
\nonumber L_{903}(\boldsymbol{\alpha}_{4}) :=&\ \left\{\kappa \leqslant \alpha_1 < \tau^{\prime},\ \kappa \leqslant \alpha_3 < \min\left(\alpha_2, \frac{1}{2}(1-\alpha_1 -\alpha_2) \right),\ \boldsymbol{\alpha}_{3} \notin \boldsymbol{G}_{3}, \right. \\
\nonumber & \quad \text{either } \boldsymbol{\alpha}_{3} \notin \boldsymbol{U}_3 \text{ or } (\alpha_1, \alpha_2, \alpha_3, \alpha_3) \notin \boldsymbol{U}_4, \\
\nonumber & \quad \alpha_3 < \alpha_4 < \frac{1}{2}(1-\alpha_1-\alpha_2-\alpha_3),\ \boldsymbol{\alpha}_{4} \in \boldsymbol{G}_{4}, \\
\nonumber & \left. \quad \kappa \leqslant \alpha_1 < \frac{1}{2},\ \kappa \leqslant \alpha_2 < \min\left(\alpha_1, \frac{1}{2}(1-\alpha_1) \right) \right\}, \\
\nonumber L_{904}(\boldsymbol{\alpha}_{5}) :=&\ \left\{\kappa \leqslant \alpha_1 < \tau^{\prime},\ \kappa \leqslant \alpha_3 < \min\left(\alpha_2, \frac{1}{2}(1-\alpha_1 -\alpha_2) \right),\ \boldsymbol{\alpha}_{3} \notin \boldsymbol{G}_{3}, \right. \\
\nonumber & \quad \text{either } \boldsymbol{\alpha}_{3} \notin \boldsymbol{U}_3 \text{ or } (\alpha_1, \alpha_2, \alpha_3, \alpha_3) \notin \boldsymbol{U}_4, \\
\nonumber & \quad \alpha_3 < \alpha_4 < \frac{1}{2}(1-\alpha_1-\alpha_2-\alpha_3),\ \boldsymbol{\alpha}_{4} \notin \boldsymbol{G}_{4}, \\
\nonumber & \quad \alpha_4 < \alpha_5 < \frac{1}{2}(1-\alpha_1-\alpha_2-\alpha_3-\alpha_4),\ \boldsymbol{\alpha}_{5} \in \boldsymbol{G}_{5}, \\
\nonumber & \left. \quad \kappa \leqslant \alpha_1 < \frac{1}{2},\ \kappa \leqslant \alpha_2 < \min\left(\alpha_1, \frac{1}{2}(1-\alpha_1) \right) \right\}, \\
\nonumber L_{905}(\boldsymbol{\alpha}_{5}) :=&\ \left\{\kappa \leqslant \alpha_1 < \tau^{\prime},\ \kappa \leqslant \alpha_3 < \min\left(\alpha_2, \frac{1}{2}(1-\alpha_1 -\alpha_2) \right),\ \boldsymbol{\alpha}_{3} \notin \boldsymbol{G}_{3}, \right. \\
\nonumber & \quad \boldsymbol{\alpha}_{3} \in \boldsymbol{U}_3,\ (\alpha_1, \alpha_2, \alpha_3, \alpha_3) \in \boldsymbol{U}_4, \\
\nonumber & \quad \kappa \leqslant \alpha_4 < \min\left(\alpha_3, \frac{1}{2}(1-\alpha_1 -\alpha_2-\alpha_3) \right),\ \boldsymbol{\alpha}_{4} \notin \boldsymbol{G}_{4}, \\
\nonumber & \quad \kappa \leqslant \alpha_5 < \min\left(\alpha_4, \frac{1}{2}(1-\alpha_1 -\alpha_2-\alpha_3-\alpha_4) \right),\ \boldsymbol{\alpha}_{5} \notin \boldsymbol{G}_{5}, \\
\nonumber & \quad \text{either } \boldsymbol{\alpha}_{5} \notin \boldsymbol{U}_5 \text{ or } (\alpha_1, \alpha_2, \alpha_3, \alpha_4, \alpha_5, \alpha_5) \notin \boldsymbol{U}_6, \\
\nonumber & \left. \quad \kappa \leqslant \alpha_1 < \frac{1}{2},\ \kappa \leqslant \alpha_2 < \min\left(\alpha_1, \frac{1}{2}(1-\alpha_1) \right) \right\}, \\
\nonumber L_{906}(\boldsymbol{\alpha}_{6}) :=&\ \left\{\kappa \leqslant \alpha_1 < \tau^{\prime},\ \kappa \leqslant \alpha_3 < \min\left(\alpha_2, \frac{1}{2}(1-\alpha_1 -\alpha_2) \right),\ \boldsymbol{\alpha}_{3} \notin \boldsymbol{G}_{3}, \right. \\
\nonumber & \quad \boldsymbol{\alpha}_{3} \in \boldsymbol{U}_3,\ (\alpha_1, \alpha_2, \alpha_3, \alpha_3) \in \boldsymbol{U}_4, \\
\nonumber & \quad \kappa \leqslant \alpha_4 < \min\left(\alpha_3, \frac{1}{2}(1-\alpha_1 -\alpha_2-\alpha_3) \right),\ \boldsymbol{\alpha}_{4} \notin \boldsymbol{G}_{4}, \\
\nonumber & \quad \kappa \leqslant \alpha_5 < \min\left(\alpha_4, \frac{1}{2}(1-\alpha_1 -\alpha_2-\alpha_3-\alpha_4) \right),\ \boldsymbol{\alpha}_{5} \notin \boldsymbol{G}_{5}, \\
\nonumber & \quad \text{either } \boldsymbol{\alpha}_{5} \notin \boldsymbol{U}_5 \text{ or } (\alpha_1, \alpha_2, \alpha_3, \alpha_4, \alpha_5, \alpha_5) \notin \boldsymbol{U}_6, \\
\nonumber & \quad \alpha_5 < \alpha_6 < \frac{1}{2}(1-\alpha_1-\alpha_2-\alpha_3-\alpha_4-\alpha_5),\ \boldsymbol{\alpha}_{6} \in \boldsymbol{G}_{6}, \\
\nonumber & \left. \quad \kappa \leqslant \alpha_1 < \frac{1}{2},\ \kappa \leqslant \alpha_2 < \min\left(\alpha_1, \frac{1}{2}(1-\alpha_1) \right) \right\}, \\
\nonumber L_{907}(\boldsymbol{\alpha}_{7}) :=&\ \left\{\kappa \leqslant \alpha_1 < \tau^{\prime},\ \kappa \leqslant \alpha_3 < \min\left(\alpha_2, \frac{1}{2}(1-\alpha_1 -\alpha_2) \right),\ \boldsymbol{\alpha}_{3} \notin \boldsymbol{G}_{3}, \right. \\
\nonumber & \quad \boldsymbol{\alpha}_{3} \in \boldsymbol{U}_3,\ (\alpha_1, \alpha_2, \alpha_3, \alpha_3) \in \boldsymbol{U}_4, \\
\nonumber & \quad \kappa \leqslant \alpha_4 < \min\left(\alpha_3, \frac{1}{2}(1-\alpha_1 -\alpha_2-\alpha_3) \right),\ \boldsymbol{\alpha}_{4} \notin \boldsymbol{G}_{4}, \\
\nonumber & \quad \kappa \leqslant \alpha_5 < \min\left(\alpha_4, \frac{1}{2}(1-\alpha_1 -\alpha_2-\alpha_3-\alpha_4) \right),\ \boldsymbol{\alpha}_{5} \notin \boldsymbol{G}_{5}, \\
\nonumber & \quad \text{either } \boldsymbol{\alpha}_{5} \notin \boldsymbol{U}_5 \text{ or } (\alpha_1, \alpha_2, \alpha_3, \alpha_4, \alpha_5, \alpha_5) \notin \boldsymbol{U}_6, \\
\nonumber & \quad \alpha_5 < \alpha_6 < \frac{1}{2}(1-\alpha_1-\alpha_2-\alpha_3-\alpha_4-\alpha_5),\ \boldsymbol{\alpha}_{6} \notin \boldsymbol{G}_{6}, \\
\nonumber & \quad \alpha_6 < \alpha_7 < \frac{1}{2}(1-\alpha_1-\alpha_2-\alpha_3-\alpha_4-\alpha_5-\alpha_6),\ \boldsymbol{\alpha}_{7} \in \boldsymbol{G}_{7}, \\
\nonumber & \left. \quad \kappa \leqslant \alpha_1 < \frac{1}{2},\ \kappa \leqslant \alpha_2 < \min\left(\alpha_1, \frac{1}{2}(1-\alpha_1) \right) \right\}, \\
\nonumber L_{908}(\boldsymbol{\alpha}_{7}) :=&\ \left\{\kappa \leqslant \alpha_1 < \tau^{\prime},\ \kappa \leqslant \alpha_3 < \min\left(\alpha_2, \frac{1}{2}(1-\alpha_1 -\alpha_2) \right),\ \boldsymbol{\alpha}_{3} \notin \boldsymbol{G}_{3}, \right. \\
\nonumber & \quad \boldsymbol{\alpha}_{3} \in \boldsymbol{U}_3,\ (\alpha_1, \alpha_2, \alpha_3, \alpha_3) \in \boldsymbol{U}_4, \\
\nonumber & \quad \kappa \leqslant \alpha_4 < \min\left(\alpha_3, \frac{1}{2}(1-\alpha_1 -\alpha_2-\alpha_3) \right),\ \boldsymbol{\alpha}_{4} \notin \boldsymbol{G}_{4}, \\
\nonumber & \quad \kappa \leqslant \alpha_5 < \min\left(\alpha_4, \frac{1}{2}(1-\alpha_1 -\alpha_2-\alpha_3-\alpha_4) \right),\ \boldsymbol{\alpha}_{5} \notin \boldsymbol{G}_{5}, \\
\nonumber & \quad \boldsymbol{\alpha}_{5} \in \boldsymbol{U}_5,\ (\alpha_1, \alpha_2, \alpha_3, \alpha_4, \alpha_5, \alpha_5) \in \boldsymbol{U}_6, \\
\nonumber & \quad \kappa \leqslant \alpha_6 < \min\left(\alpha_5, \frac{1}{2}(1-\alpha_1 -\alpha_2-\alpha_3-\alpha_4-\alpha_5) \right),\ \boldsymbol{\alpha}_{6} \notin \boldsymbol{G}_{6}, \\
\nonumber & \quad \kappa \leqslant \alpha_7 < \min\left(\alpha_6, \frac{1}{2}(1-\alpha_1 -\alpha_2-\alpha_3-\alpha_4-\alpha_5-\alpha_6) \right),\ \boldsymbol{\alpha}_{7} \notin \boldsymbol{G}_{7}, \\
\nonumber & \left. \quad \kappa \leqslant \alpha_1 < \frac{1}{2},\ \kappa \leqslant \alpha_2 < \min\left(\alpha_1, \frac{1}{2}(1-\alpha_1) \right) \right\}, \\
\nonumber L_{909}(\boldsymbol{\alpha}_{3}) :=&\ \left\{ \tau^{\prime} \leqslant \alpha_1 < \tau,\ \kappa^{\prime} \leqslant \alpha_3 < \min\left(\alpha_2, \frac{1}{2}(1-\alpha_1 -\alpha_2) \right),\ \boldsymbol{\alpha}_{3} \notin \boldsymbol{G}_{3}, \right. \\
\nonumber & \quad \text{either } \boldsymbol{\alpha}_{3} \notin \boldsymbol{U}^{\prime}_3 \text{ or } (\alpha_1, \alpha_2, \alpha_3, \alpha_3) \notin \boldsymbol{U}^{\prime}_4, \\
\nonumber & \left. \quad \kappa \leqslant \alpha_1 < \frac{1}{2},\ \kappa^{\prime} \leqslant \alpha_2 < \min\left(\alpha_1, \frac{1}{2}(1-\alpha_1) \right) \right\}, \\
\nonumber L_{910}(\boldsymbol{\alpha}_{4}) :=&\ \left\{\tau^{\prime} \leqslant \alpha_1 < \tau,\ \kappa^{\prime} \leqslant \alpha_3 < \min\left(\alpha_2, \frac{1}{2}(1-\alpha_1 -\alpha_2) \right),\ \boldsymbol{\alpha}_{3} \notin \boldsymbol{G}_{3}, \right. \\
\nonumber & \quad \text{either } \boldsymbol{\alpha}_{3} \notin \boldsymbol{U}^{\prime}_3 \text{ or } (\alpha_1, \alpha_2, \alpha_3, \alpha_3) \notin \boldsymbol{U}^{\prime}_4, \\
\nonumber & \quad \alpha_3 < \alpha_4 < \frac{1}{2}(1-\alpha_1-\alpha_2-\alpha_3),\ \boldsymbol{\alpha}_{4} \in \boldsymbol{G}_{4}, \\
\nonumber & \left. \quad \kappa \leqslant \alpha_1 < \frac{1}{2},\ \kappa^{\prime} \leqslant \alpha_2 < \min\left(\alpha_1, \frac{1}{2}(1-\alpha_1) \right) \right\}, \\
\nonumber L_{911}(\boldsymbol{\alpha}_{5}) :=&\ \left\{\tau^{\prime} \leqslant \alpha_1 < \tau,\ \kappa^{\prime} \leqslant \alpha_3 < \min\left(\alpha_2, \frac{1}{2}(1-\alpha_1 -\alpha_2) \right),\ \boldsymbol{\alpha}_{3} \notin \boldsymbol{G}_{3}, \right. \\
\nonumber & \quad \text{either } \boldsymbol{\alpha}_{3} \notin \boldsymbol{U}^{\prime}_3 \text{ or } (\alpha_1, \alpha_2, \alpha_3, \alpha_3) \notin \boldsymbol{U}^{\prime}_4, \\
\nonumber & \quad \alpha_3 < \alpha_4 < \frac{1}{2}(1-\alpha_1-\alpha_2-\alpha_3),\ \boldsymbol{\alpha}_{4} \notin \boldsymbol{G}_{4}, \\
\nonumber & \quad \alpha_4 < \alpha_5 < \frac{1}{2}(1-\alpha_1-\alpha_2-\alpha_3-\alpha_4),\ \boldsymbol{\alpha}_{5} \in \boldsymbol{G}_{5}, \\
\nonumber & \left. \quad \kappa \leqslant \alpha_1 < \frac{1}{2},\ \kappa^{\prime} \leqslant \alpha_2 < \min\left(\alpha_1, \frac{1}{2}(1-\alpha_1) \right) \right\}, \\
\nonumber L_{912}(\boldsymbol{\alpha}_{5}) :=&\ \left\{\tau^{\prime} \leqslant \alpha_1 < \tau,\ \kappa^{\prime} \leqslant \alpha_3 < \min\left(\alpha_2, \frac{1}{2}(1-\alpha_1 -\alpha_2) \right),\ \boldsymbol{\alpha}_{3} \notin \boldsymbol{G}_{3}, \right. \\
\nonumber & \quad \boldsymbol{\alpha}_{3} \in \boldsymbol{U}^{\prime}_3,\ (\alpha_1, \alpha_2, \alpha_3, \alpha_3) \in \boldsymbol{U}^{\prime}_4, \\
\nonumber & \quad \kappa^{\prime} \leqslant \alpha_4 < \min\left(\alpha_3, \frac{1}{2}(1-\alpha_1 -\alpha_2-\alpha_3) \right),\ \boldsymbol{\alpha}_{4} \notin \boldsymbol{G}_{4}, \\
\nonumber & \quad \kappa^{\prime} \leqslant \alpha_5 < \min\left(\alpha_4, \frac{1}{2}(1-\alpha_1 -\alpha_2-\alpha_3-\alpha_4) \right),\ \boldsymbol{\alpha}_{5} \notin \boldsymbol{G}_{5}, \\
\nonumber & \quad \text{either } \boldsymbol{\alpha}_{5} \notin \boldsymbol{U}^{\prime}_5 \text{ or } (\alpha_1, \alpha_2, \alpha_3, \alpha_4, \alpha_5, \alpha_5) \notin \boldsymbol{U}^{\prime}_6, \\
\nonumber & \left. \quad \kappa \leqslant \alpha_1 < \frac{1}{2},\ \kappa^{\prime} \leqslant \alpha_2 < \min\left(\alpha_1, \frac{1}{2}(1-\alpha_1) \right) \right\}, \\
\nonumber L_{913}(\boldsymbol{\alpha}_{6}) :=&\ \left\{\tau^{\prime} \leqslant \alpha_1 < \tau,\ \kappa^{\prime} \leqslant \alpha_3 < \min\left(\alpha_2, \frac{1}{2}(1-\alpha_1 -\alpha_2) \right),\ \boldsymbol{\alpha}_{3} \notin \boldsymbol{G}_{3}, \right. \\
\nonumber & \quad \boldsymbol{\alpha}_{3} \in \boldsymbol{U}^{\prime}_3,\ (\alpha_1, \alpha_2, \alpha_3, \alpha_3) \in \boldsymbol{U}^{\prime}_4, \\
\nonumber & \quad \kappa^{\prime} \leqslant \alpha_4 < \min\left(\alpha_3, \frac{1}{2}(1-\alpha_1 -\alpha_2-\alpha_3) \right),\ \boldsymbol{\alpha}_{4} \notin \boldsymbol{G}_{4}, \\
\nonumber & \quad \kappa^{\prime} \leqslant \alpha_5 < \min\left(\alpha_4, \frac{1}{2}(1-\alpha_1 -\alpha_2-\alpha_3-\alpha_4) \right),\ \boldsymbol{\alpha}_{5} \notin \boldsymbol{G}_{5}, \\
\nonumber & \quad \text{either } \boldsymbol{\alpha}_{5} \notin \boldsymbol{U}^{\prime}_5 \text{ or } (\alpha_1, \alpha_2, \alpha_3, \alpha_4, \alpha_5, \alpha_5) \notin \boldsymbol{U}^{\prime}_6, \\
\nonumber & \quad \alpha_5 < \alpha_6 < \frac{1}{2}(1-\alpha_1-\alpha_2-\alpha_3-\alpha_4-\alpha_5),\ \boldsymbol{\alpha}_{6} \in \boldsymbol{G}_{6}, \\
\nonumber & \left. \quad \kappa \leqslant \alpha_1 < \frac{1}{2},\ \kappa^{\prime} \leqslant \alpha_2 < \min\left(\alpha_1, \frac{1}{2}(1-\alpha_1) \right) \right\}, \\
\nonumber L_{914}(\boldsymbol{\alpha}_{7}) :=&\ \left\{\tau^{\prime} \leqslant \alpha_1 < \tau,\ \kappa^{\prime} \leqslant \alpha_3 < \min\left(\alpha_2, \frac{1}{2}(1-\alpha_1 -\alpha_2) \right),\ \boldsymbol{\alpha}_{3} \notin \boldsymbol{G}_{3}, \right. \\
\nonumber & \quad \boldsymbol{\alpha}_{3} \in \boldsymbol{U}^{\prime}_3,\ (\alpha_1, \alpha_2, \alpha_3, \alpha_3) \in \boldsymbol{U}^{\prime}_4, \\
\nonumber & \quad \kappa^{\prime} \leqslant \alpha_4 < \min\left(\alpha_3, \frac{1}{2}(1-\alpha_1 -\alpha_2-\alpha_3) \right),\ \boldsymbol{\alpha}_{4} \notin \boldsymbol{G}_{4}, \\
\nonumber & \quad \kappa^{\prime} \leqslant \alpha_5 < \min\left(\alpha_4, \frac{1}{2}(1-\alpha_1 -\alpha_2-\alpha_3-\alpha_4) \right),\ \boldsymbol{\alpha}_{5} \notin \boldsymbol{G}_{5}, \\
\nonumber & \quad \text{either } \boldsymbol{\alpha}_{5} \notin \boldsymbol{U}^{\prime}_5 \text{ or } (\alpha_1, \alpha_2, \alpha_3, \alpha_4, \alpha_5, \alpha_5) \notin \boldsymbol{U}^{\prime}_6, \\
\nonumber & \quad \alpha_5 < \alpha_6 < \frac{1}{2}(1-\alpha_1-\alpha_2-\alpha_3-\alpha_4-\alpha_5),\ \boldsymbol{\alpha}_{6} \notin \boldsymbol{G}_{6}, \\
\nonumber & \quad \alpha_6 < \alpha_7 < \frac{1}{2}(1-\alpha_1-\alpha_2-\alpha_3-\alpha_4-\alpha_5-\alpha_6),\ \boldsymbol{\alpha}_{7} \in \boldsymbol{G}_{7}, \\
\nonumber & \left. \quad \kappa \leqslant \alpha_1 < \frac{1}{2},\ \kappa^{\prime} \leqslant \alpha_2 < \min\left(\alpha_1, \frac{1}{2}(1-\alpha_1) \right) \right\}, \\
\nonumber L_{915}(\boldsymbol{\alpha}_{7}) :=&\ \left\{\tau^{\prime} \leqslant \alpha_1 < \tau,\ \kappa^{\prime} \leqslant \alpha_3 < \min\left(\alpha_2, \frac{1}{2}(1-\alpha_1 -\alpha_2) \right),\ \boldsymbol{\alpha}_{3} \notin \boldsymbol{G}_{3}, \right. \\
\nonumber & \quad \boldsymbol{\alpha}_{3} \in \boldsymbol{U}^{\prime}_3,\ (\alpha_1, \alpha_2, \alpha_3, \alpha_3) \in \boldsymbol{U}^{\prime}_4, \\
\nonumber & \quad \kappa^{\prime} \leqslant \alpha_4 < \min\left(\alpha_3, \frac{1}{2}(1-\alpha_1 -\alpha_2-\alpha_3) \right),\ \boldsymbol{\alpha}_{4} \notin \boldsymbol{G}_{4}, \\
\nonumber & \quad \kappa^{\prime} \leqslant \alpha_5 < \min\left(\alpha_4, \frac{1}{2}(1-\alpha_1 -\alpha_2-\alpha_3-\alpha_4) \right),\ \boldsymbol{\alpha}_{5} \notin \boldsymbol{G}_{5}, \\
\nonumber & \quad \boldsymbol{\alpha}_{5} \in \boldsymbol{U}^{\prime}_5,\ (\alpha_1, \alpha_2, \alpha_3, \alpha_4, \alpha_5, \alpha_5) \in \boldsymbol{U}^{\prime}_6, \\
\nonumber & \quad \kappa^{\prime} \leqslant \alpha_6 < \min\left(\alpha_5, \frac{1}{2}(1-\alpha_1 -\alpha_2-\alpha_3-\alpha_4-\alpha_5) \right),\ \boldsymbol{\alpha}_{6} \notin \boldsymbol{G}_{6}, \\
\nonumber & \quad \kappa^{\prime} \leqslant \alpha_7 < \min\left(\alpha_6, \frac{1}{2}(1-\alpha_1 -\alpha_2-\alpha_3-\alpha_4-\alpha_5-\alpha_6) \right),\ \boldsymbol{\alpha}_{7} \notin \boldsymbol{G}_{7}, \\
\nonumber & \left. \quad \kappa \leqslant \alpha_1 < \frac{1}{2},\ \kappa^{\prime} \leqslant \alpha_2 < \min\left(\alpha_1, \frac{1}{2}(1-\alpha_1) \right) \right\}.
\end{align}

Finally, for $\frac{7}{13} \leqslant \theta < \frac{6}{11}$ we have
\begin{align}
\nonumber C(\theta) \leqslant C_{9}(\theta) :=&\ 1 + \left(E_{901}(\theta) + E_{902}(\theta) - E_{903}(\theta) - E_{904}(\theta) + E_{905}(\theta) - E_{906}(\theta) - E_{907}(\theta) + E_{908}(\theta) \right. \\
& \left. \qquad + E_{909}(\theta) - E_{910}(\theta) - E_{911}(\theta) + E_{912}(\theta) - E_{913}(\theta) - E_{914}(\theta) + E_{915}(\theta) + E_{916}(\theta) \right).
\end{align}

\begin{center}
\begin{tabular}{|c|c|c|c|c|c|c|c|}
\hline \boldmath{$\theta$} & \boldmath{$C_{9}(\theta)$} & \boldmath{$\theta$} & \boldmath{$C_{9}(\theta)$} & \boldmath{$\theta$} & \boldmath{$C_{9}(\theta)$} & \boldmath{$\theta$} & \boldmath{$C_{9}(\theta)$} \\
\hline $0.5385$ & $1.95009$ & $0.5402$ & $2.00356$ & $0.5422$ & $2.08389$ & $0.544$ & $2.16035$ \\
\hline $0.5388$ & $1.95279$ & $0.5405$ & $2.023$ & $0.5425$ & $2.10289$ & $0.5442$ & $2.16134$ \\
\hline $0.539$ & $1.95677$ & $0.5408$ & $2.03985$ & $0.5428$ & $2.10903$ & $0.5445$ & $2.18608$ \\
\hline $0.5392$ & $1.95916$ & $0.541$ & $2.05815$ & $0.543$ & $2.12223$ & $0.5448$ & $2.20478$ \\
\hline $0.5395$ & $1.96639$ & $0.5412$ & $2.06933$ & $0.5432$ & $2.12238$ & $0.545$ & $2.2179$ \\
\hline $0.5398$ & $1.986$ & $0.5415$ & $2.0785$ & $0.5435$ & $2.13851$ & $0.5452$ & $2.22554$ \\
\hline $0.54$ & $1.99921$ & $0.542$ & $2.07929$ & $0.5438$ & $2.15308$ & $0.5454$ & $2.22589$ \\
\hline
\end{tabular}
\end{center}

\subsection{Case 10. $\frac{6}{11} \leqslant \theta \leqslant 0.679$}
In the last case, we cannot ensure that $(\alpha_1, \alpha_2, 2 \theta - 1) \in \boldsymbol{S}_{3}$ since $\theta \geqslant \frac{6}{11}$. Using Buchstab's identity, we have
\begin{align}
\nonumber S\left(\mathcal{A}, (2x)^{\frac{1}{2}} \right) =&\ S\left(\mathcal{A}, x^{\kappa} \right) - \sum_{\kappa \leqslant \alpha_1 < \frac{1}{2}} S\left(\mathcal{A}_{p_1}, p_1 \right) \\
\nonumber =&\ S\left(\mathcal{A}, x^{\kappa} \right) - \sum_{\kappa \leqslant \alpha_1 < \tau} S\left(\mathcal{A}_{p_1}, p_1 \right) - \sum_{\tau \leqslant \alpha_1 \leqslant \frac{3}{7}} S\left(\mathcal{A}_{p_1}, p_1 \right) \\
\nonumber & - \sum_{\frac{3}{7} < \alpha_1 < 1 - \theta} S\left(\mathcal{A}_{p_1}, p_1 \right) - \sum_{1 - \theta \leqslant \alpha_1 < \frac{1}{2}} S\left(\mathcal{A}_{p_1}, p_1 \right) \\
=&\ S_{X1} - S_{X2} - S_{X3} - S_{X4} - S_{X5}.
\end{align}
The next decomposing process is quite similar to the process in \textbf{Case 9}. The only difference is that we do not need to consider
$$
E_{909}(\theta) - E_{910}(\theta) - E_{911}(\theta) + E_{912}(\theta) - E_{913}(\theta) - E_{914}(\theta) + E_{915}(\theta)
$$
since $\tau = \tau^{\prime} = \frac{5-6\theta}{7}$ now. The total loss when $\frac{6}{11} \leqslant \theta < \frac{4}{7}$ is no more than
\begin{align}
\nonumber & \left(\int_{1-\theta}^{\frac{1}{2}} \frac{\omega\left(\frac{1-t_1}{t_1} \right)}{t_1^2} d t_1 \right) \\
\nonumber + & \left(\int_{(t_1, t_2, t_3) \in L_{X02}} \frac{\omega\left(\frac{1-t_1-t_2-t_3}{t_3} \right)}{t_1 t_2 t_3^2} d t_3 d t_2 d t_1 \right) \\
\nonumber - & \left(\int_{(t_1, t_2, t_3, t_4) \in L_{X03}} \frac{\omega\left(\frac{1-t_1-t_2-t_3-t_4}{t_4} \right)}{t_1 t_2 t_3 t_4^2} d t_4 d t_3 d t_2 d t_1 \right) \\
\nonumber - & \left(\int_{(t_1, t_2, t_3, t_4, t_5) \in L_{X04}} \frac{\omega\left(\frac{1-t_1-t_2-t_3-t_4-t_5}{t_5} \right)}{t_1 t_2 t_3 t_4 t_5^2} d t_5 d t_4 d t_3 d t_2 d t_1 \right) \\
\nonumber + & \left(\int_{(t_1, t_2, t_3, t_4, t_5) \in L_{X05}} \frac{\omega\left(\frac{1-t_1-t_2-t_3-t_4-t_5}{t_5} \right)}{t_1 t_2 t_3 t_4 t_5^2} d t_5 d t_4 d t_3 d t_2 d t_1 \right) \\
\nonumber - & \left(\int_{(t_1, t_2, t_3, t_4, t_5, t_6) \in L_{X06}} \frac{\omega\left(\frac{1-t_1-t_2-t_3-t_4-t_5-t_6}{t_6} \right)}{t_1 t_2 t_3 t_4 t_5 t_6^2} d t_6 d t_5 d t_4 d t_3 d t_2 d t_1 \right) \\
\nonumber - & \left(\int_{(t_1, t_2, t_3, t_4, t_5, t_6, t_7) \in L_{X07}} \frac{\omega\left(\frac{1-t_1-t_2-t_3-t_4-t_5-t_6-t_7}{t_7} \right)}{t_1 t_2 t_3 t_4 t_5 t_6 t_7^2} d t_7 d t_6 d t_5 d t_4 d t_3 d t_2 d t_1 \right) \\
\nonumber + & \left(\int_{(t_1, t_2, t_3, t_4, t_5, t_6, t_7) \in L_{X08}} \frac{\omega\left(\frac{1-t_1-t_2-t_3-t_4-t_5-t_6-t_7}{t_7} \right)}{t_1 t_2 t_3 t_4 t_5 t_6 t_7^2} d t_7 d t_6 d t_5 d t_4 d t_3 d t_2 d t_1 \right) \\
\nonumber + & \left(\int_{\tau}^{\frac{3}{7}} \frac{\omega\left(\frac{1-t_1}{t_1} \right)}{t_1^2} d t_1 \right) \\
=&\ E_{X01}(\theta) + E_{X02}(\theta) - E_{X03}(\theta) - E_{X04}(\theta) + E_{X05}(\theta) - E_{X06}(\theta) - E_{X07}(\theta) + E_{X08}(\theta) + E_{X09}(\theta),
\end{align}
where
\begin{align}
\nonumber L_{X02}(\boldsymbol{\alpha}_{3}) :=&\ \left\{ \kappa \leqslant \alpha_1 < \tau,\ \kappa \leqslant \alpha_3 < \min\left(\alpha_2, \frac{1}{2}(1-\alpha_1 -\alpha_2) \right),\ \boldsymbol{\alpha}_{3} \notin \boldsymbol{G}_{3}, \right. \\
\nonumber & \quad \text{either } \boldsymbol{\alpha}_{3} \notin \boldsymbol{U}_3 \text{ or } (\alpha_1, \alpha_2, \alpha_3, \alpha_3) \notin \boldsymbol{U}_4, \\
\nonumber & \left. \quad \kappa \leqslant \alpha_1 < \frac{1}{2},\ \kappa \leqslant \alpha_2 < \min\left(\alpha_1, \frac{1}{2}(1-\alpha_1) \right) \right\}, \\
\nonumber L_{X03}(\boldsymbol{\alpha}_{4}) :=&\ \left\{\kappa \leqslant \alpha_1 < \tau,\ \kappa \leqslant \alpha_3 < \min\left(\alpha_2, \frac{1}{2}(1-\alpha_1 -\alpha_2) \right),\ \boldsymbol{\alpha}_{3} \notin \boldsymbol{G}_{3}, \right. \\
\nonumber & \quad \text{either } \boldsymbol{\alpha}_{3} \notin \boldsymbol{U}_3 \text{ or } (\alpha_1, \alpha_2, \alpha_3, \alpha_3) \notin \boldsymbol{U}_4, \\
\nonumber & \quad \alpha_3 < \alpha_4 < \frac{1}{2}(1-\alpha_1-\alpha_2-\alpha_3),\ \boldsymbol{\alpha}_{4} \in \boldsymbol{G}_{4}, \\
\nonumber & \left. \quad \kappa \leqslant \alpha_1 < \frac{1}{2},\ \kappa \leqslant \alpha_2 < \min\left(\alpha_1, \frac{1}{2}(1-\alpha_1) \right) \right\}, \\
\nonumber L_{X04}(\boldsymbol{\alpha}_{5}) :=&\ \left\{\kappa \leqslant \alpha_1 < \tau,\ \kappa \leqslant \alpha_3 < \min\left(\alpha_2, \frac{1}{2}(1-\alpha_1 -\alpha_2) \right),\ \boldsymbol{\alpha}_{3} \notin \boldsymbol{G}_{3}, \right. \\
\nonumber & \quad \text{either } \boldsymbol{\alpha}_{3} \notin \boldsymbol{U}_3 \text{ or } (\alpha_1, \alpha_2, \alpha_3, \alpha_3) \notin \boldsymbol{U}_4, \\
\nonumber & \quad \alpha_3 < \alpha_4 < \frac{1}{2}(1-\alpha_1-\alpha_2-\alpha_3),\ \boldsymbol{\alpha}_{4} \notin \boldsymbol{G}_{4}, \\
\nonumber & \quad \alpha_4 < \alpha_5 < \frac{1}{2}(1-\alpha_1-\alpha_2-\alpha_3-\alpha_4),\ \boldsymbol{\alpha}_{5} \in \boldsymbol{G}_{5}, \\
\nonumber & \left. \quad \kappa \leqslant \alpha_1 < \frac{1}{2},\ \kappa \leqslant \alpha_2 < \min\left(\alpha_1, \frac{1}{2}(1-\alpha_1) \right) \right\}, \\
\nonumber L_{X05}(\boldsymbol{\alpha}_{5}) :=&\ \left\{\kappa \leqslant \alpha_1 < \tau,\ \kappa \leqslant \alpha_3 < \min\left(\alpha_2, \frac{1}{2}(1-\alpha_1 -\alpha_2) \right),\ \boldsymbol{\alpha}_{3} \notin \boldsymbol{G}_{3}, \right. \\
\nonumber & \quad \boldsymbol{\alpha}_{3} \in \boldsymbol{U}_3,\ (\alpha_1, \alpha_2, \alpha_3, \alpha_3) \in \boldsymbol{U}_4, \\
\nonumber & \quad \kappa \leqslant \alpha_4 < \min\left(\alpha_3, \frac{1}{2}(1-\alpha_1 -\alpha_2-\alpha_3) \right),\ \boldsymbol{\alpha}_{4} \notin \boldsymbol{G}_{4}, \\
\nonumber & \quad \kappa \leqslant \alpha_5 < \min\left(\alpha_4, \frac{1}{2}(1-\alpha_1 -\alpha_2-\alpha_3-\alpha_4) \right),\ \boldsymbol{\alpha}_{5} \notin \boldsymbol{G}_{5}, \\
\nonumber & \quad \text{either } \boldsymbol{\alpha}_{5} \notin \boldsymbol{U}_5 \text{ or } (\alpha_1, \alpha_2, \alpha_3, \alpha_4, \alpha_5, \alpha_5) \notin \boldsymbol{U}_6, \\
\nonumber & \left. \quad \kappa \leqslant \alpha_1 < \frac{1}{2},\ \kappa \leqslant \alpha_2 < \min\left(\alpha_1, \frac{1}{2}(1-\alpha_1) \right) \right\}, \\
\nonumber L_{X06}(\boldsymbol{\alpha}_{6}) :=&\ \left\{\kappa \leqslant \alpha_1 < \tau,\ \kappa \leqslant \alpha_3 < \min\left(\alpha_2, \frac{1}{2}(1-\alpha_1 -\alpha_2) \right),\ \boldsymbol{\alpha}_{3} \notin \boldsymbol{G}_{3}, \right. \\
\nonumber & \quad \boldsymbol{\alpha}_{3} \in \boldsymbol{U}_3,\ (\alpha_1, \alpha_2, \alpha_3, \alpha_3) \in \boldsymbol{U}_4, \\
\nonumber & \quad \kappa \leqslant \alpha_4 < \min\left(\alpha_3, \frac{1}{2}(1-\alpha_1 -\alpha_2-\alpha_3) \right),\ \boldsymbol{\alpha}_{4} \notin \boldsymbol{G}_{4}, \\
\nonumber & \quad \kappa \leqslant \alpha_5 < \min\left(\alpha_4, \frac{1}{2}(1-\alpha_1 -\alpha_2-\alpha_3-\alpha_4) \right),\ \boldsymbol{\alpha}_{5} \notin \boldsymbol{G}_{5}, \\
\nonumber & \quad \text{either } \boldsymbol{\alpha}_{5} \notin \boldsymbol{U}_5 \text{ or } (\alpha_1, \alpha_2, \alpha_3, \alpha_4, \alpha_5, \alpha_5) \notin \boldsymbol{U}_6, \\
\nonumber & \quad \alpha_5 < \alpha_6 < \frac{1}{2}(1-\alpha_1-\alpha_2-\alpha_3-\alpha_4-\alpha_5),\ \boldsymbol{\alpha}_{6} \in \boldsymbol{G}_{6}, \\
\nonumber & \left. \quad \kappa \leqslant \alpha_1 < \frac{1}{2},\ \kappa \leqslant \alpha_2 < \min\left(\alpha_1, \frac{1}{2}(1-\alpha_1) \right) \right\}, \\
\nonumber L_{X07}(\boldsymbol{\alpha}_{7}) :=&\ \left\{\kappa \leqslant \alpha_1 < \tau,\ \kappa \leqslant \alpha_3 < \min\left(\alpha_2, \frac{1}{2}(1-\alpha_1 -\alpha_2) \right),\ \boldsymbol{\alpha}_{3} \notin \boldsymbol{G}_{3}, \right. \\
\nonumber & \quad \boldsymbol{\alpha}_{3} \in \boldsymbol{U}_3,\ (\alpha_1, \alpha_2, \alpha_3, \alpha_3) \in \boldsymbol{U}_4, \\
\nonumber & \quad \kappa \leqslant \alpha_4 < \min\left(\alpha_3, \frac{1}{2}(1-\alpha_1 -\alpha_2-\alpha_3) \right),\ \boldsymbol{\alpha}_{4} \notin \boldsymbol{G}_{4}, \\
\nonumber & \quad \kappa \leqslant \alpha_5 < \min\left(\alpha_4, \frac{1}{2}(1-\alpha_1 -\alpha_2-\alpha_3-\alpha_4) \right),\ \boldsymbol{\alpha}_{5} \notin \boldsymbol{G}_{5}, \\
\nonumber & \quad \text{either } \boldsymbol{\alpha}_{5} \notin \boldsymbol{U}_5 \text{ or } (\alpha_1, \alpha_2, \alpha_3, \alpha_4, \alpha_5, \alpha_5) \notin \boldsymbol{U}_6, \\
\nonumber & \quad \alpha_5 < \alpha_6 < \frac{1}{2}(1-\alpha_1-\alpha_2-\alpha_3-\alpha_4-\alpha_5),\ \boldsymbol{\alpha}_{6} \notin \boldsymbol{G}_{6}, \\
\nonumber & \quad \alpha_6 < \alpha_7 < \frac{1}{2}(1-\alpha_1-\alpha_2-\alpha_3-\alpha_4-\alpha_5-\alpha_6),\ \boldsymbol{\alpha}_{7} \in \boldsymbol{G}_{7}, \\
\nonumber & \left. \quad \kappa \leqslant \alpha_1 < \frac{1}{2},\ \kappa \leqslant \alpha_2 < \min\left(\alpha_1, \frac{1}{2}(1-\alpha_1) \right) \right\}, \\
\nonumber L_{X08}(\boldsymbol{\alpha}_{7}) :=&\ \left\{\kappa \leqslant \alpha_1 < \tau,\ \kappa \leqslant \alpha_3 < \min\left(\alpha_2, \frac{1}{2}(1-\alpha_1 -\alpha_2) \right),\ \boldsymbol{\alpha}_{3} \notin \boldsymbol{G}_{3}, \right. \\
\nonumber & \quad \boldsymbol{\alpha}_{3} \in \boldsymbol{U}_3,\ (\alpha_1, \alpha_2, \alpha_3, \alpha_3) \in \boldsymbol{U}_4, \\
\nonumber & \quad \kappa \leqslant \alpha_4 < \min\left(\alpha_3, \frac{1}{2}(1-\alpha_1 -\alpha_2-\alpha_3) \right),\ \boldsymbol{\alpha}_{4} \notin \boldsymbol{G}_{4}, \\
\nonumber & \quad \kappa \leqslant \alpha_5 < \min\left(\alpha_4, \frac{1}{2}(1-\alpha_1 -\alpha_2-\alpha_3-\alpha_4) \right),\ \boldsymbol{\alpha}_{5} \notin \boldsymbol{G}_{5}, \\
\nonumber & \quad \boldsymbol{\alpha}_{5} \in \boldsymbol{U}_5,\ (\alpha_1, \alpha_2, \alpha_3, \alpha_4, \alpha_5, \alpha_5) \in \boldsymbol{U}_6, \\
\nonumber & \quad \kappa \leqslant \alpha_6 < \min\left(\alpha_5, \frac{1}{2}(1-\alpha_1 -\alpha_2-\alpha_3-\alpha_4-\alpha_5) \right),\ \boldsymbol{\alpha}_{6} \notin \boldsymbol{G}_{6}, \\
\nonumber & \quad \kappa \leqslant \alpha_7 < \min\left(\alpha_6, \frac{1}{2}(1-\alpha_1 -\alpha_2-\alpha_3-\alpha_4-\alpha_5-\alpha_6) \right),\ \boldsymbol{\alpha}_{7} \notin \boldsymbol{G}_{7}, \\
\nonumber & \left. \quad \kappa \leqslant \alpha_1 < \frac{1}{2},\ \kappa \leqslant \alpha_2 < \min\left(\alpha_1, \frac{1}{2}(1-\alpha_1) \right) \right\}.
\end{align}

For $\frac{6}{11} \leqslant \theta < \frac{4}{7}$ we have
\begin{align}
\nonumber C(\theta) \leqslant C_{X}(\theta) :=&\ 1 + \left(E_{X01}(\theta) + E_{X02}(\theta) - E_{X03}(\theta) - E_{X04}(\theta) \right. \\
& \left. \qquad + E_{X05}(\theta) - E_{X06}(\theta) - E_{X07}(\theta) + E_{X08}(\theta) + E_{X09}(\theta) \right).
\end{align}

\begin{center}
\begin{tabular}{|c|c|c|c|c|c|}
\hline \boldmath{$\theta$} & \boldmath{$C_{X}(\theta)$} & \boldmath{$\theta$} & \boldmath{$C_{X}(\theta)$} & \boldmath{$\theta$} & \boldmath{$C_{X}(\theta)$} \\
\hline $0.546$ & $2.33807$ & $0.553$ & $2.44031$ & $0.561$ & $2.68948$ \\
\hline $0.547$ & $2.35981$ & $0.555$ & $2.45396$ & $0.562$ & $2.71434$ \\
\hline $0.549$ & $2.3809$ & $0.556$ & $2.50757$ & $0.563$ & $2.80532$ \\
\hline $0.55$ & $2.39955$ & $0.558$ & $2.51193$ & $0.564$ & $2.87871$ \\
\hline $0.551$ & $2.43178$ & $0.56$ & $2.53449$ & $0.565$ & $2.95471$ \\
\hline
\end{tabular}
\end{center}

However, the alternative sieve approach may not be the best way to bound $C(\theta)$ when $\theta$ increases. Because of the rapid increase of the size of loss, using the Rosser--Iwaniec sieve is better than using Harman's sieve when $\theta$ is large. Building on the previous work of Fouvry \cite{Fouvry}, Baker and Harman \cite{676} showed that
\begin{equation}
C(\theta) \leqslant \frac{14}{12 - 13\theta} - \log\left(\frac{4(1-\theta)}{3\theta} \right)
\end{equation}
for $\frac{17}{32} \leqslant \theta < \frac{4}{7}$. After calculating, we find that the turning point is between $0.565$ and $0.566$. That is, we use (54) when $\theta < 0.565$ and (55) when $\theta \geqslant 0.565$. For $\theta \geqslant \frac{4}{7}$, Baker and Harman \cite{676} also got
\begin{align}
C(\theta) \leqslant
\begin{cases}
\frac{14}{12 - 13\theta}, & \frac{4}{7} \leqslant \theta < 0.6, \\
\frac{8}{3 - \theta}, & 0.6 \leqslant \theta \leqslant 0.679.
\end{cases}
\end{align}

\subsection{Proof of Theorem 1}
Note that $C(\theta)$ is monotonic increasing, we can use the various upper bounds for $C(\theta)$ in above 10 cases to show that
\begin{equation}
\int_{0.5}^{0.679} C(\theta) < 0.4993,
\end{equation}
and Theorem~\ref{t1} is now proved. The upper constants for the Brun--Titchmarsh theorem with almost all moduli $q \leqslant x^{\theta}$ are given in the above 10 tables.

\section{Applications}
Clearly our Theorem~\ref{t1} has many interesting applications (just like the previous Baker--Harman's result), and we state some of them in this section. Let $g > 1$ denote an integer and $f(X) \in \mathbb{Z}[X]$ denote a polynomial of positive degree that has a positive leading coefficient and no multiple roots in its splitting field. Put $u(n) = f\left(g^n\right)$. The following two applications focus on the quadratic fields $\mathbb{Q}\left(\sqrt{u(n)}\right)$. Note that this type of quadratic fields was first considered by Shanks \cite{Shanks} with the special case $u(n) = \left(2^n + 3\right)^2 - 8$. Let $Q_u(M, N; s)$ denote the number of integers $n \in [M + 1, M + N]$ for which $\mathbb{Q}\left(\sqrt{u(n)}\right) = \mathbb{Q}\left(\sqrt{s}\right)$, where $M, N, s \in \mathbb{Z}$, $N \geqslant 1$ and $s$ square-free. For the case $\text{deg } f \geqslant 3$, Luca and Shparlinski [\cite{LucaShparlinski2009}, Theorem 1.3] proved in 2009 that
\begin{equation}
Q_u(M, N; s) \ll N^{\frac{3}{2(1+\boldsymbol{\Theta})}} (\log N)^{\frac{4 + \boldsymbol{\Theta}}{1 + \boldsymbol{\Theta}}}.
\end{equation}
Clearly we have $Q_u(M, N; s) \ll N^{\frac{500}{559} + \varepsilon}$ with $\boldsymbol{\Theta} = 0.677$ given by Baker and Harman \cite{677} and $Q_u(M, N; s) \ll N^{\frac{3}{4} + \varepsilon}$ with conjectured value $\boldsymbol{\Theta} = 1$. Note that $\frac{500}{559} \approx 0.895$ and $\frac{3}{4} = 0.75$. In 2017, Banks and Shparlinski [\cite{BanksShparlinski2017}, Theorem 1.1] improved (58) to
\begin{equation}
Q_u(M, N; s) \ll N^{\frac{1}{2 \boldsymbol{\Theta} }} (\log N)^{2 - \frac{1}{\boldsymbol{\Theta}}}.
\end{equation}
Now we have $Q_u(M, N; s) \ll N^{\frac{500}{677} + \varepsilon}$ with $\boldsymbol{\Theta} = 0.677$ given by Baker and Harman \cite{677} and $Q_u(M, N; s) \ll N^{\frac{1}{2} + \varepsilon}$ with conjectured value $\boldsymbol{\Theta} = 1$. Note that $\frac{500}{677} \approx 0.739 < 0.75$, the unconditional result given by (59) is even better than the conditional result given by (58). As our first application, we apply our Theorem~\ref{t1} ($\boldsymbol{\Theta} = 0.679$) directly to (59) to give that
\begin{theorem1number}\label{t2}
$$
Q_u(M, N; s) \ll N^{\frac{500}{679}} (\log N)^{\frac{358}{679}}, \quad \text{where} \quad \frac{500}{679} \approx 0.737.
$$
\end{theorem1number}
\noindent We can also obtain the corresponding improved lower bound for $R_u(M, N)$, the number of distinct fields among $\mathbb{Q}\left(\sqrt{u(n)}\right)$ as $n$ varies over $M+1, \ldots, M+N$.

In the same paper, Banks and Shparlinski \cite{BanksShparlinski2017} also considered the bound for $Q_u(M, N; s)$ on average over square-free $s$. Their [\cite{BanksShparlinski2017}, Corollary 1.4] gives
\begin{align}
\sum_{s \leqslant S} Q_u(M, N; s) \leqslant
\begin{cases}
S^{1-\frac{1}{4 \boldsymbol{\Theta}}} N^{\frac{1}{2 \boldsymbol{\Theta}} + \varepsilon}, & S \leqslant N^{\frac{2(1-\boldsymbol{\Theta})}{1+3\boldsymbol{\Theta}}}, \\
S^{\frac{1}{2}} N^{\frac{3-\boldsymbol{\Theta}}{1+3\boldsymbol{\Theta}} + \varepsilon}, & N^{\frac{2(1-\boldsymbol{\Theta})}{1+3\boldsymbol{\Theta}}} < S \leqslant N^{\frac{4(1-\boldsymbol{\Theta})}{1+3\boldsymbol{\Theta}}}, \\
S^{\frac{3}{3+\boldsymbol{\Theta}}} N^{\frac{3-\boldsymbol{\Theta}}{3+\boldsymbol{\Theta}} + \varepsilon}, & N^{\frac{4(1-\boldsymbol{\Theta})}{1+3\boldsymbol{\Theta}}} < S \leqslant N^{\frac{2\boldsymbol{\Theta}}{3}}, \\
N ~(\text{trivial bound}), & N^{\frac{2\boldsymbol{\Theta}}{3}} < S.
\end{cases}
\end{align}
As our second application, using our Theorem~\ref{t1} ($\boldsymbol{\Theta} = 0.679$) we have
\begin{theorem1number}\label{t3}
\begin{align*}
\sum_{s \leqslant S} Q_u(M, N; s) \leqslant
\begin{cases}
S^{\frac{429}{679}} N^{\frac{500}{679} + \varepsilon}, & S \leqslant N^{\frac{642}{3037}}, \\
S^{\frac{1}{2}} N^{\frac{2321}{3037} + \varepsilon}, & N^{\frac{642}{3037}} < S \leqslant N^{\frac{1284}{3037}}, \\
S^{\frac{3000}{3679}} N^{\frac{2321}{3679} + \varepsilon}, & N^{\frac{1284}{3037}} < S \leqslant N^{\frac{679}{1500}}, \\
N, & N^{\frac{679}{1500}} < S.
\end{cases}
\end{align*}
\end{theorem1number}

\section*{Appendix: Shifted primes without large prime factors}
Since the shifting process destroys the multiplicative structure of prime numbers, it is conjectured that $P^{+}(p + a)$ has a distribution roughly the same as the distribution of $P^{+}(n)$ over all integers $n \leqslant x$. In fact, one has a special version of this general conjecture:
\begin{conjecture*}
There are infinitely many primes $p$ with $P^{+}(p + a) < p^{\boldsymbol{\vartheta}}$ and infinitely many primes $p$ with $P^{+}(p + a) > p^{\boldsymbol{\Theta}}$ for any $0 < \boldsymbol{\vartheta}, \boldsymbol{\Theta} < 1$.
\end{conjecture*}
Note that we have already mentioned the historical progress towards the second assertion in the Introduction. For the first assertion, Erd\H{o}s \cite{Erdos1935} first proved in 1935 that there exists some $\boldsymbol{\vartheta} < 1$ such that $P^{+}(p + a) < p^{\boldsymbol{\vartheta}}$ for infinitely many $p$. After that, many mathematicians improved the upper bound for $\boldsymbol{\vartheta}$. The progression of records for $\boldsymbol{\vartheta}$ can be seen in the following table.

\begin{center}
\begin{tabular}{|c|c|c|c|}
\hline \boldmath{$\boldsymbol{\vartheta}$} & \boldmath{$\approx$} & \textbf{Author} & \textbf{Year} \\
\hline $2 \sqrt{2} - 2$ & $0.8285$ & Wooldridge \cite{Wooldridge1979} & 1979 \\
\hline $\frac{625}{512 e}$ & $0.4491$ & Pomerance \cite{Pomerance1980} & 1980 \\
\hline $\frac{1}{2} \exp \left(- \frac{1}{2.73}\right)$ & $0.3465$ & Balog \cite{Balog1984} & 1984 \\
\hline $\frac{1}{2} \exp \left(- \frac{1}{2}\right)$ & $0.3033$ & Friedlander \cite{Friedlander1989} & 1989 \\
\hline $0.2961$ & $0.2961$ & Baker and Harman \cite{677} & 1998 \\
\hline $\frac{15}{32} \exp \left(- \frac{1}{2}\right)$ & $0.2844$ & Lichtman \cite{Lichtman2} & 2022 \\
\hline
\end{tabular}
\end{center}

However, we cannot improve Lichtman's result by applying any sieves since we do not have a distribution result for primes in arithmetic progressions with multi-linear forms of moduli larger than $x^{\frac{17}{32}}$.

\section*{Acknowledgements}
The author would like to thank Igor Shparlinski for providing information about the papers \cite{LucaShparlinski2009} and \cite{BanksShparlinski2017}.

\bibliographystyle{plain}
\bibliography{bib}
\end{document}